\documentclass[11pt]{amsart}
\usepackage{geometry}    
\usepackage{graphicx}                    
\usepackage{amsmath}                     
\usepackage{amsfonts}                    
\usepackage{amssymb}                     
\usepackage{amsthm}
\usepackage{xcolor} 
\usepackage{comment}
\usepackage{url}
\usepackage{subfigure}
\usepackage{float}
\usepackage{lipsum}
\usepackage{marginnote}

\newcommand{\N}{{\mathbb N}}

\newcommand{\R}{{\mathbb R}}

\newcommand{\bfu}{{\bf u}}

\newtheorem{definition}{Definition}[section]
\newtheorem{lemma}{Lemma}[section]

\newtheorem{remark}{Remark}[section]

\DeclareMathOperator{\Id}{Id}
\DeclareMathOperator{\Df}{Df}
\DeclareMathOperator{\DF}{DF}

\title{Dynamics of perturbed elliptical billiard tables}

\author{Patrick R. Bishop}

\author{Summer Chenoweth}
\thanks{SC was partially supported by the GMU Math Industrial Immersion Program.}

\author{Emmanuel Fleurantin}
\address{Department of Mathematics, Morehouse College, Atlanta, GA 30314}
\thanks{EF was partially supported by the National Science Foundation under grant DMS-2137947.}

\author{Evelyn Sander}
\address{Department of Mathematical Sciences, George Mason University, Fairfax, VA 22030}
\thanks{ES was partially supported by the Simons Foundation under Award 636383.}

\author{Jason Mireles James}
\address{Department of Mathematics and Statistics, Florida Atlantic University, Boca Raton, FL 33431}
\thanks{JMJ was partially supported by the National Science Foundation under grant DMS-2307987.}

\usepackage[colorlinks=true]{hyperref}
\newcommand{\correction}[2]{#2}

\begin{document}
\maketitle
\markboth{BISHOP ET AL.}{PERTURBED ELLIPTICAL BILLIARD TABLES}

\vspace{-2em}

\begin{abstract}
%
Dynamical billiards consist of a particle on a two-dimensional table, 
bouncing elastically off the boundary. The state of the system is 
given by two numbers: one describing the 
location along the curve where the bounce occurs, and another describing 
the incoming angle of the trajectory before the bounce. 
Tracking these numbers over successive bounces defines a two-dimensional 
area preserving map, and iterating this map gives a dynamical system first 
studied by Birkhoff. 
Although there are powerful theoretical results showing that generic 
(strictly convex) billiards exhibit chaotic dynamics, 
it is nevertheless difficult 
(if not impossible) to decide when a given billiard table is generic and  
one will often resort to numerical computations.  In this paper, we employ the 
parameterization method to compute high order Taylor expansions of the
local stable/unstable manifolds attached to periodic orbits
of billiard maps.  Globalizing appropriate fundamental domains
locates transverse intersections, 
providing insight into the existence and location of chaotic invariant sets.
A key step in implementing the parameterization method is computing the composition
of a given polynomial with the (implicitly defined) billiard map, and we show that
this can be done efficiently using the discrete Fourier transform
(DFT). The DFT requires evaluating the composition
on a disk in the complex plane, so that we must first extend the billiard 
table/map to a complex domain.  This problem is addressed from a 
computational perspective.
\end{abstract}

\vspace{1em}

\begin{flushleft}
\textbf{Keywords:} Dynamical Billiards, Chaotic Dynamics, 
Stable/Unstable Manifolds, Parameterization Method, Birkhoff Conjecture

\textbf{2020 AMS Subject Classifications:} 37C83, 37M21, 37M05, 37D10
\end{flushleft}


\section{Introduction}\label{sec:intro} 


The game of billiards, beloved since the 
14th century, inspired the 
mathematical models proposed by Birkhoff
in his 1927 paper~\cite{Birkhoff_1927}. A goal of this paper was to
solidify the foundations of thermodynamics 
by restricting attention to a single gas particle.
Polygonal billiards on \correction{B-R1-7}{rational} tables are 
intimately related to geodesic flows on surfaces \cite{Gutkin_1986}, and 
generalizations of Birkhoff's models have applications in classical mechanics 
and optics 
\cite{Bordeianu_Felea_Besliu_Jipa_Grossu_2011b, Barutello_De_Blasi_Terracini_2023a, Damour_Henneaux_Nicolai_2003, Berglund_Kunz_1996}.
The reader will find more complete discussion of this history,
as well as many additional references, 
in the excellent sources of 
\cite{Gutkin_1986,Tabachnikov_2009a,Chernov_Markarian_2006,
Himmelstrand2013,Bordeianu_Felea_Besliu_Jipa_Grossu_2011b,
Bunimovich_2019,Bialy_Tabachnikov_2022,
Franks03, Gutkin_2012, Levi_Tabachnikov_2007, Martin_2016,Meiss_1992}.

Mathematical billiards are introduced formally in Section 
\ref{sec:dyna} but, loosely speaking, the idea is to consider an 
infinitesimal particle moving without friction on a planar 
billiard table.  The particle bounces off the table's 
boundary in an elastic collision 
and moves off in a new direction determined by the law of specular reflection.
That is, the angle of reflection, measured with respect to the tangent 
to the boundary, is equal to the angle of incidence. 
If the boundary of the table is a simple closed curve,
then location along the boundary is described by a single number. 
\correction{B-R1-8}{The angle of 
incidence is a second number and, taken together,
these describe the state of the system at the moment a bounce occurs.}
In between collisions, the particle moves in a straight line with 
constant velocity, and tracking this pair of numbers at consecutive
bounces reduces the system from 
a differential equation with impacts to a smooth map on a cylinder.  
The reduced system is known as a billiard map.  
When the boundary is strictly convex, the billiard map is an area preserving 
twist map. If the boundary is at least $C^3$, then orbits exist for all 
time (i.e. the flow is complete).  
The relation between the flow on the table and the 
billiard map is illustrated graphically in Figure \ref{fig:TableOrbits} 
for a pair of tables which, while simple, already suggest 
that billiard dynamics can be quite complicated.

\begin{figure}[tbh!]\label{fig:TableOrbits}
\begin{center}
 \includegraphics[width=0.33\textwidth]{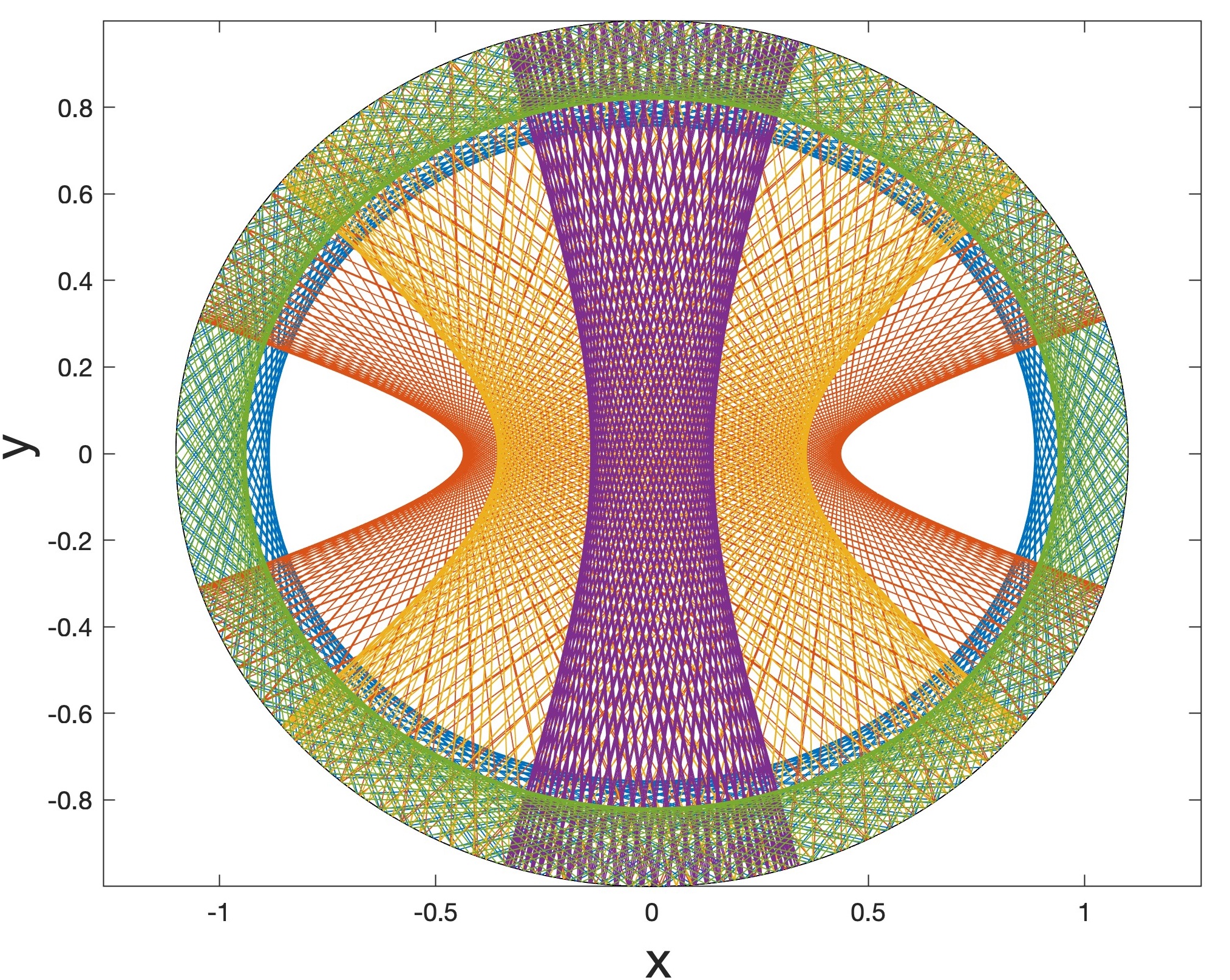}
 \includegraphics[width=0.33\textwidth]{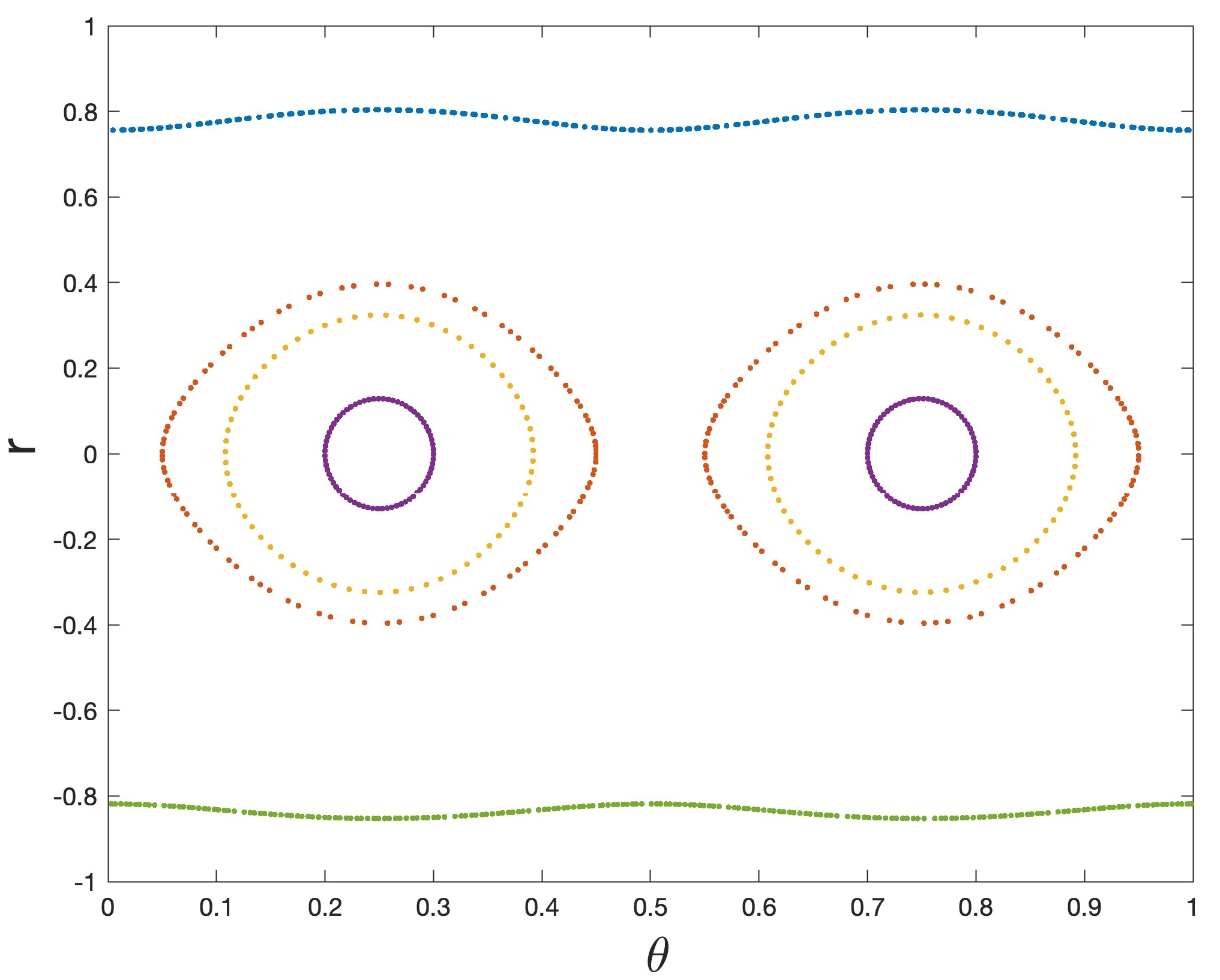}
 \includegraphics[width=0.32\textwidth]{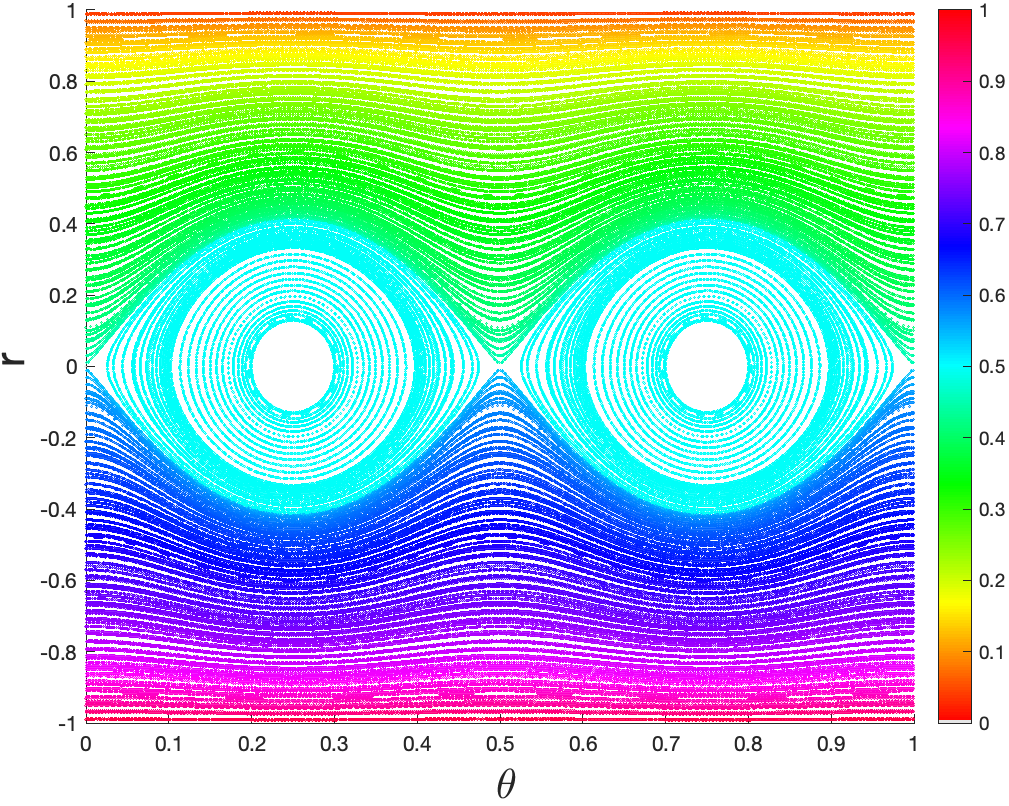}\\
 \includegraphics[width=0.33\textwidth]{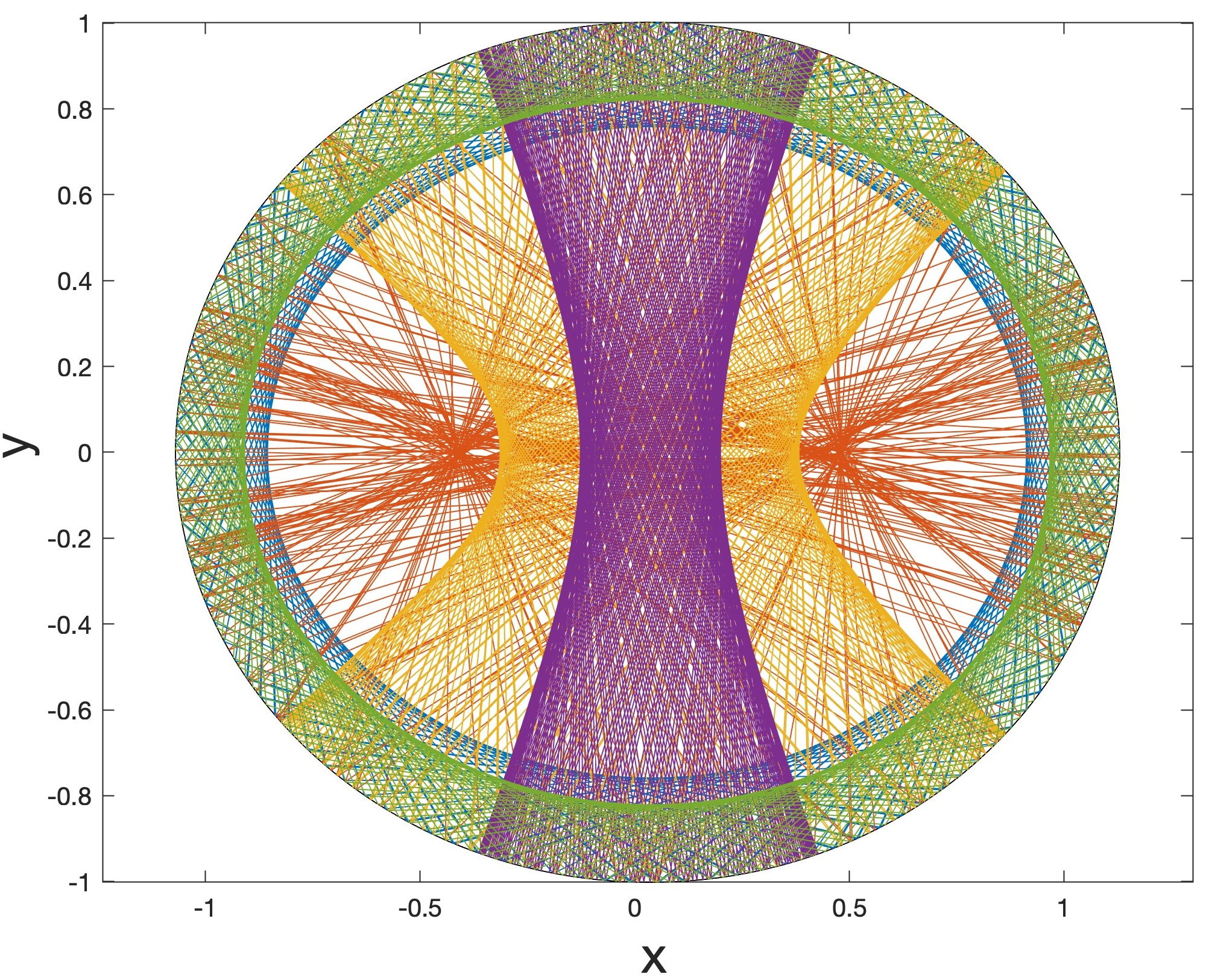}
 \includegraphics[width=0.33\textwidth]{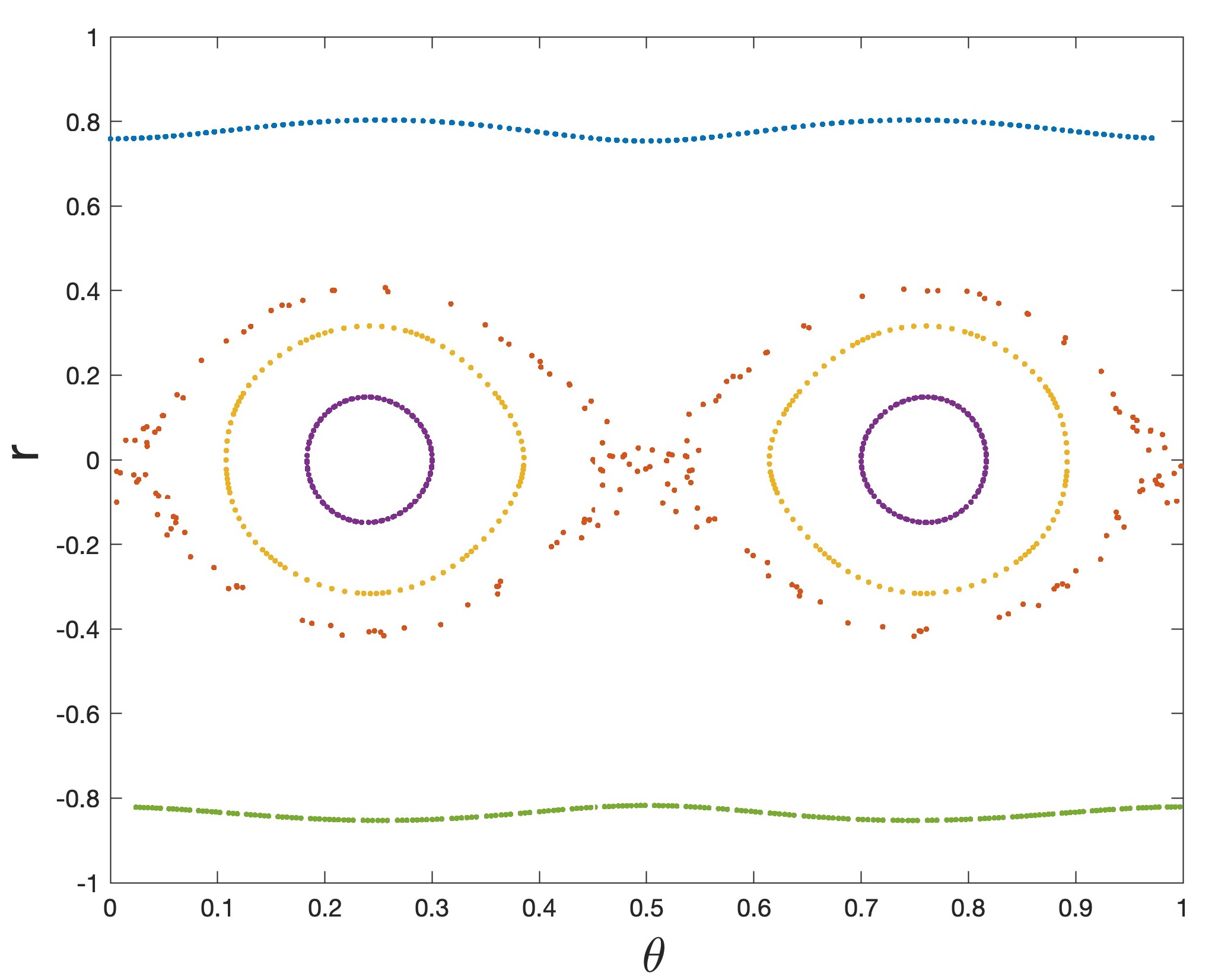}
 \includegraphics[width=0.32\textwidth]{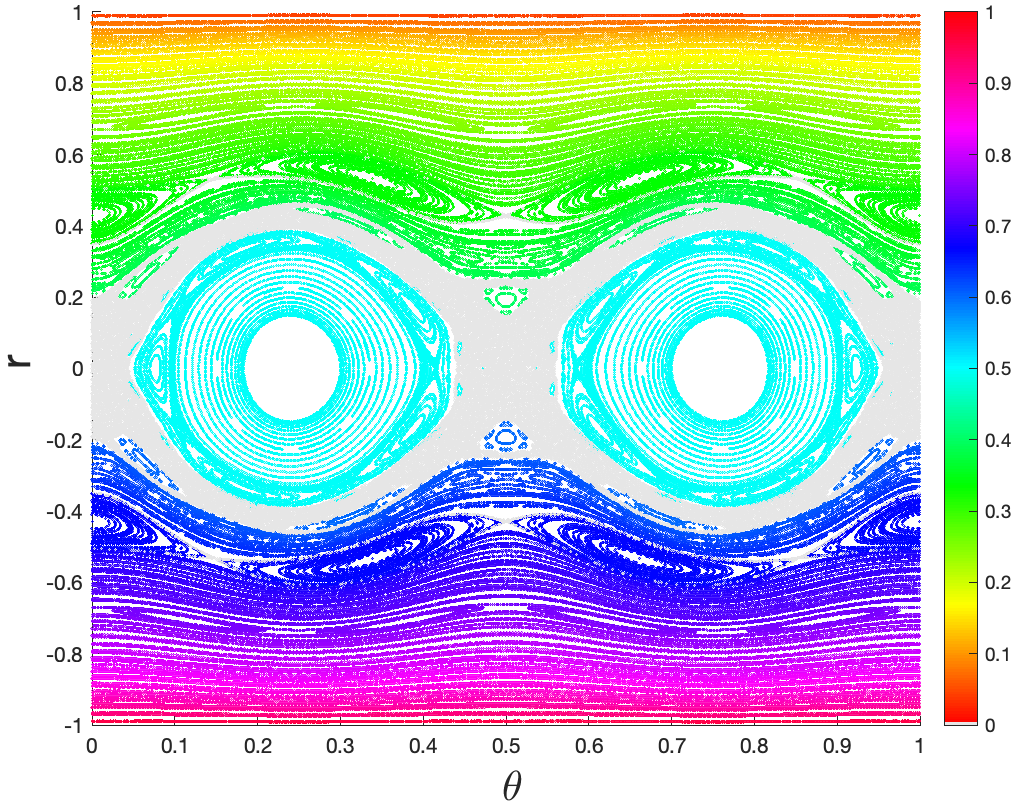}
\caption{ Top Left: Orbits on an elliptic table 
with eccentricity $\approx 0.4166$. Each trajectory has a caustic, 
i.e. a curve which is tangent to every line. Top Middle: Phase space orbit 
corresponding to the trajectory on the top left. The value of $\theta$ 
indicates the location of bounce, and  $r$ measures the angle of the bounce. 
Top Right: The full phase space of orbits is quasiperiodic, colored by their 
rotation number. Bottom Left, Middle: The table trajectories and corresponding phase space orbits, but for a small perturbation of the ellipse (Table A described in~\eqref{eq:pert_ell} and Table~\ref{table:coeff}). Note that some caustics 
remain, but that the red trajectory no longer has a caustic; its corresponding 
orbit is chaotic in phase space. 
Bottom Right: For the full phase space, many quasiperiodic orbits persist and 
are colored by frequency, but some orbits are now chaotic, colored in gray.}
\end{center}
\end{figure}

The discussion above makes it clear that the dynamics of a billiard map are
completely determined by the shape of the boundary of the table.  
An important conjecture in this context is 
the Birkhoff conjecture, which claims that a strictly convex billiard 
map is integrable if and only if the table is an ellipse
(this includes the case of the circle).
The Birkhoff conjecture serves as an organizing center for the theory
and we review the literature in Section \ref{sec:birkhoff}.  
See also~\cite{MR3868417, Bialy_Tabachnikov_2022, Schwartz_2022}.  
For the sake of the present introductory discussion 
we simply remark that, while elliptic billiards
are completely understood, the intuition informed by the 
Birkhoff conjecture is that non-elliptic tables result in 
richer dynamics.

Indeed, the Birkhoff conjecture is closely related to questions 
about chaotic dynamics. Chaos in billiards was first established by 
Bunimovich  in the 1970s~\cite{Bunimovich_1975, Bunimovich_1979}. 
He explicitly constructed tables with only focusing components 
(focusing trajectories like a lens) and proved that they 
admit chaotic dynamics. However, the example of Bunimovich is 
not strictly convex (his ``stadium'' like table has flat boundary segments).

There is by now a well developed literature for
strictly convex billiards showing that, under various generic assumptions, 
typical billiard maps are chaotic.  This literature  
is reviewed in Section \ref{sec:chaos}, but it must be
stressed that care is needed when appealing to genericity results.  
Billiard maps are not generic in the category of 
area preserving twist maps, nor are 
real analytic billiard tables generic in the $C^k$ category 
(even when $k = \infty$). 
When we fix our attention on a particular example, the simple fact that the 
boundary of the table can be written down means it is unlikely 
to be generic (though there must be  generic 
tables nearby).

Because of this, numerical simulation is still among the best ways to 
investigate the properties of explicit examples, and in this paper
we develop computational tools for studying chaotic dynamics in billiards.
More precisely, we implement a flexible and accurate numerical 
scheme based on the parameterization method 
of Fontich, Cabr{\'e}, and de la Llave 
\cite{Cabre_Fontich_delaLlave_2003,Cabre_Fontich_delaLlave_2003a,Cabre_Fontich_delaLlave_2005}
for computing high order Taylor expansions of 
local stable/unstable manifolds attached to periodic orbits 
in billiard maps.  Iterating appropriate fundamental domains
allows us to ``grow'' the local manifolds, often locating transverse homoclinic 
intersections.  

Assuming that the billiard table has a real analytic boundary,
our series expansions are guaranteed to converge, and 
heuristic a-posteriori indicators provide practical 
guides for determining reliable domains for the truncated series.
While the results of the present work are strictly numerical, 
we note that the 
parametrization method is well suited for 
computer-assisted proofs. See for example \cite{MR3068557,MR4344964}.
In a follow-up paper the computations described  
here will be validated, leading to computer assisted existence 
proofs for chaotic dynamics in explicit billiard problems.

\correction{B-R1-9}{
To illustrate the results obtained with our method, 
Figure~\ref{fig:PS1} depicts
stable and unstable manifolds, along with some
transverse homoclinic intersections,
for a variety of periodic orbits for the table of 
Figure ~\ref{fig:TableOrbits}.  Figures ~\ref{fig:PS2}
and~\ref{fig:PS3} show stable and unstable manifolds with transverse 
homoclinic intersections for four other perturbed elliptical tables. }

\begin{figure}[tbh!]
\begin{center}
 \includegraphics[width=0.7\textwidth]{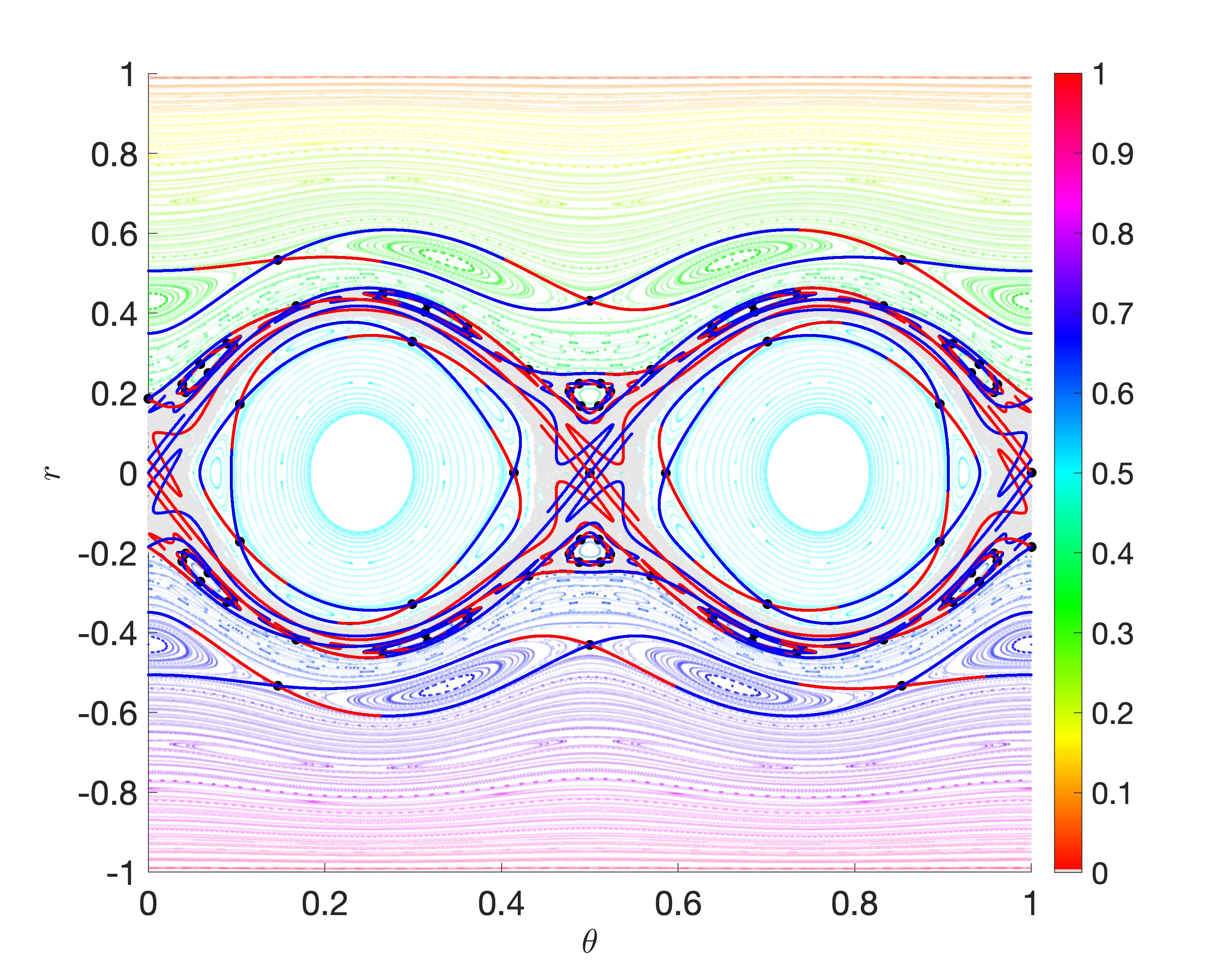}
 \caption{This figure shows the stable and unstable manifolds and their intersections for periodic orbits with a variety of periods.  This data is for  the same perturbed elliptical billiard table as shown in Fig.~\ref{fig:TableOrbits}, (Table A in Tbl.~\ref{table:coeff}).  \label{fig:PS1}}
\end{center}
\end{figure}
\begin{figure}[tbh!]
\begin{center}
 \includegraphics[width=0.45\textwidth]{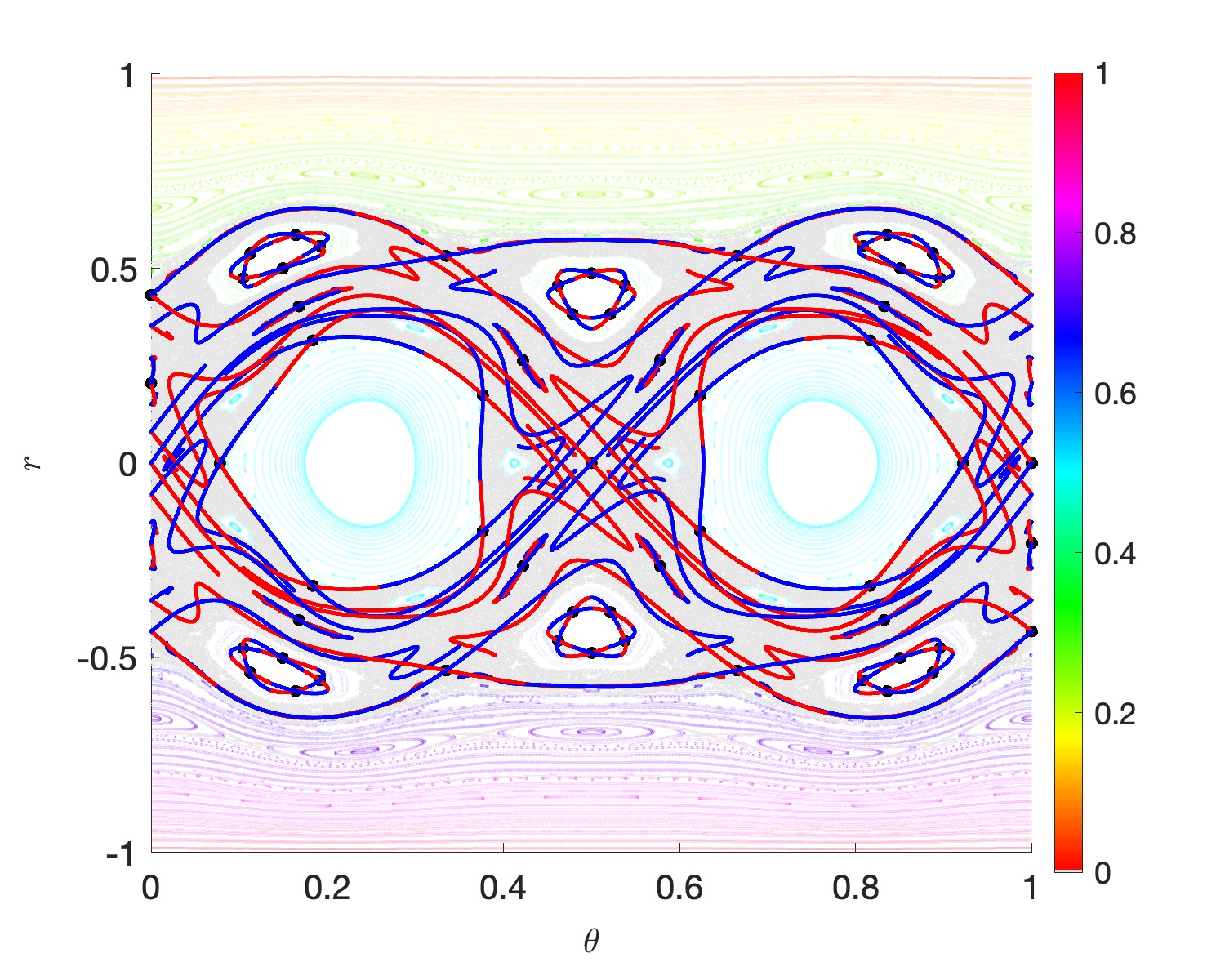}
  \includegraphics[width=0.45\textwidth]{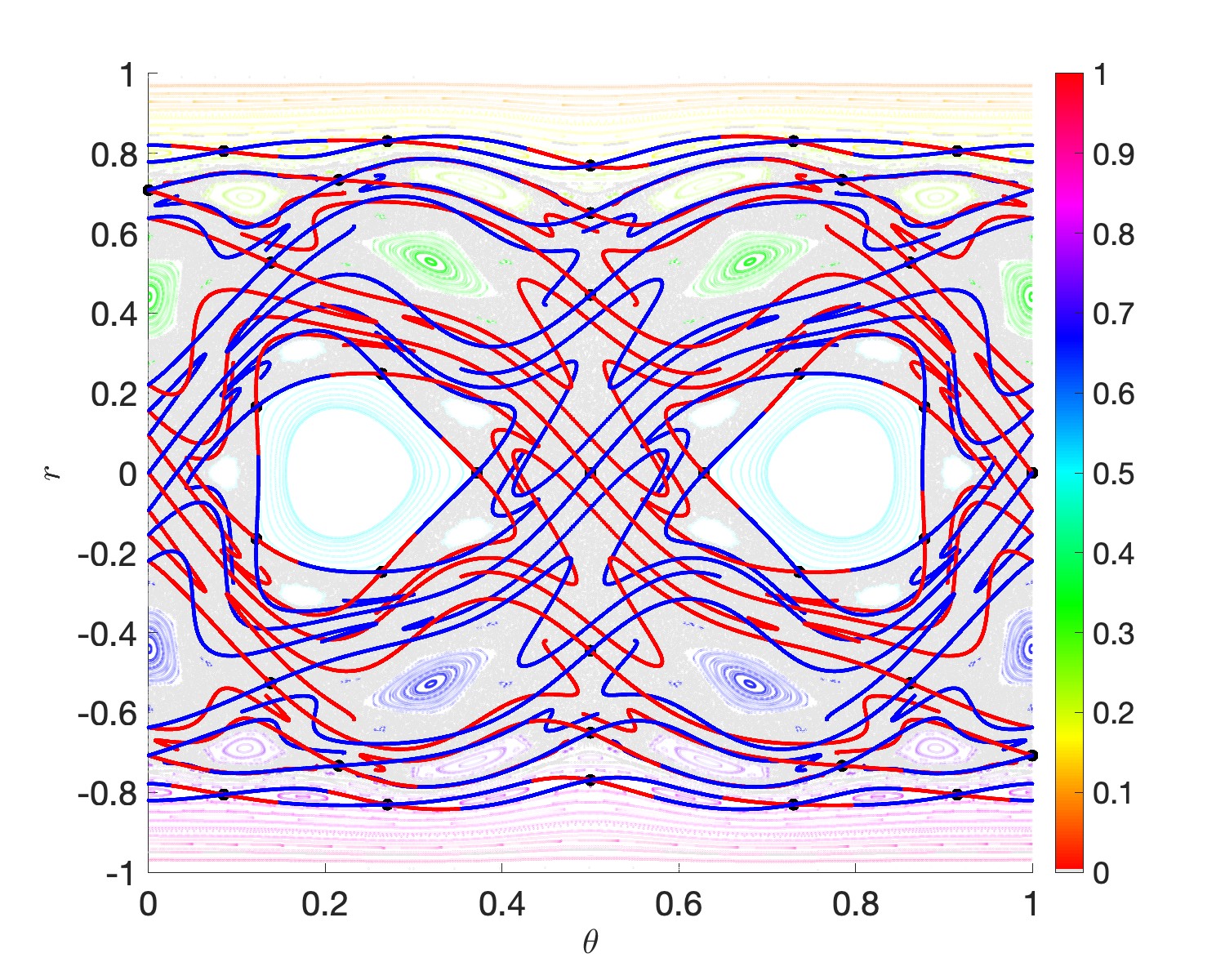}
\caption{The figure shows stable and unstable manifolds for periodic orbits computed for Tables B and C (cf. Tbl.~\ref{table:coeff}), larger perturbations of 
the same ellipse as in Figs.~\ref{fig:TableOrbits} and~\ref{fig:PS1}. The phase space in the background is colored by frequency for regular orbits and gray for chaotic orbits. 
\label{fig:PS2} }
\end{center}
\end{figure}
\begin{figure}[tbh!]
\begin{center}
 \includegraphics[width=0.45\textwidth]{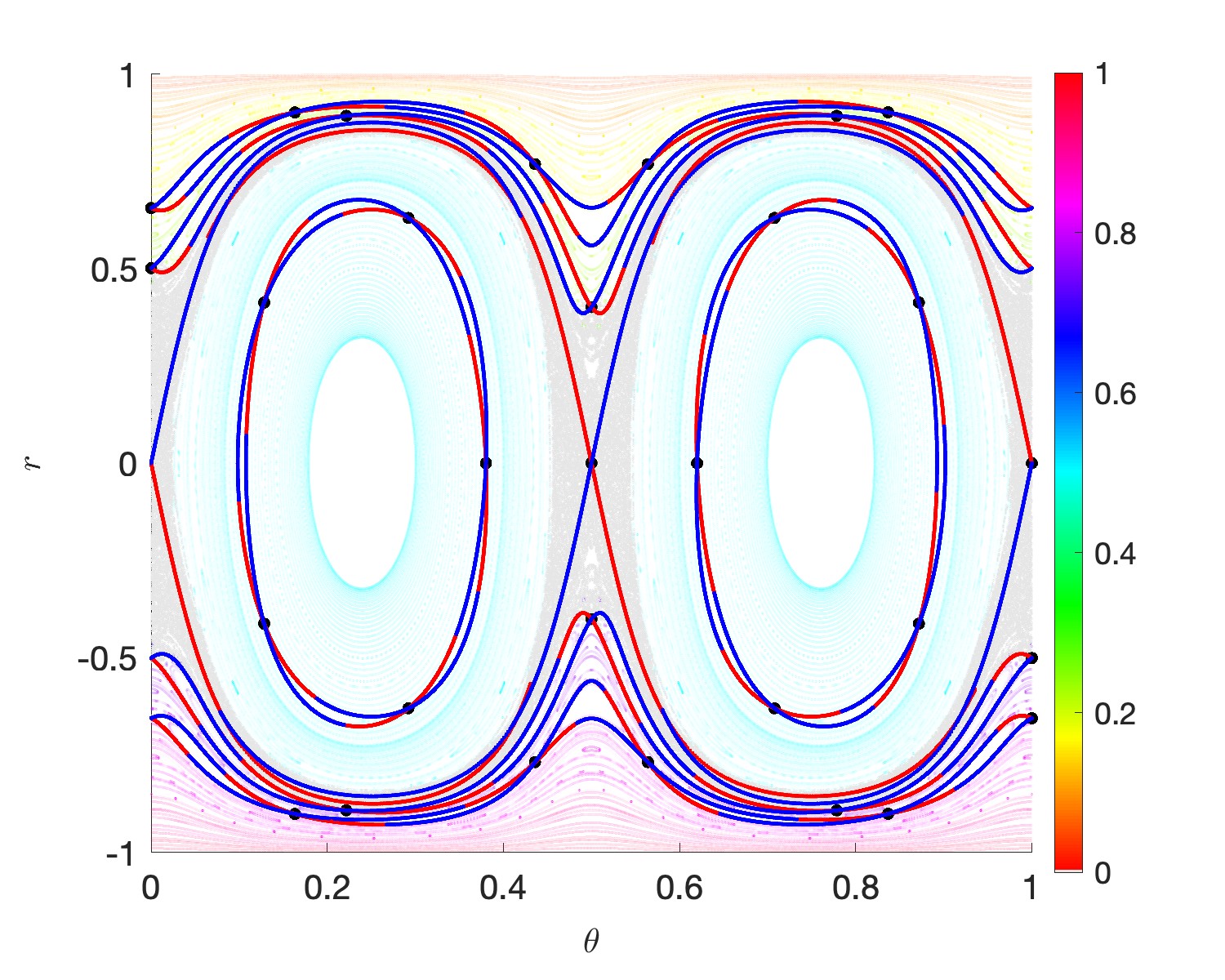}
 \includegraphics[width=0.45\textwidth]{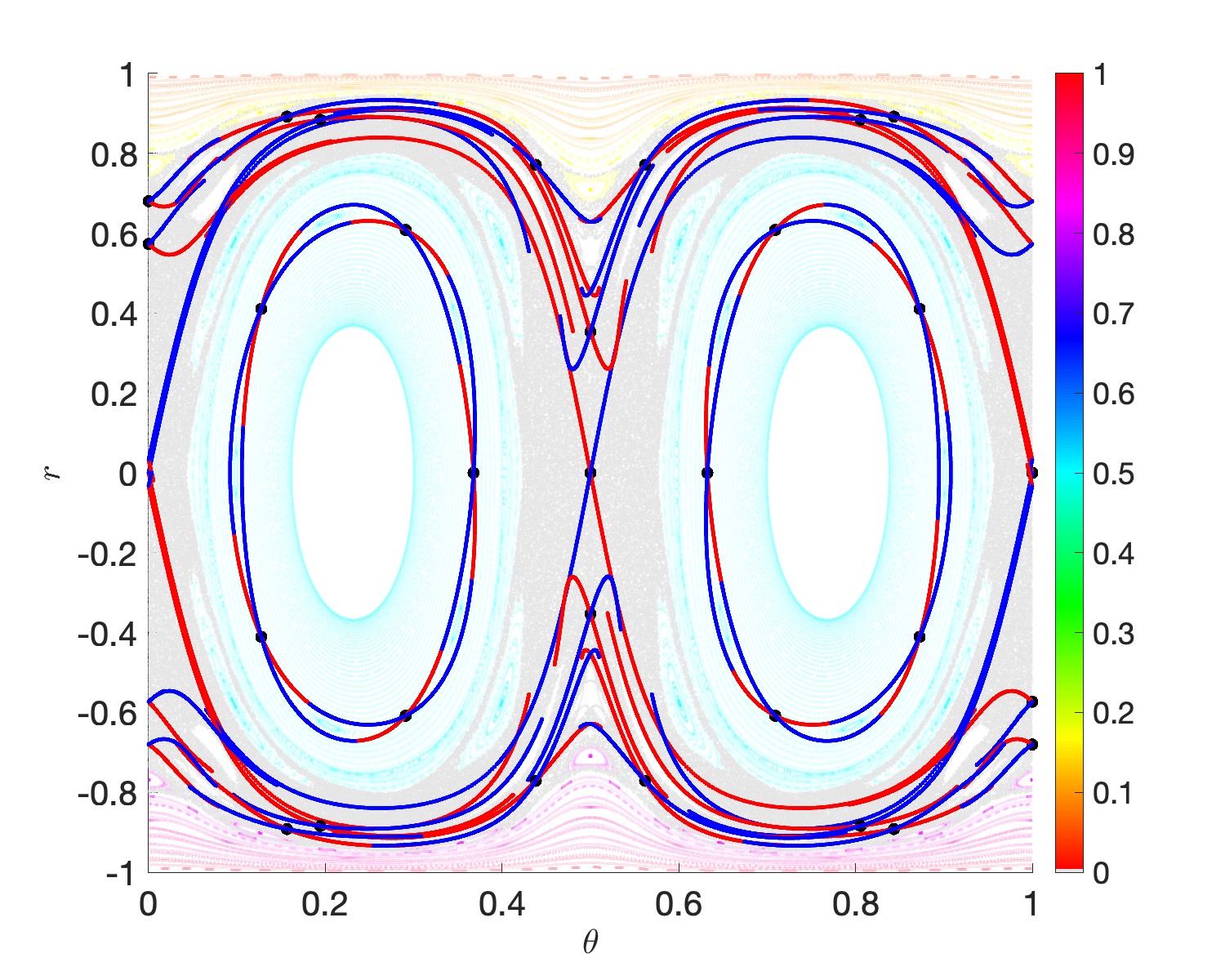}
\caption{The figure shows stable and unstable manifolds for periodic orbits computed for Tables D and E (cf. Tbl.~\ref{table:coeff}).
These tables are perturbations of an ellipse with larger eccentricity. The phase space in the background is colored by frequency for regular orbits and gray for chaotic orbits.   \label{fig:PS3}}
\end{center}
\end{figure}

The desired series expansions are computed via an iterative
procedure which requires computing the composition 
(on the level of power series) of the billiard map with the current 
polynomial approximation.
If the billiard map were polynomial, 
such a composition could be worked out explicitly via Cauchy products, 
resulting in explicit recursion relations for the power series coefficients.
See, for example, the worked problems in 
\cite{MR3706909,fleurantin2020resonant}, and also 
Appendix \ref{sec:manipulationPowerSeries} of the current work.

Elementary non-polynomial nonlinearities can be treated by appending 
the differential equations describing the nonlinear terms, and 
several examples of this procedure are given in the first reference just cited,
and also in Appendix \ref{sec:manipulationPowerSeries}.
Moreover, formal series techniques for implicitly defined 
maps are developed in \cite{MR3032844, timsina_parameterized_2022}.  

Nevertheless, when the implicit map is complicated enough, the techniques in 
the references just cited are unwieldy.  Examples of 
such situations include Poincar\'{e} maps associated with 
differential equations, and also the billiard 
maps studied in the present work.
In these situations, an alternative is to compute the 
desired composition via interpolation.  The discrete Fourier transform 
(DFT) is known to provide especially accurate, stable, and 
efficient interpolation for power series, 
and is employed throughout our computations.  

A complication is that the DFT requires evaluating the composition 
on the boundary of a disk in the complex plane. 
This involves non-trivial modifications of the numerical implementation
of the map.  We refer to the recent work of Kumar \cite{MR4848687,MR4361879,MR4209693}, 
 where the parameterization method is 
implemented for Poincar\'{e} sections for problems in 
Celestial Mechanics applications. There, it was necessary to extend both the 
numerical integrator, and the surface of section in an analytic way.

In the present work, we perform a similar complex extension for
billiards.  That is, we extend the boundary curve to 
$\mathbb{C}^2$, and define a complex billiard map by considering 
particle motion on a complex line.
We make no claim that this complex billiard map has physical 
significance. Rather, it is a numerical convenience. The 
manifolds we compute are ultimately real.  This said, it is interesting to 
note that analytic continuation for billiards has appeared 
in the theory, and we refer for example to the discussion 
following Remark 9 in the introduction of the paper
\cite{MR3815464}, where the authors explain 
that the main idea of their proof is to analyze the singularities
in the complex plane associated with the action angle 
form of perturbed elliptic billiards.  Note that 
the coordinates used in the reference just cited
are chosen for their theoretical properties, and do 
not appear to be easy to work with numerically.

The coordinates used in the present work, on the other hand, are chosen 
to obtain reliable numerical calculations. 
Indeed, for the purposes of the present work, the most important
thing is to have a single formulation of the complex map which works uniformly 
well throughout the complex phase space.
Because of this, we avoid certain simplifications which could be used
to improve efficiency of the map evaluation except on sets of measure zero.  
Our work builds on results in the PhD Thesis of the first author \cite{bishop24}.

We note that the focus on (strictly) convex boundaries in the present 
work is mainly for convenience.  
Non-convex boundaries lead to discontinuities, 
making numerical iteration more delicate.  This said,
our method applies in principle to periodic orbits for non-convex billiards, 
so long as the boundary of the table is analytic in a neighborhood of the 
periodic orbit. That is, as long as the table is piecewise 
real analytic, and the periodic orbit does not ``collide'' with any ``joint''.  
We will return to this topic in a future work.

\medskip

The remainder of the paper is structured as follows:  
Section~\ref{sec:definitions} introduces the 
setup, definitions, notation, and prior results. Section~\ref{sec:numerics1} gives 
the details of our numerical methods for iterating the billiard map on any 
perturbed elliptical table with smooth boundaries that is convex. 
In Section~\ref{sec:planarMaps}, we discuss how we compute the stable 
and unstable manifolds for periodic orbits. 
In Section~\ref{sec:results}, we detail our results, presenting 
computed stable and unstable manifolds for periodic orbits on 
five different billiard tables shown in 
Figures~\ref{fig:manifolds_p10}--\ref{fig:manifold_per2tab4_10it}. Section~\ref{sec:conclusions} 
contains our conclusions and an outline of future work.

\section{Definitions, background, and prior results}\label{sec:definitions}

This section introduces the definitions and notation 
along with a review of prior results. 


\subsection{Dynamical billiards -- from physical table to phase space}\label{sec:dyna}

We focus on connected, strictly-convex, planar domains with 
smooth (in fact real analytic) boundaries, as done in \cite{Himmelstrand2013}.
Such a billiard system is given by the following data. 
\begin{definition}\label{bill_system}
A billiard system consists of a point particle confined to move in a 
connected, strictly convex, and closed planar 
domain $\mathcal{D} \subset \R^2$ with boundary curve $\Gamma$. 
We assume that $\Gamma$ is a smooth non-self-intersecting curve. 
We further assume that $\Gamma$ is sufficiently smooth that it can be 
parameterized by a real analytic map $B: [0,1] \to \R^2$,  
where $B(0) = B(1)$, and $B'(\theta)$ is a 
nonvanishing tangent vector. 
We assume that $B$ has counter-clockwise orientation on $\Gamma$.

The point particle moves within domain $\mathcal{D}$ according to the following criteria. 
\begin{enumerate}
\item The trajectory curve $G(t) \in \mathcal{D}$ for all $t \ge 0$. We require $G(t)$ to be a continuous union of line segments, such that for any time interval for which $G$ is in the interior of $\mathcal{D}$, $\mathbf{v}(t) = G'(t)$ is constant, and $|\mathbf{v}(t)| = 1$.  
\item The trajectory of the particle begins at $t=0$ at a point  $G(0) \in \Gamma$ on the boundary of the domain.
\item Whenever the trajectory $G(t)$ reaches the boundary,  it is reflected elastically back into the interior of $\mathcal{D}$ according to the law of 
specular reflection. 
\end{enumerate}
\end{definition}
Based on the above definition, the particle trajectory is fully determined by the location of the boundary collisions, along with the direction of vector $\mathbf{v}$ after the collision. 
Let $\{t_k\}_{k=0}^\infty$  be a discrete increasing sequence of times at which the particle hits the boundary $\Gamma$, and assume that the particle experiences no collisions for 
$t_{k} < t < t_{k+1}$.
Define $\theta_k$ by $G(t_k) = B(\theta_k)$. 
Let $\gamma_k$ be the angle that $\mathbf{v}(t)$ makes with the tangent vector $B'(\theta_k)$ immediately after the bounce, and let \[r = \cos{\gamma}, \mbox{ where } -1<r<1. \]
Then the vector $(\theta_k,r_k)$ defines the $k$-th point of collision\footnote{To avoid cumbersome notation, when we talk about a vector, unless otherwise specified, we implicitly assume that it is a column vector.}.
The restriction on $r$ is due to the fact that $|r| \ge 1$ corresponds to the particle  leaving the convex billiard table. 

For a given boundary parametrization $B(\cdot)$, there is a uniquely determined diffeomorphism $f$ on $S^1 \times (-1,1)$, which is as smooth as the boundary, such that 
\[
f \begin{pmatrix} \theta_k\\ r_k \end{pmatrix} = \begin{pmatrix}
    \theta_{k+1} \\
    r_{k+1}
\end{pmatrix}.
\]
When the boundary is strictly convex, it
is standard to extend continuously to $r \in [-1,1]$
by taking $f(\theta, \pm 1)$ fixed.

\correction{B-R1-1}{Birkhoff showed in ~\cite{Birkhoff_1927} that, if 
$\Gamma$ is a convex curve, then $f$ is an area preserving twist map.  See also~\cite{Himmelstrand2013,Franks03}}.

Before moving to the general case, we review the simplest smooth tables. 
\begin{lemma}[Circular Table]
The billiard map on a circular table with boundary parameterized by
\[B(\theta) = \begin{pmatrix}
    \cos (2 \pi \theta) \\
     \sin (2 \pi \theta)
\end{pmatrix}, \; 0\le \theta<1\] \correction{B-R2-13}{has the following explicit solution:
\[ f \begin{pmatrix} \theta\\ r \end{pmatrix} = \begin{pmatrix}
    \theta + C(r) \\ r
\end{pmatrix},   \]
where $C(r) = \arccos (r)/ \pi$ depends only on $r$. In particular, an orbit starting at $(\theta_0, r_0)$ satisfies $r_k = r_0$ for all $k$, and $\theta_k = \theta_0 + kC(r_0) \pmod 1$.}
\end{lemma}

This implies that for the circle, orbits consist of rigid rotation on horizontal lines in the phase space $\theta$ versus $r$.
 The proof  of this result is due to Birkhoff~\cite{Birkhoff_1927}. 

\correction{B-R1-10}{
\begin{lemma}[Ellipse]
For any ellipse, the function $f$ is fully integrable, meaning that there is a continuous non-constant quantity (integral) that is preserved along each trajectory. 
\end{lemma}
Interestingly, for an elliptical table, the eccentricity of the caustic conic section of each physical orbit can be used as an integral for orbits of $f$ in the phase space. }
\correction{B-R2-3}{
Also of note for an ellipse, all orbits are quasiperiodic, periodic, or homoclinic to a period-two orbit. See~\cite{siburg:04} p.42--43 for more details.
} We 
 will give more details on how to compute $f$  in the next section. 


\subsection{Birkhoff conjecture}\label{sec:birkhoff}

As already remarked above, when the boundary of a table is strictly convex, 
then its billiard map fixes top and bottom boundary circles of the cylinder 
(the boundary circles are invariant with rotation number zero).
A celebrated result of Lazutkin says that, for these strictly convex billiard maps, 
there exists a full measure Cantor set of invariant circles in a neighborhood 
of the fixed boundary circles \cite{lazutkin73}.  Their rotation numbers have
positive measure and accumulate to zero.  Moreover, 
these Lazutkin curves are robust under perturbations 
by a result of Popov \cite{MR1266488}.

The concept of integrability asks that the set of invariant circles be even richer.  
More precisely, one says that the area preserving twist map $f$ on the plane
is \textit{locally integrable} if an open subset of its phase space is 
foliated by smooth invariant curves (tori), and says that $f$ is 
\textit{globally integrable} if this is true on an open and dense subset
of the phase space.    
This definition makes room for the fact that a globally integrable map will 
often have separatrices of measure zero where the dynamics is not described
by invariant tori.  Note that we are maintaining a certain informality, in that 
we do not define caustics for billiards, which allow one to 
re-frame the discussion on the table rather than in 
the phase space of the map. The interested reader is directed to the 
review of \cite{MR3868417} for more precise definitions.

This said, we paraphrase the Birkhoff conjecture as follows: 
a strictly convex billiard map is globally integrable
if and only if its table is an ellipse.  The conjecture is first articulated
in a paper of Poritsky in 1950 \cite{MR32960}. However Poritsky was a 
postdoc at Harvard in the late 1920's working with 
Birkhoff, and he attributes the question to his supervisor.  
The hypothesis of strict convexity is 
essential, as a theorem of Mather shows that vanishing curvature at a single 
point destroys the Lazutkin curves, and with them all hope of integrability.  

A first major step in the direction of a resolution is found in the work of Bialy, who 
showed that, if the phase space of a billiard map is completely foliated by 
invariant circles which are not null-homotopic, then the table is a circle 
\cite{MR1234604}.  Note that this is a stronger assumption than 
global integrability, as no set of measure zero is excluded in the hypothesis.  
Later, the algebraic version of the conjecture was established by Glutsyuk 
\cite{MR4210728}.  This says that, if the integral of motion is polynomial 
in the velocity variables, then the table is an ellipse.

In \cite{MR1373998}, Delshams and Ram\'irez-Ros proved the first perturbative 
version of the Birkhoff theorem, showing that symmetric, non-constant, entire perturbations 
of the ellipse destroy integrability.  That is, they showed that -- 
within a certain category of 
perturbations -- an integrable billiard close to an ellipse must be an 
ellipse.  This result was extended to more general perturbations of the 
circle by Avila, De Simoi, and Kaloshin in \cite{MR3548532}.
Another deep result in this direction was established by Kaloshin and Sorrentino
in \cite{MR3815464}, which shows that if a billiard map has  enough rational 
invariant circles, then it is an ellipse.  
In Appendix F of the paper just cited, the authors sketch a proof of the complete
global (non-perturbative) Birkhoff conjecture, introducing a novel curve shortening 
flow.  They are however careful to say that completing this sketch would require 
overcoming serious technical obstacles.  

Very recently an extension by Koval has appeared showing that, 
for almost every eccentricity $e$, a smooth enough integrable billiard sufficiently 
close to an $e$-ellipse is itself an ellipse \cite{MR5008167}.  Note that the notion 
of integrability considered here focuses on the dynamics near the fixed boundary 
circles, and that
this is one of the most complete results to date: establishing a version of 
the Birkhoff conjecture except on a set of measure zero.
Another recent result due to Bialy and Mironov
shows that centrally-symmetric billiard tables 
satisfy the Birkhoff conjecture \cite{MR4429262}, and 
even more recently Kaloshin, Koudjinan, and Zhang have 
completed a perturbative extension \cite{MR4823214}.

Note that an interesting and subtle point has emerged in recent years.
Namely, Treschev has constructed a numerical example of a system 
which appears to be locally integrable in an open neighborhood of its 
period-two orbit, far from the fixed boundary circles of the billiard map,
without being locally integrable near the boundary circles
\cite{MR3064868}.  If this turns out to be true, it 
would show that the ``inner" and ``outer"
dynamics of billiard maps (far from and near the boundary respectively)
are more independent than once suspected.  In this case,
the Birkhoff conjecture
would be rephrased as follows: every billiard which is locally integrable 
near its boundary is an ellipse.
This would not lead to any contradiction, as a close look at the notions of 
integrability in all of the preceding works shows that the hypotheses 
treat a neighborhood of the boundary.

Finally we note that the higher dimensional situation seems less clear.
While billiards with ellipsoid boundaries
are integrable, Mironov and Yin have shown that certain conical tables
in $\mathbb{R}^d$ also lead to integrable billiards \cite{MR5086411}.
The appropriate generalization of the Birkhoff conjecture in 
$\mathbb{R}^d$ has, to the best of our knowledge, not yet emerged.

\subsection{Chaotic dynamics} \label{sec:chaos}

A complementary line of research asks: 
how common is chaotic behavior in billiard systems?
Let us focus once again on smooth convex boundary curves.
A theorem of Donnay in this context shows that there exist $C^\infty$
curves, $C^2$ close to the ellipse, having periodic orbits which generate homoclinic
tangles \cite{MR1993030}.  These systems therefore have positive topological 
entropy and are not integrable. 

A powerful $C^2$ result of 
Carneiro, Kamphorst, and Pinto-de-Carvalho 
\cite{MR2356119}, proven independently in the $C^\infty$ category by 
Xia and Zhang \cite{MR3207929}, says that generic billiards have (i) finitely 
many orbits of period $K$, (ii) at least one of these period $K$ orbits 
is hyperbolic, and (iii)
the stable and unstable manifolds of the hyperbolic orbit intersect transversally.
Again, this gives rise to invariant sets with positive topological entropy 
and rules out integrability.  A very recent work by 
Baldom\'{a}, Florio, Leguil, and Seara extends this result to  
billiards with real analytic boundaries 
\cite{baldoma2026chaoticitygenericanalyticconvex}.  
These results show that chaos is ``typical'' in billiard maps.
Generalizations to higher dimensional billiards are found for example in the 
works of Delshams, Fedorov, and Ram\'irez-Ros \cite{MR1862816} 
as well as in the work of Bessa, Del Magno, Lopes Dias, Gaiv\~ao, and Torres 
\cite{MR4710877}.

Many other studies have focused on numerical simulation of billiard systems, and 
we refer the interested reader to 
Levi and Tabachnikov~\cite{Levi_Tabachnikov_2007} for a method
which utilizes the minimization property to find iterates of the map.  
We also note that 
Lansel and Porter \cite{Lansel_Porter_2004} and \cite{Turaev_2016} simulate 
billiard systems that are not necessarily convex, 
while Solanp{\"a}{\"a} et al.~\cite{Solanpaa_Luukko_Rasanen_2016} consider more general nonconvex billiard systems with multiple charged particles subject to magnetic fields. 
In ~\cite{daCosta_Hansen_Silva_Leonel_2022}, da Costa et al. describe billiard 
dynamics on oval-like tables using a polar equation, with a circle approximation 
to estimate the next point of contact with the boundary. 
See also~\cite{ Anderson_2021,Bordeianu_Felea_Besliu_Jipa_Grossu_2011b,Church_2021,
Datseris_2017, Knill_1998, Martin_2016,Pnueli_2022,Plum_2023, Reznik_2021}.
We only note that none of these works required numerical continuation of 
the billiard map to complex tables and lines.  

In \cite{ramirez-ros:05} Ram{\'i}rez-Ros numerically computes high order
expansions for the stable/unstable manifolds of the period two orbit 
in some perturbed elliptic billiards, in order to study 
exponentially small separatrix splitting for these systems.
The philosophy behind these computations is quite different from the 
one developed in the present work, as the author expands around the 
known local stable/unstable manifolds of the period two point in the 
elliptic billiard.  In particular, the author of the 
work just cited computes coefficients by exact recursion rather 
than by interpolating via the DFT.
On the other hand, our computations are designed to work for arbitrary periodic 
orbits and far from the perturbative regime.

\section{Numerical computation of the billiard map}\label{sec:numerics1}

We now discuss our numerical procedure for evaluating 
the real billiard map.


\subsection{Table shapes}\label{sec:tableshapes}
Note that an ellipse with eccentricity $\sqrt{1 - b_1^2/a_1^2}$ 
has parameterization 
\begin{equation}\label{eq:ellipsetable}
E(\theta) = \begin{pmatrix}
    a_1 \cos (2 \pi \theta) \\ b_1 \sin (2 \pi \theta)
\end{pmatrix}. 
\end{equation}
We focus on perturbed ellipses of the form 
$B(\theta) = (x(\theta),y(\theta))$ 
with Fourier series given by  
\begin{equation}\label{eq:pert_ell}
\begin{split}
x(\theta ) &= a_{1,1} \cos{(2\pi \theta )} + \sum_{k = 2}^n a_{k,1} \cos{(2k\pi \theta)} + b_{k,1} \sin{(2k\pi \theta)} \\
y(\theta ) &= b_{1,2} \sin{(2\pi \theta)} + \sum_{k = 2}^n a_{k,2} \sin{(2k\pi \theta )} +  b_{k,2} \sin{(2k\pi \theta )}.
\end{split}
\end{equation}

\begin{table}[tbph!]
    \centering
    \small{
    \begin{tabular}{|c|c|c|c|c|c|}     \hline
        \textbf{Table} & A & B & C & D & E  \\ \hline
        $(a_{1,1},a_{2,1},a_{3,1})$ & (1.1,0.03,0) & (1.1,0.05, 0.00015) & (1.1, 0.08, 0.0002) & (2, 0.04,0) & (2, 0.05,0) \\ \hline 
        $(b_{1,2},b_{2,2},b_{3,2})$ & (1,0.03,0) & (1,0.035, 0.0001) & (1, 0.095, 0.0001) & (1, 0.035,0) & (1, 0.065,0) \\ \hline
        Eccentricity & 0.4166 & 0.4166 & 0.4166 & 0.8660 & 0.8660
        \\ \hline 
    \end{tabular}  
    }
          \vspace{.1in}
    \caption{The coefficients used for our choice of perturbed elliptical billiard tables, referred to as Billiard Tables A--E. Any unspecified coefficients are zero. The bottom row states the approximate eccentricity of the associated unperturbed ellipse. 
    \label{table:coeff}}
\end{table}

Notice that for a table specified by~\eqref{eq:pert_ell}, if $a_{1,1}$ and $b_{1,2}$ are the only nonzero coefficients, then the table is an ellipse. In particular, for any table, the {\em ellipse associated with this table} is the ellipse with the same $a_{1,1}$ and $b_{1,2}$ values (all other coefficients set to zero). 
The particular tables we have used for our numerics are referred to as Billiard Tables A--E. The coefficients $a_{k,i}$ and $b_{k,i}$ are \correction{B-R2-5}{given in Tbl.~\ref{table:coeff}}. For Tables A-C, the associated ellipse has $a_{1,1} = 1.1$ and $b_{1,2} = 1$. For Tables D and E the associated ellipse has $a_{1,1} = 2$ and $b_{1,2} = 1$. 
Tables A-E and the associated ellipses are \correction{B-R2-6}{depicted in Figure~\ref{fig:tables}}.

\begin{figure}[tbph!]
\centering
 \includegraphics[width=0.49\textwidth]{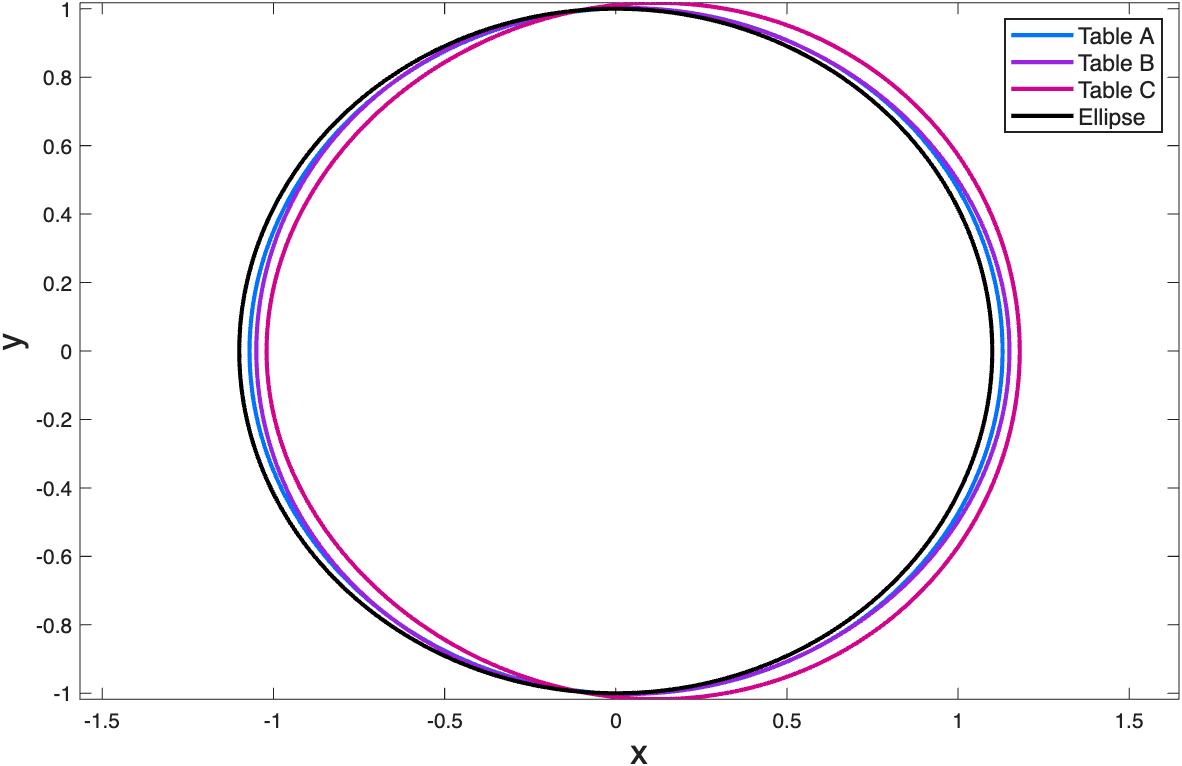}
 \includegraphics[width=0.49\textwidth]
 {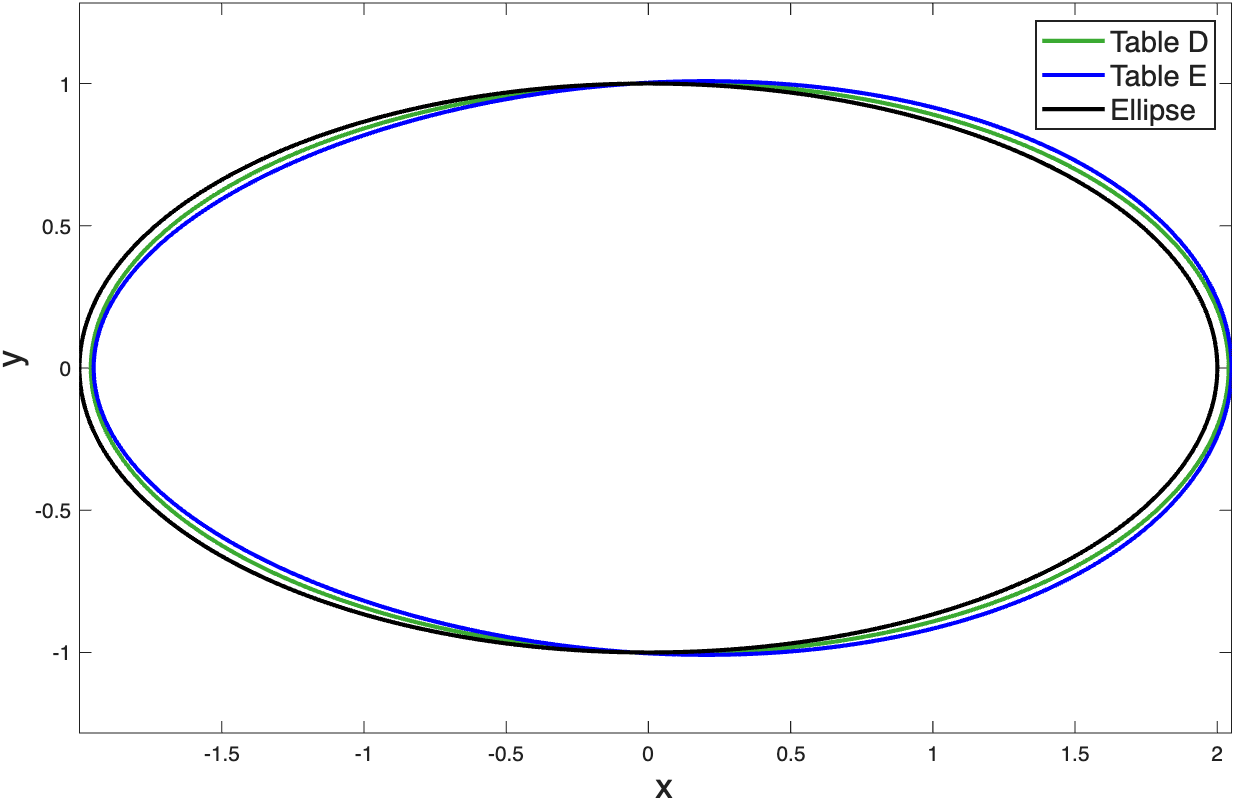}
   \caption{
 Shape of billiard tables used in the numerics. (Left) Billiard Tables A-C are perturbations of an ellipse with  eccentricity $\approx 0.4166$.
(Right) Billiard Tables D and E are perturbations of 
an ellipse with eccentricity $\approx~0.8660$. 
 \label{fig:tables}
}
\end{figure}
Since ellipses are convex and since $B(\theta)$ varies continuously with respect to the coefficients, the table remains convex when the higher-order Fourier coefficients are sufficiently small -- specifically, if for $i=1,2$ and for all $k>1$, $|a_{k,i}|$ and $|b_{k,i}|$ are significantly smaller than $|a_{1,1}|$ and $|b_{1,2}|$, the table remains convex. To illustrate how convexity is lost as these perturbations grow, consider the parameterized table with coefficients given by:
\begin{equation}\label{eq:convtable}
\begin{split}
(a_{1,1},a_{2,1}) = (1.1, 0.03 \epsilon)\\
(b_{1,2},b_{2,2}) = (1,0.025 \epsilon).
\end{split}
\end{equation}
For $\epsilon = 0$, this table is an ellipse. Figure~\ref{fig:convex} (left panel) shows how the table shape changes as $\epsilon$ increases, while the right panel shows the minimum value of the signed curvature measured along the boundary of the table. At $\epsilon \approx 9$, the minimum signed curvature passes through zero, and the table loses convexity. 
 The tables given by (\ref{eq:convtable}) have the same associated ellipse as Billiard Tables A, B, and C; in particular Billiard Tables A-C have coefficients $a_{2,1}$ and $b_{2,2}$ with $\epsilon < 4$, which is clearly before convexity is lost. 
\begin{figure}[tbph!]
\begin{center}
\includegraphics[width=.9\textwidth]{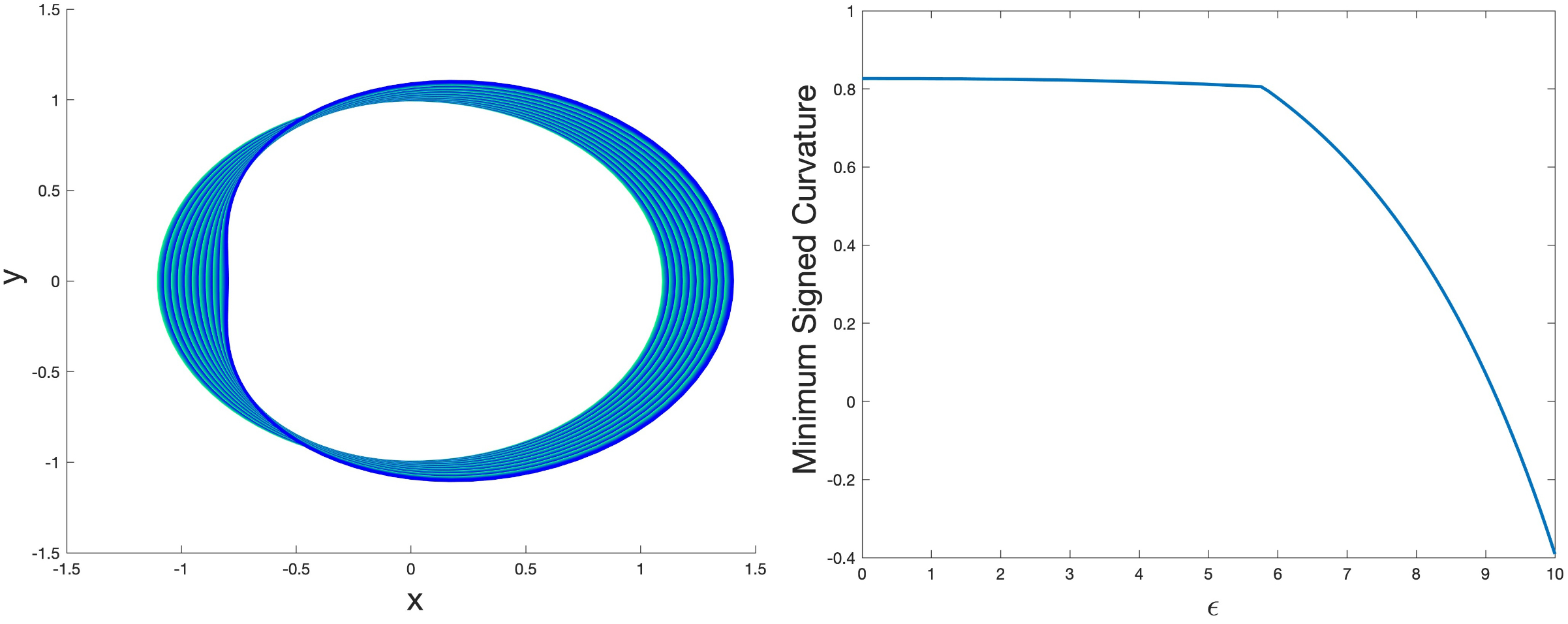}
\caption{(Left) The shape of a table  with the nonzero coefficients given by ~\eqref{eq:convtable} for values of $\epsilon$ between 0 and 10. (Right) The minimum signed curvature of the table. When $\epsilon \approx 9$, the signed curvature is zero at a point on the boundary, implying that the table has lost convexity.\\
\label{fig:convex}}
\end{center}
\end{figure}

\subsection{Iteration for elliptical tables}\label{sec:numellipse}

In order to compute the billiard map for general tables, we use the fact that we can directly compute the billiard map for an ellipse. We compute the location of each bounce on the table's corresponding ellipse, then use this as the initial guess for Newton's method to find the location of the bounce on the perturbed table.
In this section we explain how to compute the ellipse case. 

Let $B(\theta ) = (x(\theta),y(\theta)) = (a_1  \cos{(2 \pi \theta)}, b_1 \sin{(2 \pi \theta)}) $ parameterize an ellipse, as introduced in~\eqref{eq:ellipsetable}. Given an input $(\theta,r)$, we seek $(\hat{\theta},\hat{r})$ such that $f(\theta,r) = (\hat{\theta},\hat{r})$. 
Based on $r$, we can use the arccosine function to find the direction vector $\mathbf{v}$. 
Since $B(\theta)$ and $B( \hat{\theta})$ must
lie on the line parallel to the vector  $\mathbf{v} = (v_1,v_2)$, there exists $s_*$ such that 
\begin{equation}\label{eq:line}
B( \hat{\theta}) = B( \theta) + s_* \mathbf{v} .
\end{equation}
Let 
\[d_1 = v_1/a_1 \mbox{ and } d_2= v_2/b_1 . \]
Combining these equations with the equation for an ellipse and solving for $s_*$ gives 
\begin{equation}\label{eq:ellipsesstar}  s_* = -2 \left( \frac{d_1 \cos{(2\pi \theta)} + d_2 \sin{(2\pi \theta)}}{d_1^2 + d_2^2} \right) . \end{equation}
Plugging this value back into~\eqref{eq:line} to get $B(\hat{\theta}) = (\hat{x},\hat{y})$. 
We find $\hat{\theta}$ from $B(\hat{\theta})$ since 
\[ \frac{b_1}{a_1} \tan(2 \pi \hat{\theta}) = \frac{\hat{y}}{\hat{x}}. 
\]
This last statement is where the simplification of working with the ellipse is really clear, since for a perturbed ellipse, we cannot solve directly and must instead use a two-dimensional root finding method. 
Finally, from this point finding $\hat{r}$ is exactly the same as for the perturbed case, and it will be explained in the next section. 
\subsection{Iteration for general billiard tables}\label{sec:realmap}
\begin{figure}[tbph!]
 \begin{center}
\includegraphics[width = .8 \textwidth]{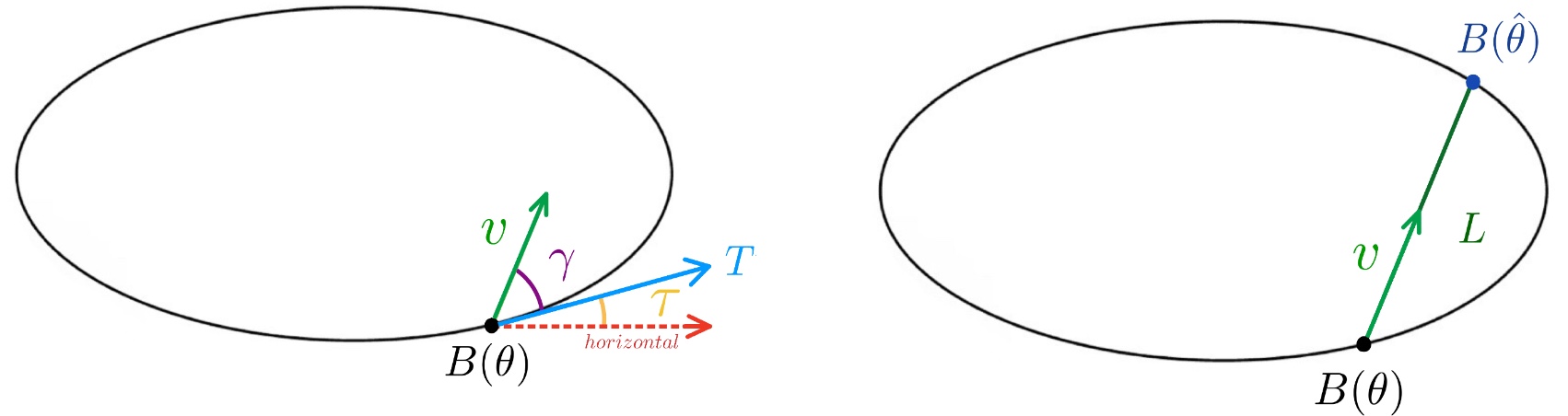}
\caption{ Notation for our numerical methods. Left: $\mathbf{v}, \gamma, T,$ and $\tau$. Right: The line $L$ contains $B(\theta)$ and $B(\hat{\theta})$. 
\label{fig:notation}}
\end{center}
\end{figure}
\begin{figure}[tbph!]
\begin{center}
\includegraphics[width=0.45\textwidth]{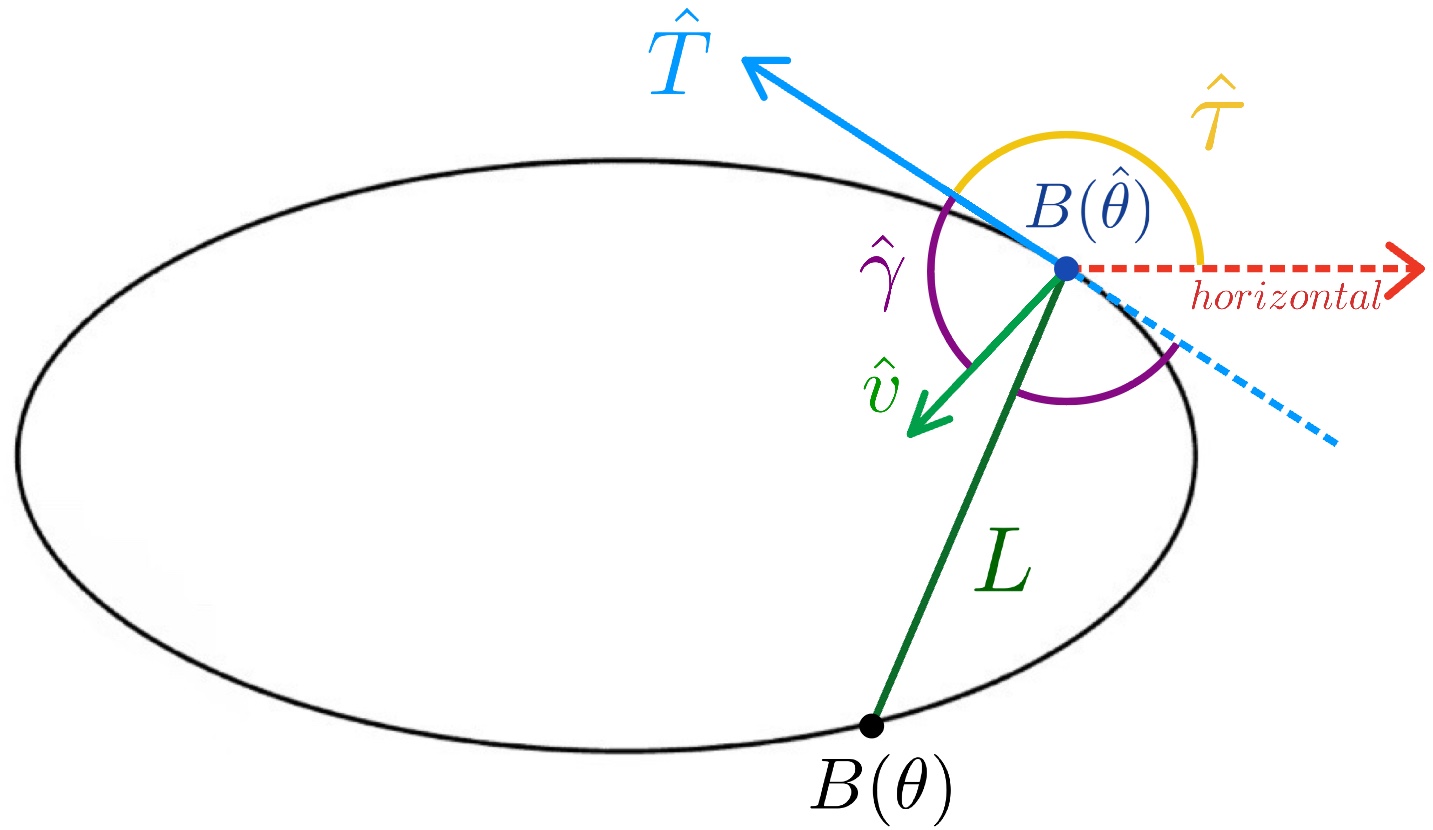}
\caption{Computation of the $\hat{r}$ from $\mathbf{v}$ and $\hat{\theta}$. \\
 \label{fig:reflectionangle}}
\end{center}
\end{figure}
%
%
Given  $\theta$ and $r$, we use the following notation, depicted in Fig.~\ref{fig:notation} (left). 
\begin{enumerate}
\item $T = B'(\theta)$ the tangent vector to the boundary curve.
\bigskip 

\item $\gamma$ is defined implicitly by $r = \cos{\gamma}$. We therefore define it directly by $\gamma = \arccos{r}$. 
\bigskip

\item $\tau$ is the angle between the positive horizontal direction and $T$. Thus $\tau$ is defined implicitly by $T_2/T_1 = \tan \tau$.
To find $\tau$, we use the quadrant-specific version of the arctan function ({\tt atan2} in MATLAB) and set $\tau = {\tt atan2}(T)$.
\bigskip

\item $\mathbf{v}$ is defined implicitly by saying that $\gamma$ is the angle between $\mathbf{v}$ and $T$.  From the previous two definitions, we get 
\[\mathbf{v} = \begin{pmatrix}
    \cos{( \gamma + \tau)}\\ \sin{( \gamma + \tau)}
\end{pmatrix}.\]
\end{enumerate}

Using the notation defined above, we are ready to find $\hat{\theta}$ using the fact that $B(\theta)$ and $B(\hat{\theta})$ lie on a line parallel to $\mathbf{v}$, given by 
\begin{equation}
 L(s) = B(\theta) + s \mathbf{v } ,
\end{equation}
where $s \in \R$ is a scalar.  See right panel of Figure~\ref{fig:notation}. Since the table is convex, this line has exactly two intersection points with the boundary, the first being the original point  $B(\theta)$ at $s=0$. The second point $B( \hat{\theta})$ occurs at $s>0$. As long as we have a good guess near the desired solution, this intersection can be formulated as a zero finding problem
\begin{equation}\label{intersectbill}
\mathbf{0} = h(s,\hat{\theta}) = L(s) - B(\hat{\theta}),
\end{equation}
where the function $h:\R^2 \to \R^2$. Since $L$ and $B$ are both smooth, we solve using Newton's method. \correction{B-R2-1}{It seems counterintuitive to solve for $\hat{\theta}$ and $s$ rather than solving only for $\hat{\theta}$. However, a one-variable root finding problem for $\hat{\theta}$ alone has higher-multiplicity zeroes in some important cases (specifically, whenever $\bf{v}$ is orthogonal to $T$).}

We now explain how to find a good initial guess  $(s_{\rm guess},\hat{\theta}_{\rm guess})$ for our Newton's method to find a root for $h$. 
We saw in the last section that if $f_e$ is the billiard map for an ellipse, then the iterate $f_e(\theta_e,r_e) = (\hat{\theta}_e,\hat{r}_e)$ can be computed in an explicit way, and this calculation also yields $s_*$. Therefore, given  the billiard map $f_e$ for the ellipse associated with our table, we compute $f_e(\theta,r)=(\hat{\theta}_e,\hat{r}_e)$. We use 
$s_{\rm guess} = s_*$, the $s$ value for the ellipse given in~\eqref{eq:ellipsesstar}  and $\theta_{\rm guess} = \hat{\theta}_e$ for our guess for Newton's method. This is  usually a good enough guess to allow us to find  $(\hat{\theta},\hat{r})$ for the perturbed table. 

Though the method above works the vast majority of the time, it is not enough to guarantee convergence to the correct root in some small number of cases. We make two modifications. The first modification takes into account that 
 when the ball only bounces a short distance, or more specifically $|r|$ is close to 1, then  Newton's method is likely to find the spurious root that occurs at $s=0$. In order to avoid this problem, and make sure that we converge to the correct root in these cases, we use the method of \textit{deflation} on the function $h$ to eliminate the incorrect zero to which we do not want Newton's method to converge.  That is, since we know that $h$ has a root at $(s,\hat{\theta}) = (0,\theta)$, we modify $h$ as follows
\[
\tilde{h}(s,\hat{\theta}) = \frac{h(s,\hat{\theta})}{\|(s-0,\hat{\theta}-\theta)\|}.
\]
This has the same second root as $h$ but avoids the spurious solution. 
For a more complete discussion on deflation see \cite{Brow_Gearhart_1971a, PETERS_WILKINSON_1971, Beentjes_2015, Farrell_Birkisson_Funke_2015a}. 

The second modification is a technical modification in making sure that the initial guess is well defined. It corrects for the fact that $\mathbf{v}$ for the perturbed table might point outside the associated ellipse table. See the discussion and Figure~\ref{fig:modbounce} in  Appendix A.

After we have found $\hat{\theta}$, the last step is to find $\hat{r}$. 
 Using similar notation as before, we use the notation below.     See Fig.~\ref{fig:reflectionangle}.
\begin{enumerate}
    \item $\hat{T} = B'(\hat{\theta})$ is the tangent vector.
    \item $\hat{\tau}$ is the angle between the tangent line and the positive horizontal axis. 
    \item $\hat{\rho}$ is the angle between  $\mathbf{v}$ and the positive horizontal axis. 
    \item $\hat{\gamma} = \hat{\rho}-\hat{\tau}$. This is the angle between $\mathbf{v}$ and $\hat{T}$, which is up to sign the same as the angle between  $\hat{T}$ and the reflected vector $\hat{\mathbf{v}}$ . \correction{B-R2-7}{(The value of $\hat{r}$ depends only on $\hat{\gamma}$ and not on  $\hat{\mathbf{v}}$.)}
    \item $\hat{r} =\cos{\hat{\gamma}}$ . 
\end{enumerate}

We have found this method to work without failure for small perturbations of ellipses. In particular, since we initialize Newton's method with $(s_{\rm guess},\hat{\theta}_{\rm guess})= (s_*, \hat{\theta}_e)$, the boundary of a too greatly perturbed table may be insufficiently close to the guess to get convergence. In order to find an upper bound on the perturbation which results in Newton's method converging within a reasonable number of iterations, we studied the table with coefficients given in ~\eqref{eq:convtable}.
This table was previously found to have lost convexity around $\epsilon \approx 9$ (see Figure ~\ref{fig:convex}), but Newton's method was found to converge reliably only for $\epsilon \in [0,5]$. Convergence of Newton's method was studied in depth for $\epsilon \in [5,6.5]$. 
We calculated orbits for an initial location of $(\theta_0, r_0)$ with $\theta_0 = 0.3$ and 100 values of $r_0$ ranging from $-0.9$ to $0.9$. These initial conditions were tested on the table given by ~\eqref{eq:convtable} for each value of $\epsilon$, where 23 values of $\epsilon$ were chosen, ranging from $\epsilon=5$ to $\epsilon=6.5$. We call Newton's method non-convergent if for any bounce in the first 100 iterations of the orbit, Newton's method did not converge. Any $\epsilon$ which resulted in a non-convergent Newton's method for any initial condition was said to be too large of a perturbation. We found that for $\epsilon \in [5.0,5.778]$, the table was sufficiently close to elliptical for our method to converge, while for $\epsilon \geq 5.8434$, the table was non-convergent often enough to make the method impractical without further modifications. This limitation is not an essential barrier for the method. Instead of starting at an initial guess at an unperturbed ellipse, we could repeat the process of guesses with a series of progressively larger perturbations until we get to a desired large perturbation. 


\subsection{Continuation -- analytic billiard maps}\label{sec:analytic}

We extend the billiard map analytically to a complex domain 
via the following steps.
\begin{itemize}
\item[(1)] Implicitly define $\mathbf{v}$ as a function of $(\theta,r)$: We rewrite our 
statements about $\tau, \gamma, \theta, T,  r,$ and $\mathbf{v}$ implicitly as a root finding problem. 
\[
\begin{array}{c}
g(\tau,\gamma,\theta,r)= 
\left( \begin{array}{cc}
\cos \gamma  - r  \\[2ex]
T_1 \sin \tau - T_2 \cos \tau
\end{array}
\right)  = \left(\begin{array}{cc}0\\0\end{array}\right),
\mathbf{v} = \left(\begin{array}{c} \cos(\gamma + \tau)\\ \sin(\gamma + \tau) \end{array}\right)
\end{array}
\]
Note that $g$ is perfectly well defined for complex values 
of $\tau, \gamma, \theta$, and $r$.  Indeed, $g$ is complex analytic in these
variables.

By implicit differentiation
\[ \frac{\partial \mathbf{v}}{\partial (\theta,r)}  =  \frac{\partial \mathbf{v}}{\partial(\gamma,\tau)}  \frac{\partial (\gamma, \tau)}{\partial (\theta,r)}\]
so that 
\[ \frac{\partial (\gamma, \tau)}{\partial (\theta,r)}  =- \left( \frac{\partial g} {\partial (\gamma,\tau )} \right) ^{-1} \frac{\partial g }{\partial (\theta,r)}\]
\[
\frac{\partial \mathbf{v}}{\partial \gamma} = \frac{\partial \mathbf{v}}{\partial \tau} 
= \left( \begin{array}{r}
    -\sin(\gamma + \tau)\\ \cos(\gamma + \tau)
\end{array} \right)
\]
\begin{eqnarray*}
\frac{\partial g}{\partial (\theta,r)}= 
\left( \begin{array}{cc}
0 \hspace{0.1in} & -1  \\[2ex]
\displaystyle \frac{\partial{T_1}}{\partial \theta} \sin \tau - \frac{\partial{T_2}}{\partial{\theta}} \cos \tau   & 0 
\end{array}
\right)    \\[2ex]  
\frac{\partial g }{\partial (\gamma,\tau)}= 
\left( \begin{array}{cc}
-\sin \gamma & 0  \\[2ex]
0  & T_1  \cos \tau + T_2 \sin\tau 
\end{array}
\right)  
\end{eqnarray*}
The term $\partial g/\partial{(\gamma,\tau)}$ 
is non-singular under the following condition, which holds for any 
well defined billiard iterate. 
\begin{itemize}
\item[(a)]  $\sin \gamma \ne 0$ meaning that $\gamma \ne 0, \pi$. These values correspond to the cases where the direction of 
bounce is tangent to the boundary -- which is not a possibility for any convex billiard table. 
\item[(b)] $T_1 \cos \tau + T_2 \sin \tau \ne 0$.  If this was equal to zero and $g = \mathbf{0}$, then we can show that $T_1 = T_2 =0$. 
\end{itemize}

\item[(2)] Implicitly define $(s,\hat{\theta)}$ as a function of $(\theta,r)$: We quickly restate the function $h$ so that we can write down  the derivative. 
\[h(\theta,r,s,\hat{\theta}) =  B(\theta) + s \mathbf{v} - B(\hat{\theta}) = \mathbf{0}\]
To compute the derivative, define $\hat{T} = B'(\hat{\theta})$ and  $T = B'(\theta)$. 
Using implicit differentiation, we get 
\[
\frac{\partial (s,\hat{\theta})}{\partial (\theta,r)} = 
- \left( \begin{array}{ccc}
\mathbf{v} &| & -\hat{T}
\end{array}
\right)^{-1}
\left( \begin{array}{ccc}
\displaystyle T + s \frac{\partial \mathbf{v}}{\partial \theta} & | & \displaystyle s \frac{\partial \mathbf{v}}{\partial r}  
\end{array}\right)
\]
For the part in the inverse to be singular, we would need $\hat{T}_2/\hat{T}_1 = v_2/v_1$. This only would occur if $\mathbf{v}$ and $\hat{T}$ were parallel, which cannot happen for a convex billiard table. 

\item[(3)] Implicitly define $\hat{r}$ as a function of $(\theta,r)$:   
This is formulated implicitly as follows. 
\[
k(\theta,r,\hat{\rho},\hat{\tau}) =   
\left( \begin{array}{cc}
\hat{T}_1 \sin\hat{\tau} - \hat{T}_2 \cos \hat{\tau} \\[2ex]
v_1 \sin \hat{\rho} - v_2 \cos \hat{\rho}
\end{array}
\right)  = \left(\begin{array}{cc}0\\0\end{array}\right), \mbox{ where }
\hat{r} = \cos(\hat{\rho}-\hat{\tau}).\]
Note that this is written in terms of $\mathbf{v}$ and not $\hat{\mathbf{v}}$. 
Using implicit differentiation 
\[
\frac{\partial \hat{r}}{\partial (\theta,r)} = \frac{\partial \hat{r}}{\partial (\hat{\rho},\hat{\tau})} \frac{\partial (\hat{\rho},\hat{\tau})}{\partial (\theta,r)}
\]
and 
\[
\frac{\partial (\hat{\rho},\hat{\tau})}{\partial (\theta,r)} = 
- \left(\frac{\partial k}{\partial (\hat{\rho},\hat{\tau})} \right)^{-1} 
\frac{\partial k}{\partial (\theta,r)}. 
\]
These partial derivatives are given by  
\[
\frac{\partial k}{\partial (\hat{\rho},\hat{\tau})} = \left( \begin{array}{cc} 
0  & \hat{T}_1 \cos \hat{\tau} +  \hat{T}_2 \sin\hat{\tau} \\[2ex]
v_1 \cos \hat{\rho} + v_2 \sin \hat{\rho}  & 0 
\end{array}
\right),
\]
and 
\[
\frac{\partial k}{\partial (\theta,r)} = \left( \begin{array}{cc} 
\displaystyle
\frac{\partial \hat{T}_1}{\partial \theta} \sin \hat{\tau} - \frac{\partial \hat{T}_2}{\partial \theta} \cos \hat{\tau}  \hspace{0.1in}& 
\displaystyle \frac{\partial \hat{T}_1}{\partial r} \sin \hat{\tau} - \frac{\partial \hat{T}_2}{\partial r} \cos \hat{\tau} \\[2ex]
\displaystyle \frac{\partial v_1}{\partial \theta} \sin \hat{\rho} - \frac{\partial v_2}{\partial \theta} \cos \hat{\rho}  \hspace{0.1in} & 
\displaystyle \frac{\partial v_1}{\partial r} \sin \hat{\rho} - \frac{\partial v_2}{\partial r} \cos \hat{\rho} 
\end{array}
\right), \mbox{ where } r = \cos{\gamma}.
\]
\item[(4)] It only  remains to show that $\partial k/\partial(\hat{\rho},\hat{\tau})$ is non-singular when $k = 0$. 
If the determinant of the Jacobian matrix is zero, then we can show that either $\mathbf{v}$ or $\hat{T}$ is the zero vector. By our assumptions on $\mathbf{v}$ and $B$, that cannot happen.
\end{itemize}

Unlike in the real-valued case, it is not feasible to use an ellipse as an initial guess. In fact, we can continue the ellipse into the complex plane along different paths to get different answers. Therefore to find the iterate of $(\theta,r)$ we require an initial guess $(\hat{\theta}_{\rm guess},\hat{r}_{\rm guess})$, as well as guesses approximating the complex-valued angles $\tau, \gamma, \hat{\tau}, \hat{\gamma}$. These guesses come from the exact iterate values for a point $(\theta_p,r_p)$ which is close to $(\theta,r)$,
so the analytic continuation is performed using numerical continuation.


Note that the continuation version of the map can be continued to the case of $\hat{\theta}$  outside of $[0,1)$, whereas the way that the real version was designed means that it will never be possible to find an iterate outside this range.

\section{Parameterization Method for Hyperbolic Periodic Orbits of Planar Maps} \label{sec:planarMaps}

The idea is to formulate a family of higher dimensional maps, whose
fixed points correspond to periodic orbits of the billiard map.  In addition to providing 
a stable method for numerically computing periodic orbits, the multiple shooting 
maps can be analyzed to recover linear and nonlinear stability information
about the periodic orbit.  These ideas are reviewed in the next several 
sections.

\subsection{Multiple shooting}
\label{sec:multishoot} 

We focus on the special case of planar maps, 
as this is the case used in the present work, and this restriction 
simplifies parts of the discussion. 
The material in this section and in Section~\ref{sec:stability} is standard, 
but it is essential for the remainder of the present work.
We restate a number of well known results without proof in 
this context.  We refer to
\cite{MR2656693,MR3706909,MR4667735}
for more complete discussion, including many additional references.

Let $S \subset \mathbb{R}^2$ be an open and 
connected set and suppose that 
$f: S \to \mathbb{R}^2$ is a smooth map. 
For $k \in \mathbb{N}$, let $f^k$ denote the composition 
of $f$ with itself $k$ times, and let $f^0 = \mbox{Id}_{\mathbb{R}^2}$ 
denote the identity map.
We say that $\hat u \in S$ is a period-$K$ point for $f$
if 
\[
f^K(\hat u) = \hat u.
\]
We say that $\hat u$ has least period $K$ 
if $f^k(\hat u) \neq \hat u$ for $0 < k < K$.
Letting $u_1 = \hat u$, 
and $u_j = f(u_{j-1})$ for $1 < j \leq K$,  
we say that the ordered collection of points 
$\{ u_1, \ldots, u_K \}$, each in  $S$, is a periodic orbit 
for $f$, noting that $f(u_K) = u_1$.

Fix $K$, let $S^K \subset \mathbb{R}^{2K}$ denote the $K$-fold
cartesian product of $S$. Define the multiple shooting map 
$F_K: S^K \to \mathbb{R}^{2K}$ by 
\begin{equation} \label{eq:multShoot}
F_K(u_1, u_2, \ldots, u_{K-1}, u_K) = 
(f(u_K),f(u_1),f(u_2),\dots,f(u_{K-1})).
\end{equation}
Note that $\bfu = (u_1, \ldots, u_K) \in S^K$ 
satisfies
$F_K(\bfu) = \bfu$, if and only if $\{ u_1, \ldots, u_K \}$ is a period-$K$ orbit 
of $f$. That is, fixed points of $F_K$ correspond to period-$K$ orbits of $f$.
Moreover, $K$ is the least period 
if and only if $u_1 \neq \ldots \neq u_K$.  We
suppress the subscript and write $F_K = F$ when 
$K$ is clear from context.

The Jacobian matrix of $F$ is the following $2K \times 2K$ matrix
\begin{equation}
\DF(\bfu) = \begin{pmatrix}
0_{2 \times 2} & 0_{2 \times 2} & \cdots & 0_{2 \times 2} & \Df(u_K) \\
\Df(u_1) & 0_{2 \times 2} & \cdots & 0_{2 \times 2} & 0_{2 \times 2} \\
\vdots & \vdots & \ddots & \vdots & \vdots \\
0_{2 \times 2} & 0_{2 \times 2} & \cdots & \Df(u_{K-1}) & 0_{2 \times 2}
\end{pmatrix},\nonumber
\end{equation}
where $0_{2 \times 2}$ denotes the $2 \times 2$ zero matrix and $\Df(u_i)$ is the $2 \times 2$ Jacobian of $f$ evaluated at $u_i$.
The virtue of the multiple-shooting framework is that it dispenses with compositions.

Fixed points of $F$ are equivalent to zeros of the map 
$G: S^K \to \R^{2K}$ defined by 
\begin{equation}
    G(\bfu) = F(\bfu) - \bfu,
    \nonumber
\end{equation}
and zeros of $G$ are found computed using 
Newton's method.  That is,
if $\bfu^0 \in S^K$ has 
$\|G(\bfu^0)\|$ small enough, then the Newton
sequence $\left\{\bfu^k\right\}$
defined by 
\begin{equation}
    \bfu^{k+1} = \bfu^{k} + {\bf \Delta}^{k}, 
    \nonumber
\end{equation}
where ${\bf \Delta}^{k} \in S^K$ solves the linear equation 
\begin{equation}
DG(\bfu^{k}) {\bf \Delta}^{k} = - G(\bfu^{k}), \nonumber
\end{equation}
converges to $\bfu_*$ with $G(\bfu_*) = 0$.
The components of $\bfu_*$ provide a periodic-$K$ orbit for $f$.
Here we have the explicit formula 
\begin{equation}
    DG(\bfu) = DF(\bfu) - \Id_{2K \times 2K},
    \nonumber
\end{equation}
where $\Id_{2K \times 2K}$ denotes the $2K \times 2K$ identity matrix.

\subsection{Stability for periodic orbits}\label{sec:stability}

The stability of a periodic orbit is defined by 
applying the usual notion of fixed point stability
to the composition map $f^K$.
Moreover, the eigenvalues of the 
matrices $Df^K(u_j)$ do not depend on $1 \leq j \leq K$. 
That is, each of these matrices has the same eigenvalues 
$\lambda_1, \lambda_2 \in \mathbb{C}$, and we refer to these
as the multipliers of the periodic orbit.
Note, however, that the associated eigenvectors are 
in general different for different $j$.  
The multipliers $\lambda_1,\lambda_2$ can be recovered from 
the Jacobian of the multiple shooting map.  
Indeed, by Proposition 1 of \cite{MR2656693}, Lemma 3.1 of 
\cite{MR3706909} or Proposition 3.1 of \cite{MR4667735}, 
we get the following lemma.

\begin{lemma}[Multipliers via the multiple shooting eigenvalues]
\label{lemmq:floquet}
Let $\bfu_* = (u_1, \ldots, u_K)$ be a fixed point of the multiple shooting 
operator $F$ associated with the map $f$, so that each $u_j$ is a period-$K$ under $f$.
\begin{itemize}
\item The eigenvalues of $DF(\bfu_*)$ are  
the $K$-th roots of the multipliers of the periodic orbit.
More precisely, if
$(\alpha, \xi)$ is an eigenvalue-eigenvector pair for $DF(\bfu_*)$
with $\xi = (\xi_1, \ldots, \xi_K)$,
then \correction{B-R2-14}{$(\alpha^K, \xi_j)$} is an eigenpair for 
$Df^K(u_j)$, $1 \leq j \leq K$. 
\item If $\{u_1, \ldots, u_K\}$ is a period-$K$ saddle orbit, 
then the multipliers $|\lambda_1| < 1 < |\lambda_2|$ are real, since 
the matrices $Df^K(u_j)$ are real, and each 
has a real eigenbasis.  
\item If $(\alpha, \xi)$ is an eigenpair with $\alpha$ and 
$\xi$ complex then, since $DF(\bfu_*)$ is a real matrix, we have that 
$(\alpha, \mbox{Real}(\xi))$ and $(\alpha, \mbox{imag}(\xi))$
are both eigenpairs for $DF(\bfu_*)$. 
\end{itemize}
\end{lemma}

In the present work, the map $f$ is orientation preserving,
and the orbits we consider have two positive real multipliers
(we did not find orbits with two negative multipliers, but 
we have not ruled them out).
When there are two positive real multipliers,
there is always a $K$-th root which is real, and 
this is the root we will always choose to work with.
This leads to real eigenvectors.
Note however that, more generally,
if $Df^K(u_j)$ has a negative multiplier, 
and if $K$ is even, then there will be no real $K$-th root.
This, and more general cases involving higher dimensional manifolds
and maps, are described in \cite{MR3706909}.

\subsection{Identifying orbit types -- frequency and chaos}\label{sec:wba}

If $f$ is an area preserving map of the plane, then typical orbits of 
$f$ are either periodic, quasiperiodic, or 
chaotic, and the method of weighted Birkhoff averages can be used to distinguish
between these, as well as computing the frequency of quasiperiodic orbits. We briefly discuss this method, 
and refer to ~\cite{quanquasi_2017,das_yorke_2018, das_solving_2019,  sander_birkhoff_2020, Meiss_Sander_2021} for more 
detailed discussion.

Consider an orbit of $f$ $\left\{(\theta_k,r_k)\right\}_{k = 0}^\infty$, where $(\theta_{k+1},r_{k+1}) = f(\theta_k,r_k)$. 
The frequency (or rotation number) of the orbit is defined to be the average amount that the 
$\theta$ variable changes in a single iterate.
More precisely, define 
\[
{\rm frequency} =  \lim_{N \to \infty} \frac{1}{N} \sum_{k=0}^{N-1} (\theta_{k+1} - \theta_k).
\]
\correction{B-R1-11}{Note that this is a slight abuse of notation, as the frequency is actually defined on a lift of $f$ to $\tilde{f}:\R \times (-1,1) \to \R \times (-1,1)$.}
Under standard smoothness and ergodicity hypotheses, the limit of this average
is an invariant of the orbit and does not
depend on the initial condition/initial position along the orbit.
In practice however this sum converges very slowly, with complexity $O(1/N)$.

The weighted Birkhoff average is a tool which
accelerates the convergence, but only along nonchaotic orbits
with \correction{B-R2-16}{Diophantine} frequency.
This results in spectral convergence, faster than $O(1/N^k)$ 
for all $k \in \N$. The weighted Birkhoff average still converges 
slowly for chaotic orbits, and because of this it 
can be used as a sieve to distinguish chaotic from quasiperiodic orbits.

In order to distinguish chaotic orbits, given a length-$N$ orbit segment, we compute a weighted Birkhoff average on iterates $0,\dots,N/2$, and compare the answer 
to the same weighted Birkhoff average for iterates $N/2,\dots,N$ 
(in practice it suffices to take $N = 1000$). 
Since a Birkhoff average depends only on the orbit, if the difference in the two answers is small, this is an indication that the average is converging rapidly, and we can
classify the orbit as nonchaotic; if the difference is large, the orbit 
is classified as chaotic. 
Applying this method to a large number of randomly selected points allows us to identify regions of chaos for billiard maps. These regions are colored gray in all phase plane figures.  

Additionally, for nonchaotic orbits, the weighted-Birkhoff average allows us to give a highly accurate computation of the orbit frequency. This method was used to determine the color variation for all phase plane figures. 

\subsection{Finding periodic orbits} \label{rem:shortPeriod}

In searching for periodic orbits,  our goal is to find 
ones that possess transverse homoclinic \correction{B-R1-12}{points}. 
Therefore,  we only seek
periodic orbits within regions identified as chaotic. 

Initial guesses for the multiple shooting Newton method
for periodic orbits can be guessed by looking 
near elliptical islands.   Indeed, since  periodic orbits
typically appear in hyperbolic and elliptic pairs, and since the elliptic orbits are easy to see from the coloring of the  phase space plot, we can often find initial data for hyperbolic 
periodic orbits by looking ``between''
the elliptic islands.  Using this geometric intuition, it is possible for orbits of small period, such as say any period up to 5, to simply read good enough initial guesses directly off
a phase plane plot colored by frequencies location of chaotic orbits, and to have these guesses converge under Newton's method without further effort. Finding periodic orbits of larger periods is a bit more ad hoc and involves several iterations of the same process, including zooming in on the colored phase plane figure to see more detail of the elliptical islands. With a bit of work we were able to find periodic orbits of periods up to 30. (We did not look for  orbits beyond period 30.)

\subsection{Parameterization method for one-dimensional spectral submanifolds}\label{sec:parameterization}

The material in this section is adapted (in a simplified form appropriate for 
the present discussion) from \cite{Cabre_Fontich_delaLlave_2003,
Cabre_Fontich_delaLlave_2005}. See also \cite{MR3706909}.

Suppose that $p_0 \in \mathbb{R}^d$ is a hyperbolic fixed point
for a real analytic map $F$.  In this case there exist $d_u, d_s \in \mathbb{N}$
with $d_u + d_s = d$ such that $DF(p_0)$ has $d_s$ stable eigenvalues 
$\lambda^s_1, \ldots, \lambda_{d_s}^s \in \mathbb{C}$, and $d_u$ unstable eigenvalues 
$\lambda_1^u, \ldots, \lambda_{d_u}^u \in \mathbb{C}$. 
 Assume that $DF(p_0)$ is diagonalizable, so that there are 
$\xi_1^s, \ldots, \xi_{d_s}^s \in \mathbb{R}^d$ associated stable,
and $\xi_1^u, \ldots, \xi_{d_u}^u \in \mathbb{R}^d$ unstable 
eigenvectors.  Since $DF(p_0)$ is diagonalizable, 
these are linearly independent and form a basis for $\mathbb{R}^n$.
We are interested in the existence of one-dimensional invariant submanifolds,
tangent at $p_0$ to one of these eigenvectors.
Such manifolds are referred to as one-dimensional spectral submanifolds.

One-dimensional submanifolds of the stable/unstable 
manifold are typically not unique 
(in fact occur in continua).  
We select a unique spectral submanifold 
by imposing maximal regularity.  The issue is that,
while a typical one-dimensional stable/unstable 
manifold tangent to an eigenvector is 
only $C^\infty$ at the fixed point $p_0$),
there is a unique analytic one.
The parameterization method is the right
tool in this context. It works by looking for 
an analytic conjugacy to the linear dynamics generated by 
the eigenvalue.  
The following lemma provides a single equation which 
determines a fixed point, its eigendata, and a 
spectral submanifold. A proof is found for example in \cite{MR3706909}.

\begin{lemma}[Parameterization Lemma] \label{lemma:parm}
Let $U \subset \mathbb{R}^d$ be an open, connected set, 
and $F \colon U \to \mathbb{R}^d$ be a 
real analytic map. 
Suppose that $\lambda \in \mathbb{R}$,  that 
$P \colon [-1,1] \to \mathbb{R}^d$ 
is real analytic on $(-1,1)$, that $P$ is continuous on $[-1,1]$, 
and that the pair $(\lambda, P)$ 
 satisfies the conjugacy equation 
\begin{equation} \label{eq:parm1}
F(P(\sigma)) = P(\lambda \sigma), \quad \quad \quad \sigma \in [-1,1].
\end{equation}
If $P'(0) \neq 0$
and $0<|\lambda|<1$, then:
\begin{itemize}
\item $P(0)$ is a fixed point of $F$, 
\item $(\lambda, P'(0))$ is an eigenpair for $DF(P(0))$,
\item $P([-1,1])$ is an analytic curve, tangent 
to the stable eigenvector associated with 
$\lambda$ at $\sigma = 0$, and 
contained in the local stable manifold of $P(0)$.
\end{itemize}
\end{lemma}

The Lemma  explains our interest in Equation
\eqref{eq:parm1}, but leaves open questions of existence and uniqueness.  
Here we appeal to Theorem 4.1 of \cite{Cabre_Fontich_delaLlave_2005},
which deals precisely with the one dimensional analytic 
situation at hand (in fact, the theorem
is a little more general than we need, as it allows for non-hyperbolicity,
complex valued maps, and complex eigendata).
More precisely, assume that the following non-resonance condition holds
\begin{equation} \label{eq:nonRes}
\lambda^n \notin \mbox{Spec}(DF(P(0))), \quad \quad \quad \quad n \geq 2.
\end{equation}
From this assumption, 
Theorem 4.1 of \cite{Cabre_Fontich_delaLlave_2005}
 concludes that Equation \eqref{eq:parm1}
has exactly one analytic solution for each choice of eigenvector 
$P'(0)$.  

The choice of the scaling of the eigenvector determines the 
radius of convergence of $P$, and by using the rescaling argument 
from the proof of Theorem 4.1, one can arrange that 
$P$ is analytic on a disk of radius $1 + \epsilon$ in $\mathbb{C}$.
If $F$ is real analytic then so is $P$, and this gives real analyticity on 
 $(-1,1)$ and continuity on $[-1,1]$ as desired.
We stress that the proof of Theorem 4.1
gives that the solution of Equation \eqref{eq:parm1} 
is globally unique as soon as the length of the eigenvector
$P'(0)$ is fixed.

We also note, as observed in Remark 4.2 of \cite{Cabre_Fontich_delaLlave_2005},
that the invertibility of $DF(P(0))$ implies that there is an $N \geq 2$ so that 
Equation \eqref{eq:nonRes} holds for all $n \geq N$.  So, despite first appearances, 
the non-resonance condition imposes only a finite number of constraints on 
the eigenvalues of $DF(P(0))$, and Theorem 4.1 applies generically.

The discussion above motivates the following.

\begin{lemma}[Functional equation, stable case]\label{def:funstable}
Fix a scaling factor $s > 0$ and define the operator 
$\Psi_s \colon \mathbb{R} \times C^\omega([-1,1], \mathbb{R}^d)
\to \mathbb{R} \times C^\omega([-1,1], \mathbb{R}^d)$ by 
\begin{equation} \label{eq:FE1}
\Psi_s(\lambda, P(\sigma)) = \begin{pmatrix}\|P'(0)\|^2 - s\\
F(P(\sigma)) - P(\lambda \sigma)\end{pmatrix}\nonumber
\end{equation}
If $(\lambda, P)$ is a zero of $\Psi_s$ with $0 < |\lambda| < 1$, 
then $P$ satisfies the hypotheses
of Lemma \ref{lemma:parm}.  If the non-resonance condition of 
Equation \eqref{eq:nonRes} is satisfied, then there exists an $\epsilon > 0$
so that for all $0 < s \leq \epsilon$, $\Psi_s$ 
has a globally unique solution.
\end{lemma}
Consideration of the unstable manifold leads to a similar zero 
finding problem with a corresponding but not necessarily equal scaling 
factor $u$.

The Fr{\'e}chet derivatives are written as follows, where $\lambda, \mu, h \in \mathbb{C}$ and 
$P, Q, H \in C^\omega(\mathbb{B}, \mathbb{R}^d)$. 
\begin{equation}\label{eq:derivsall}
\begin{array}{cccc}
D_{\lambda}\Psi_s(\lambda,P(\sigma))[h] &=& \begin{pmatrix}
0\\
-P'(\lambda \sigma) \sigma h
\end{pmatrix}, \\[3ex]
D_{P}\Psi_s(\lambda,P(\sigma))[H] &=& \begin{pmatrix}
2 \langle P'(0),H'(0) \rangle \\
DF(P(\sigma)) H(\sigma) - H(\lambda \sigma)
\end{pmatrix}, \\[2ex]
\end{array}
\end{equation}

\begin{remark}[Generalizations] \label{rem:genParm}
{\em
The parameterization method generalizes to the case when there are 
resonances.  However, one cannot then analytically conjugate
to linear dynamics.  Instead, one studies a conjugacy equation 
of the form 
\[
F(P(\sigma)) = P(K(\sigma)), 
\]
where $K$ is a polynomial map whose degree and 
monomial terms are determined by the 
resonances.  More than this, the parameterization method extends to 
higher dimensional spectral submanifolds, 
non-hyperbolic fixed points, $C^k$ regularity,
and even to other kinds of invariant objects, 
for both maps and vector fields
defined on Banach spaces.  Much more general treatment, 
and many additional 
references, are found in the paper
\cite{Cabre_Fontich_delaLlave_2003}, and also to the book 
\cite{HARO_2018}.  See also Appendix B of 
\cite{Cabre_Fontich_delaLlave_2005}
for a scholarly discussion of the historical 
development of these ideas.
}
\end{remark}

\subsection{Stable and unstable manifolds of a
planar hyperbolic periodic orbit} \label{sec:parmPlanarMaps}

As an application of the ideas presented in the previous section, we discuss
a parameterization method for stable/unstable manifolds of periodic orbits
of maps.  By restricting to saddle type orbits of planar maps, we 
obtain an especially complete picture.
Note, the material in this section is adapted from 
\cite{MR3706909}, 
where generalizations to higher 
dimensional maps and manifolds are also found.
We refer also to the paper \cite{timsina_parameterized_2022}
where extensions to implicitly defined maps  
given by polynomial relations are considered.

Let $S \subset \mathbb{R}^2$ be an open and connected set, and  
suppose that $f \colon S \to \mathbb{R}^2$ is a real analytic mapping. 
Let us also assume that $f$ is orientation preserving 
(as is the case for the mathematical billiard maps considered 
in the present work).
Let $F \colon S^K \to \mathbb{R}^{2K}$ be the multiple shooting map
for $f$ as defined by Equation \eqref{eq:multShoot}. 
Let $\Psi_s$ be the operator 
defined in Lemma \ref{def:funstable} where  $s > 0$ is fixed.  Suppose that  for a fixed value of $\lambda$ with
$0<|\lambda|<1$ and a fixed function $P \colon [-1, 1] \to \mathbb{R}^{2K}$ we have that
$\Psi_s(\lambda, P) = {\bf 0}$. 
Combining Lemma \ref{lemmq:floquet} from Section \ref{sec:multishoot}
with Lemma \ref{lemma:parm} from Section \ref{sec:parameterization}, 
we have the following:
\begin{itemize}
\item $P(0) \in \mathbb{R}^{2K}$ is a fixed point of the multiple shooting map $F$.
Letting $P(0) = (p_1, \ldots, p_K)$ with $p_1, \ldots, p_K \in \mathbb{R}^2$,
each $p_j$ has period $K$ for $f$. That is, the components of $P(0)$ provide a periodic 
orbit for $f$.  If $p_k \neq p_j$ when $j \neq k$, then $K$ is the least period.
\item $\lambda^K = \alpha$ is a stable multiplier for the periodic orbit.  Moreover, writing 
$P'(0) = (\xi_1, \ldots, \xi_K)$ with $\xi_1, \ldots, \xi_K \in \mathbb{R}^2$, we have that 
$\xi_j$ is a stable eigenvector for $Df^K(p_j)$.  
\item Let $P(\sigma) = (P_1(\sigma), \ldots, P_K(\sigma))$ with each $P_j$ taking values in 
$\mathbb{R}^2$.  Then $P_j(\sigma)$ parameterizes a subset of the stable manifold 
attached to $p_j$, and is tangent at zero to the $\xi_j$.
\end{itemize}

%
%

Since we assume that the orbit is hyperbolic,
and that the mapping $f$ is planar, real analytic, and area preserving,
we have that the orbit has two real multipliers 
$\alpha, \beta$ with $0 < |\alpha| < 1$ and $|\beta| > 1$, and 
moreover, $\alpha \beta = 1$. Then, they are either 
both positive, or both negative.  In the present work 
we consider only the case when both multipliers are 
positive.  In this case the eigenvalues of $DF(P(0))$ 
sort into two sets: the $K$ complex $K$-th roots of 
$\alpha$, and the $K$ complex $K$-th roots of $\beta$. 
So, every eigenvalue of 
$DF(P(0))$ has either $|\lambda| = \alpha^{1/K}$ or $|\lambda| = \beta^{1/K}$.
Since all the stable eigenvalues have the same (non-unit)
modulus, no power of one can equal another, and the 
non-resonance conditions 
\[
\lambda^n \notin \mbox{Spec}DF(P(0)), 
\]
hold for all $n \geq 2$, and similarly for the unstable eigenvalues. 

So, the non-resonance conditions are automatically satisfied 
and the existence of the desired parameterization is assured.  
Moreover, since $\alpha$ and $\beta$ are positive, each has a 
real $K$-th root, and associated real eigenvectors.
Then the operator $\Psi_s$ 
defined in Lemma \ref{def:funstable} 
has unique, real solutions as long as 
$s$ remains in the range $0 < s < \epsilon$ for some $\epsilon>0$.  
In practice it may not necessarily be the case that 
$\epsilon$ is small, and numerical heuristics are discussed 
below.

\subsection{Numerical computation of the local parameterizations}
\label{sec:numerics}
In this section we develop a numerical Newton scheme for the parameterization 
method which exploits the DFT to compute function compositions.
A similar approach is used in \cite{MR4848687,MR4361879,MR4209693}
to solve functional equations involving implicitly defined 
Poincar\'e maps near periodic solutions in celestial mechanics applications. 
We review the procedure for the sake of completeness.

Let 
\begin{equation}\label{eq:polyorderN}
P(\sigma) = \sum_{n=0}^\infty p_n \sigma^n = 
\sum_{n=0}^\infty
\left(
\begin{array}{c}
p_n^1 \\
\vdots \\
p_n^d
\end{array}
\right)
\sigma^n, 
\end{equation}
denote a formal power series 
with $p_n \in \mathbb{R}^{d}$.
Now choose $N \in \mathbb{N}$ and write 
\begin{equation}\label{eq:PN}
P^N(\sigma) = \sum_{n=0}^N p_n \sigma^n,
\end{equation}
to denote the truncation of $P$ to order $N$.
We then embed the coefficients of 
$P^N$ in $\mathbb{R}^{d(N+1)}$ by 
forming the column vector (written as a row to save space) 
\begin{equation} \label{eq:bold_p}
\mathbf{p} = (p_0^1,\dots,p_N^1, p_0^2,\dots,p_N^2, \dots,p_0^d,\dots,p_N^d)
\in \mathbb{R}^{d(N+1)}.
\end{equation}
That is, for each $0 \leq n \leq N$, the 
numbers $p_n^i$ with $1 \leq i \leq d$ are
the components of $p_n \in \mathbb{R}^d$.

Now, given $\mathbf{p} \in \mathbb{R}^{d(N+1)}$, we write 
\[
\mathbf{q} = \mbox{DFT}[F](\mathbf{p}), 
\]
to signify that the column vector
\[
\mathbf{q} = (q_0^1,\dots,q_N^1, q_0^2,\dots,q_N^2, \dots,q_0^d,\dots,q_N^d)
\in \mathbb{R}^{d(N+1)},
\]
is the vectorization of the series coefficients of 
\[
Q^N(\sigma) = \sum_{n=0}^N q_n \sigma^n \approx F(P^N(\sigma)), 
\quad \quad |\sigma| < 1,
\]
where for each $1 \leq i \leq d$ we compute the series 
coefficients $\{q_0^i, \ldots, q_N^i\}$
by interpolating the components $F_i \circ P^N$, $1 \leq i \leq d$
of the composition via the DFT interpolation scheme discussed 
in Appendix \ref{sec:DFT}.

Similarly, for $1 \leq i, j \leq d$
define the coefficient sequences 
\[
\mathbf{b}^{ij}(\mathbf{p}) = 
(b_0^{ij},\dots,b_N^{ij}),
\]
which result from applying 
the DFT algorithm discussed in Appendix \ref{sec:DFT} to 
the composition maps $DF_{ij} \circ P$. That is, we let 
\[
\mathbf{b}^{ij}(\mathbf{p}) = \mbox{DFT}\left[\partial_j F_i\right](P^N),
\]
as these coefficients have 
\[
\sum_{n=0}^N b_n^{ij} \sigma^n \approx \partial_j F_i(P^N(\sigma)) 
\quad \quad |\sigma| < 1,
\]
where $F_i$, $1 \leq i \leq d$ denote the component maps of $F$.
We then define 
the $d^2$ many $(N+1) \times (N+1)$ matrices 
\begin{equation} \label{eq:def_Bij}
B^{ij}(\mathbf{p}) = \mbox{ConvMat}(\mathbf{b}^{ij}(\mathbf{p})),
\end{equation}
where the convolution matrix is as defined in Equation \eqref{eq:convMat}
of Appendix \ref{sec:multMat}.

To vectorize the composition $P(\lambda \sigma)$ and its derivatives
we first let $\Lambda, \Sigma$ denote the $(N+1) \times (N+1)$ diagonal matrices 
\[
\Lambda = \left(
\begin{array}{ccccc}
1 & 0 & 0 & \ldots & 0 \\
0 & \lambda & 0 & \ldots & 0 \\
0 & 0 & \lambda^2 & \ldots & 0 \\
\vdots & \vdots & \vdots & \ddots & \vdots \\
0 & 0 & 0 & \ldots & \lambda^N
\end{array}
\right), 
\quad \quad \mbox{and} \quad \quad 
\Sigma = \left(
\begin{array}{ccccc}
0 & 0 & 0 & \ldots & 0 \\
0 & 1 & 0 & \ldots & 0 \\
0 & 0 & 2 & \ldots & 0 \\
\vdots & \vdots & \vdots & \ddots & \vdots \\
0 & 0 & 0 & \ldots & N
\end{array}
\right),
\]
and define $\mathbf{L}_\lambda$
and $\mathbf{N}_N$, the $d(N+1) \times d(N+1)$ diagonal matrices 
\begin{equation*} 
\mathbf{L}_{\lambda} = \left(
\begin{array}{cccc}
\Lambda& \mathbf{0}_{\mathbb{R}^{N+1}} & \ldots & \mathbf{0}_{\mathbb{R}^{N+1}} \\
\mathbf{0}_{\mathbb{R}^{N+1}} & \Lambda & \ldots & \mathbf{0}_{\mathbb{R}^{N+1}}\\
\vdots & \vdots & \ddots & \vdots \\
\mathbf{0}_{\mathbb{R}^{N+1}} & \mathbf{0}_{\mathbb{R}^{N+1}}& \ldots  & \Lambda
\end{array}
\right),
\quad \quad \mbox{and} \quad \quad 
\mathbf{N}_{N} = \left(
\begin{array}{cccc}
\Sigma& \mathbf{0}_{\mathbb{R}^{N+1}} & \ldots & \mathbf{0}_{\mathbb{R}^{N+1}} \\
\mathbf{0}_{\mathbb{R}^{N+1}} & \Sigma & \ldots & \mathbf{0}_{\mathbb{R}^{N+1}}\\
\vdots & \vdots & \ddots & \vdots \\
\mathbf{0}_{\mathbb{R}^{N+1}} & \mathbf{0}_{\mathbb{R}^{N+1}}& \ldots  & \Sigma
\end{array}
\right),
\end{equation*}
and note that $\mathbf{L}_\lambda \mathbf{p}$ and $\mathbf{N}_N \mathbf{p}$
are, respectively, the vectorizations of the coefficients of 
$P^{N}(\lambda \sigma)$ and $\sigma (P^N)'(\sigma)$.

Truncating the domain and range of $\Psi_s$ at order $N$
and vectorizing the inputs and outputs leads to the map
$\mathbf{\Psi}_s^N \colon \mathbb{R}^{d(N+1)+1} \to \mathbb{R}^{d(N+1)+1}$
given by 
\begin{equation} \label{eq:PsiN}
    \mathbf{\Psi}_s^N(\lambda, \mathbf{p}) = \left(
        \begin{array}{c}
            (p_1^1)^2 + \ldots + (p_1^d)^2 - s \\[2ex]
             \mbox{DFT}[F](\mathbf{p}) - \mathbf{L}_\lambda \mathbf{p}
        \end{array}
    \right).
\end{equation}

To vectorize the derivative, we recall the formulas for the action 
of the Fr\'{e}chet derivatives given in Equation \eqref{eq:derivsall}.
First, for $1 \leq n \leq N$ let $e^*_n$ denote the $N+1$
dimensional row vector with a one in the $n$-th entry and 
zeros elsewhere. Keep in mind that we start with a zeroth element. 
Therefore the row vector $e_1^*$ is given by
\[
e^*_1 = (0, 1, 0, \ldots, 0).
\]
So, if 
\[
H^N(\theta) = \sum_{n=0}^N h_n \sigma^n, 
\]
with $h_n \in \mathbb{R}^d$, 
then the first component of the  partial derivative
$D_P \Psi_s(\lambda,P)[H]$
given in Equation \eqref{eq:derivsall}
is 
\begin{align*}
2 \left<(P^N)'(0), (H^N)'(0)\right> & = 
2 \left(p_1^1 h_1^1 + \ldots + p_1^d h_1^d \right) \\
& = \left(
\begin{array}{cccccccc}
2 p_1^1 e_1^* &|&
2 p_1^2 e_1^* &|& 
\ldots &|& 
2 p_1^d e_1^* \\
\end{array}
\right)
\left(
\begin{array}{c}
h_0^1 \\
\vdots \\
h_N^1 \\
\vdots \\
h_0^d \\
\vdots \\
h_N^d
\end{array}
\right).
\end{align*} 
Suppressing the action on $H^N$, the partial derivative
is represented simply by the $1 \times d (N+1)$ row vector 
\[
n(\mathbf{p}) = 
\left(
\begin{array}{cccccccc}
2 p_1^1 e_1^* &|&
2 p_1^2 e_1^* &|& 
\ldots &|& 
2 p_1^d e_1^*
\end{array}
\right).
\]

Similarly, the nontrivial component of the partial derivative $D_\lambda \Psi_s(\lambda,P)$ given in Equation \eqref{eq:derivsall} is the negative of $(P^N)'(\lambda \sigma) \sigma $, has $n$-th coefficient $\lambda^{n-1} n p_n$, and this has vector representation as 
the $d(N+1)$ dimensional column vector 
\[
\mathbf{c}_\lambda(\mathbf{p}) = \lambda^{-1}\mathbf{N}_N \mathbf{L}_{\lambda}(\mathbf{p}).
\]
For the second component of partial derivative $D_P \Psi_s(\lambda,P)[H]$, following Equation \eqref{eq:derivsall}, 
we vectorize $DF(P^N)$ as follows. 
Let $\mathbf{B}(\mathbf{p})$ denote the $d(N+1) \times d(N+1)$
matrix 
\[
\mathbf{B}(\mathbf{p}) = \left(
\begin{array}{ccc}
B^{00}(\mathbf{p}) &\ldots & B^{0d}(\mathbf{p}) \\
\vdots & \ddots & \vdots \\
B^{d0}(\mathbf{p}) & \ldots & B^{dd}(\mathbf{p})
\end{array}
\right),
\]
where for $1 \leq i, j \leq N$ the $B^{ij}(\mathbf{p})$ are the 
$(N+1) \times (N+1)$ convolution matrices 
defined in Equation \eqref{eq:def_Bij}.
So, the derivative of $\mathbf{\Psi}_s^N$ has matrix representation 
\[
D \mathbf{\Psi}_s^N(\mathbf{p}) = 
\left(
\begin{array}{cc}
0 & n(\mathbf{p}) \\[1.1ex]
-\mathbf{c}_\lambda(\mathbf{p}) & \mathbf{B}(\mathbf{p}) - \mathbf{L}_\lambda(\mathbf{p}) 
\end{array}
\right).
\]

%

Now consider the special case where 
$F = F_K$ is the multiple shooting 
map defined in Equation \eqref{eq:multShoot}, with  
underlying planar map $f$, so that $d = 2K$.
Let $\bar {\bf p}_0 \in \mathbb{R}^{2K}$ be an approximate 
fixed point of the multiple shooting map, in which case
the planar components of $\bar {\bf p}_0$ give an approximate period 
$K$ orbit for $f$.  Suppose that these components are distinct so that $K$ is 
the least period. 

Assume that $(\bar \lambda, \bar {\bf p}_1)$ is a (numerically 
computed approximate) real stable (or unstable) 
eigenpair for $DF(\bar {\bf p}_0)$. 
Fix $s$ (or $u$) a nonzero real number, and 
rescale $\bar {\bf p}_1$ so that $\| \bar {\bf p}_1 \|^2 = s$
(or $u$).  Let 
\begin{equation} \label{eq:P0}
P^0(\sigma) = \bar{{\bf p}}_0 + \sigma \bar{{\bf p}}_1.
\end{equation}
If the data $\bar \lambda, \bar{\bf p}_0, \bar{\bf p}_1$ were exact, then $P^0$ would be 
a quadratically good approximate solution of $\Psi_{s} = 0$.  That 
is, we would have 
\[
\|\Psi_s(P^0)\| \leq C s^2.
\]
So, assuming that $\bar{\bf p}_0$ and $\bar{\bf p}_1$ are good enough 
approximations, we
expect to be able to find an $s,u > 0$ so that this initial error is small
(limited only by the error in the initial data which is, presumably, 
on the order of machine epsilon).

Let $\mathfrak{p}^0$ be the vectorization of $P^0(\sigma)$,
and define the Newton sequence 
\[
\mathfrak{p}^{m+1} = \mathfrak{p}^m + \mathbf{\Delta}^m,
\]
where, for $m \geq 0$, $\mathbf{\Delta}^m$ solves the linear equation 
\[
D\mathbf{\Psi}_{s}^N(\mathfrak{p}^m) \mathbf{\Delta}_m = 
- \mathbf{\Psi}_{s}^N(\mathfrak{p}^m).
\]

In practice we experiment with the scalings $s$, as we 
would like to choose the largest such values so that
the Newton sequence converges.  More sophisticated continuation 
schemes and optimization algorithms for choosing 
computational parameters for the parameterization method 
are discussed in \cite{MR3437754,MR4217108}.

\begin{remark}[Automatic reducibility] \label{rem:reducibility}
{\em
Note that the Newton method just described is a ``large matrix method''.
More precisely, if $K$ is the period of the orbit, and $N$ is the polynomial 
order of approximation, then the linear system determining the Newton 
step has size $(2K(N+1) +1) \times (2K(N+1)+1)$.  While this presents
no problem in the present work, the complexity
can be greatly reduced by exploiting the symplectic structure of
the billiard system to define a diagonal approximate inverse for the 
Newton matrix.  This is referred to as ``approximate reducibility"
in the literature, and the interested reader will find
much more on this topic, including many references to its
use in the literature, in the book \cite{HARO_2018}.  
Using approximate reducibility would 
be helpful for studying orbits with much higher period than 
those studied here.}
\end{remark}

\begin{remark}[The method of jet transport] \label{rem:jetTransport}
{\em We note that an alternative to the DFT based interpolation scheme
described above would be to compute 
necessary compositions  
using the method of jet transport described in \cite{MR4904407}.
Indeed, the reference just cited uses jet transport to 
solve equations similar to the conjugacy equations considered 
in the present work, applying these ideas to Poincar{\'e} maps for ODEs in celestial mechanics.
These implicit Poincar{\'e} maps are, if anything, more complicated than 
the billiard maps studied in the present work, so we assume
that the method of jet transport could be applied to billiards 
as well.  The only disadvantage of jet transport  is that 
it requires special software like the TAYLOR
package developed by the authors of the reference just 
cited.  Our implementation relies only on standard
numerical linear algebra routines like linear system solvers,
eigenvalue/eigenvector solvers, and (if desired) the FFT.  
These algorithms are available in any standard numerical 
linear algebra package.  Our codes for example are implemented
in MATLAB.} 
\end{remark}

A useful a-posteriori error indicator is obtained by measuring the 
defect in the conjugacy relation.  So, define  
$E_{\mbox{conj}} \colon \mathbb{R} \times C^\omega(\mathbb{D}, \mathbb{R}^{2K})
\to \mathbb{R}$ by 
\begin{equation}
E_{\text{conj}}(\lambda, P) = \sup_{\sigma \in \mathbb{D}}
\left\|F(P(\sigma)) - P(\lambda \sigma)\right\|_{\mathbb{R}^{2K}}     \nonumber 
\end{equation}
Note that $E_{\mbox{conj}}(\bar \lambda,\bar P) = 0$ if and only if $(\bar \lambda, \bar P)$ 
is the desired solution.
A numerical proxy for the a-posteriori indicator is obtained by 
sampling the right hand side at
$M \in \mathbb{N}$ points on the boundary of $\mathbb{D}$,
and computing the maximum of this finite set.  More precisely, 
by the maximum modulus principle we have that \correction{B-R2-9}{ 
\begin{equation}
E_{\text{conj}}(\lambda, P) \approx
E_{\text{conj}}^M(\lambda, P) = \max_{0 \leq m \leq M}
\left\|F\left(P \left(e^{\frac{2 \pi i m}{M+1}}\right)\right)
- P\left(\lambda e^{\frac{2 \pi i m}{M+1}}\right)\right\|_{\mathbb{R}^{2K}},     \nonumber 
\end{equation}
}
bounds $E_{\text{conj}}(\lambda, P)$ from below, is computable in finitely many operations, and 
converges to the true a-posteriori indicator as $M\to \infty$.
Note also that, in the case that $P = P^N$ is polynomial,
it is efficiently evaluated at the desired roots of
using the FFT.




\subsection{Growing the local manifolds}\label{sec:globalmani}

To compute larger portions of the stable and unstable manifolds, we employ a fundamental domain approach that leverages the periodic structure of the system. Let $\mathcal{W}^s$ and $\mathcal{W}^u$ denote the global stable and unstable manifolds, respectively.

We work within a fundamental domain 
$\mathcal{D} \subset \mathbb{R}^2$, a bounded region 
such that every point on the global manifold can be 
obtained by applying forward or backward iterates of 
the map $F$ to points in $\mathcal{D}$.

Given the local parameterizations $P^s(\theta)$ and 
$P^u(\theta)$ of the stable and unstable manifolds, 
the global manifolds are constructed by first selecting 
parameterization points $\theta_j$ such that 
$P^{s,u}(\theta_j) \in \mathcal{D}$. Then for 
each selected point $x_j = P^{s,u}(\theta_j)$, 
where $j = 1,\dots,J$, compute:
\begin{align}\label{eq:globalmanifolds}
\mathcal{W}^u_{\text{global}} &\approx \bigcup_{j = 1}^J 
\bigcup_{n = 0}^M F^n(x_j^u)  \\
\mathcal{W}^s_{\text{global}} &\approx \bigcup_{j = 1}^J 
\bigcup_{n = 0}^M F^{-n}(x_j^s) \nonumber 
\end{align}
where $M$ is the maximum number of iterates. This approach 
ensures that the computed global manifolds cover the 
relevant phase space region while avoiding redundant 
computation outside the fundamental domain.


\section{Results}\label{sec:results}

In this section, we give our results. As mentioned in the introduction, the importance of our findings is twofold: 
the novelty of the use of the method, and also the implications for billiard maps. Correspondingly, we divide this section into two parts. 

\subsection{Results for the parameterization method}

%
We compute the stable and unstable manifolds for billiard maps using the 
methods described in Section \ref{sec:numerics}.
The two crucial sets of parameters are the truncation order of the 
polynomial approximation $N$, introduced in Equation \eqref{eq:PN} and 
the  eigenvector scaling. 
The scaling coefficients are  labeled $s$ in Lemma~\ref{def:funstable}, but in this 
section and in our Tables of numerical parameters, we refer to these parameters by
scale$_s$ and scale$_u$ to distinguish these cases. The resulting local manifolds 
are shown in the left panels of Figures~\ref{fig:manifolds_p10}--\ref{fig:manifolds_p2} 
for the period-10 orbit for Table~D, the period-3 orbit for Table~C, and the 
period-2 orbit for Table~B (cf. Tbl.~\ref{table:coeff}), respectively. 
The corresponding global manifolds
are obtained by iterating the local parameterizations under \( F \). The 
right panels of the same figures illustrate these global structures, 
together with the truncation order \( N \) of the polynomial 
approximation \( P^N(z) \) and the number of iterates used. 

\begin{figure}[tbph!]
    \centering
    \includegraphics[width=0.49\textwidth]{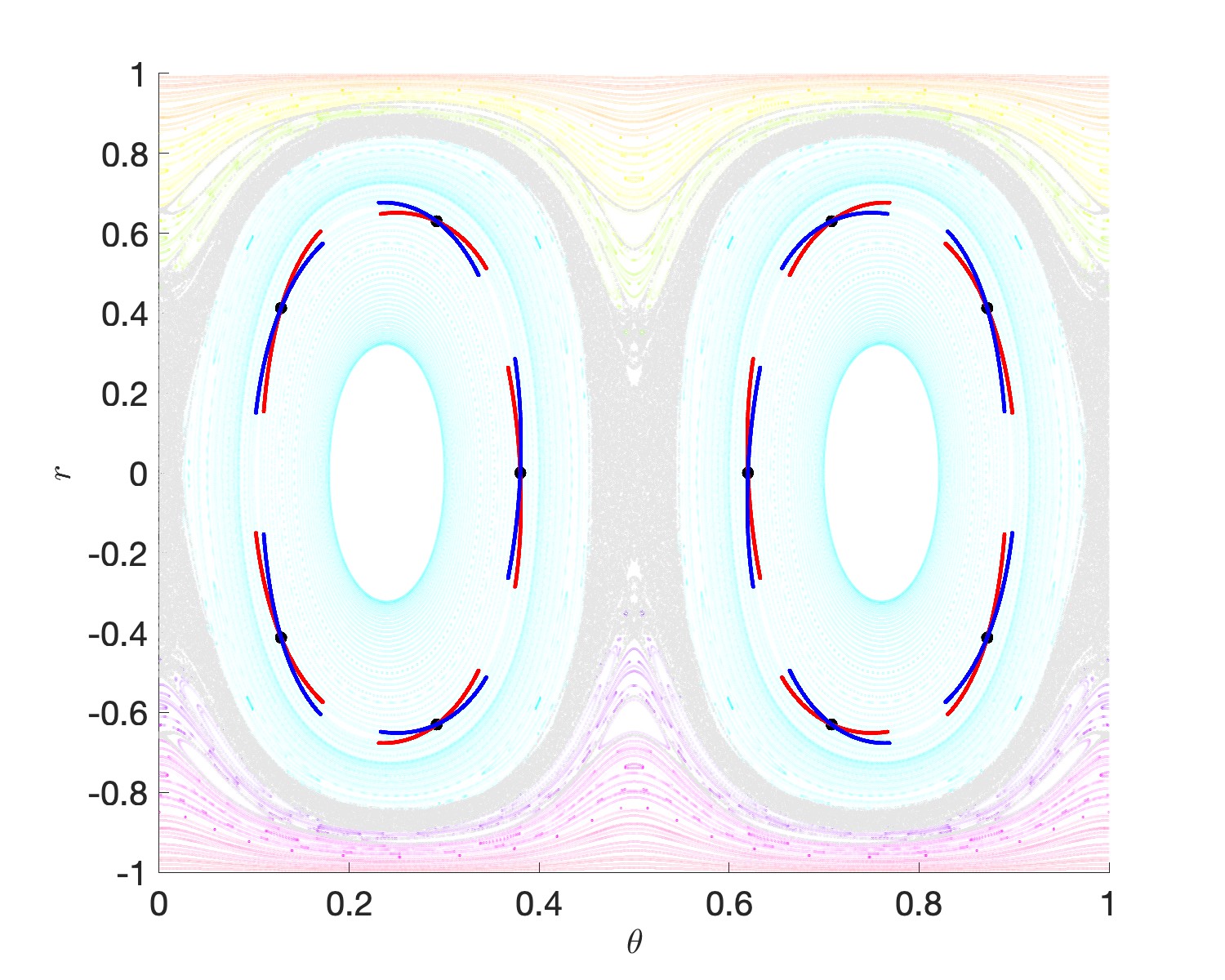}
    \hfill
    \includegraphics[width=0.49\textwidth]{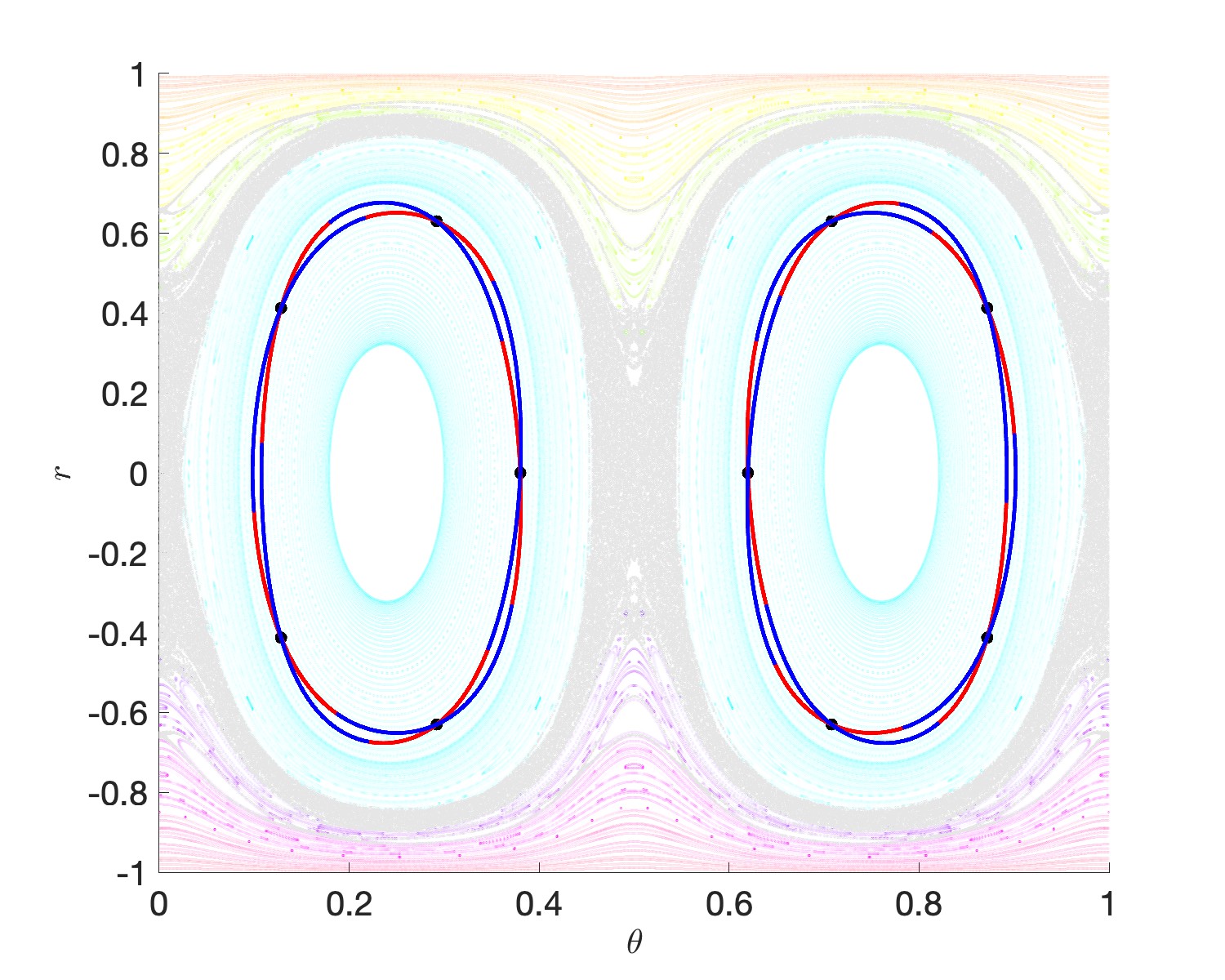}
    \caption{Local (left) and global (right) stable (blue) and unstable (red) manifolds for a period-10 orbit for Table D (cf. Tbl.~\ref{table:coeff}). The global manifolds were computed with 20 iterates of the fundamental domain and $N=60$.
    \label{fig:manifolds_p10} }
\end{figure}

\begin{figure}[tbph!]
    \centering
    \includegraphics[width=0.49\textwidth]{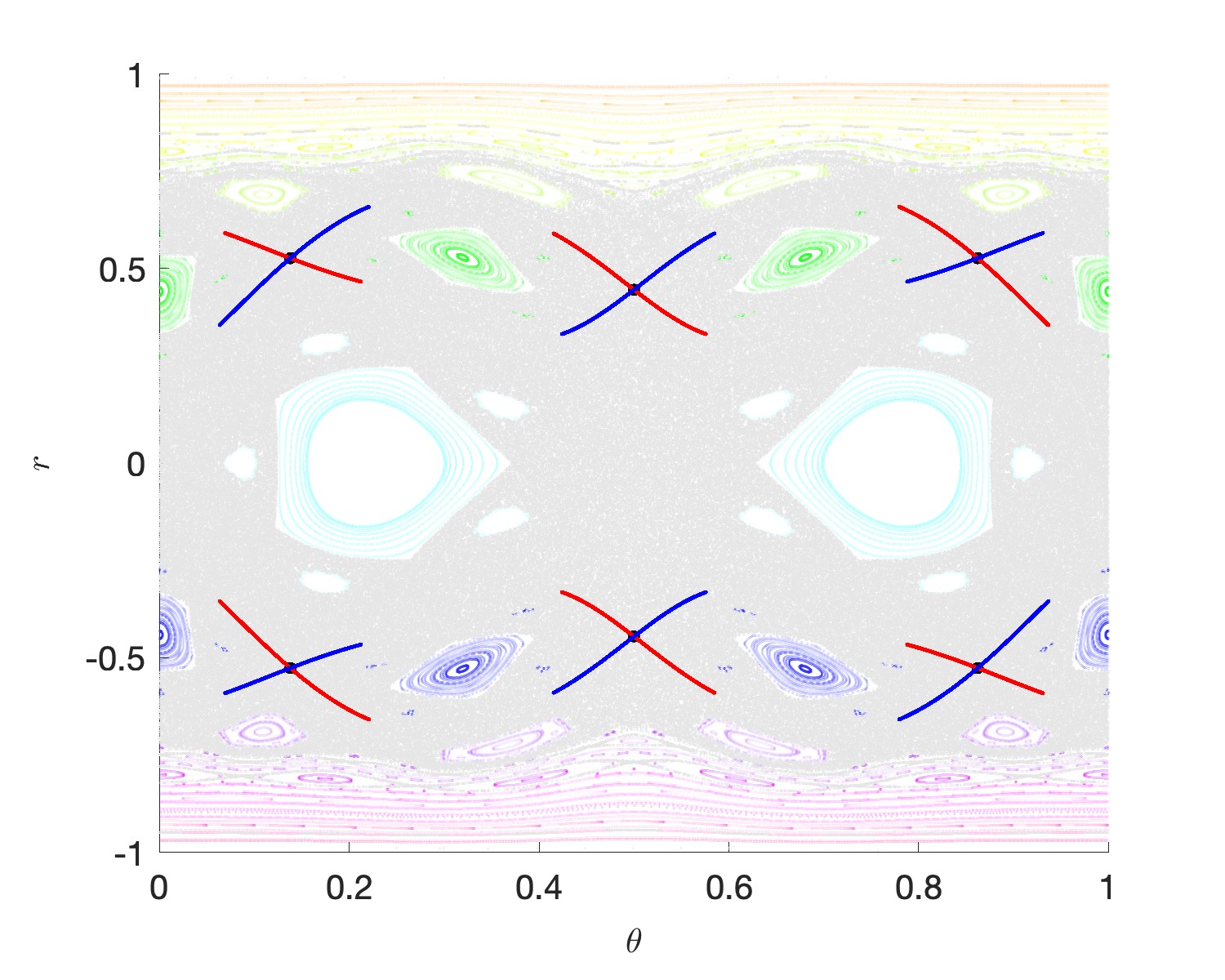}
    \hfill
    \includegraphics[width=0.49\textwidth]{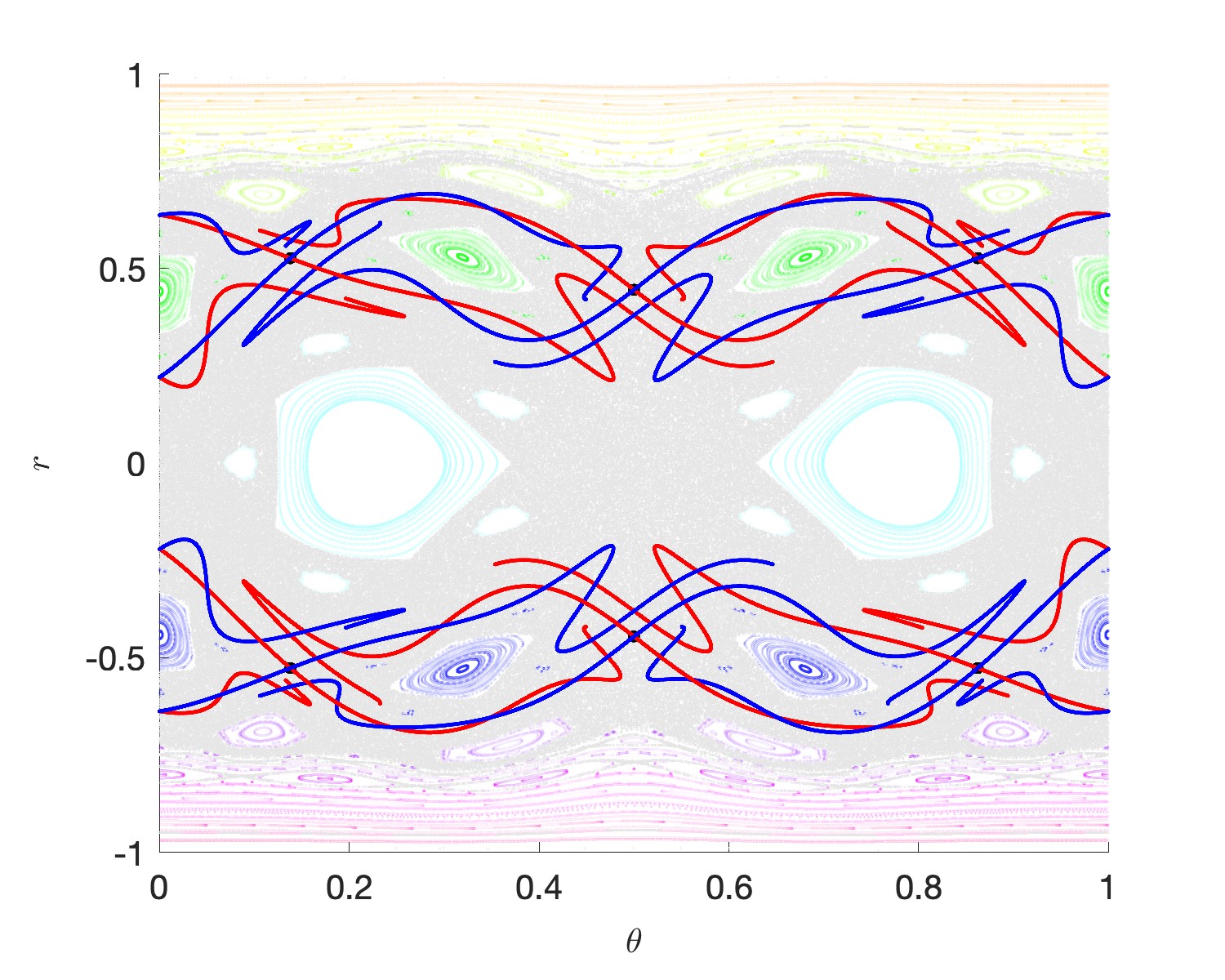}
    \caption{Local (left) and global (right) stable (blue) and unstable (red) manifolds for a period-3 orbit for Table C (cf. Tbl.~\ref{table:coeff}). The global manifolds were computed with 6 iterates of the fundamental domain and $N=60$.
    \label{fig:manifolds_p3}}
\end{figure}

\begin{figure}[tbph!]
    \centering
    \includegraphics[width=0.49\textwidth]{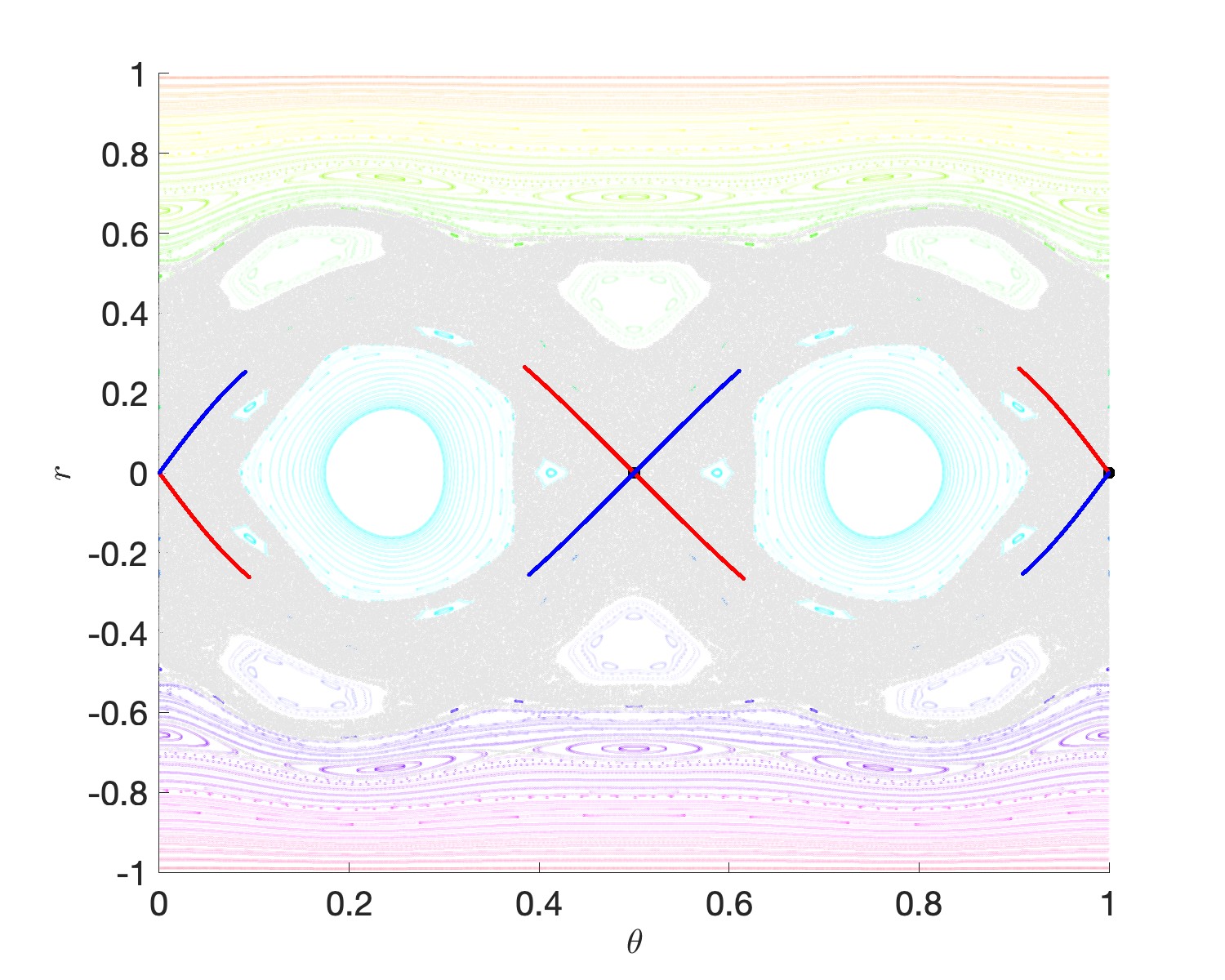}
    \hfill
    \includegraphics[width=0.49\textwidth]{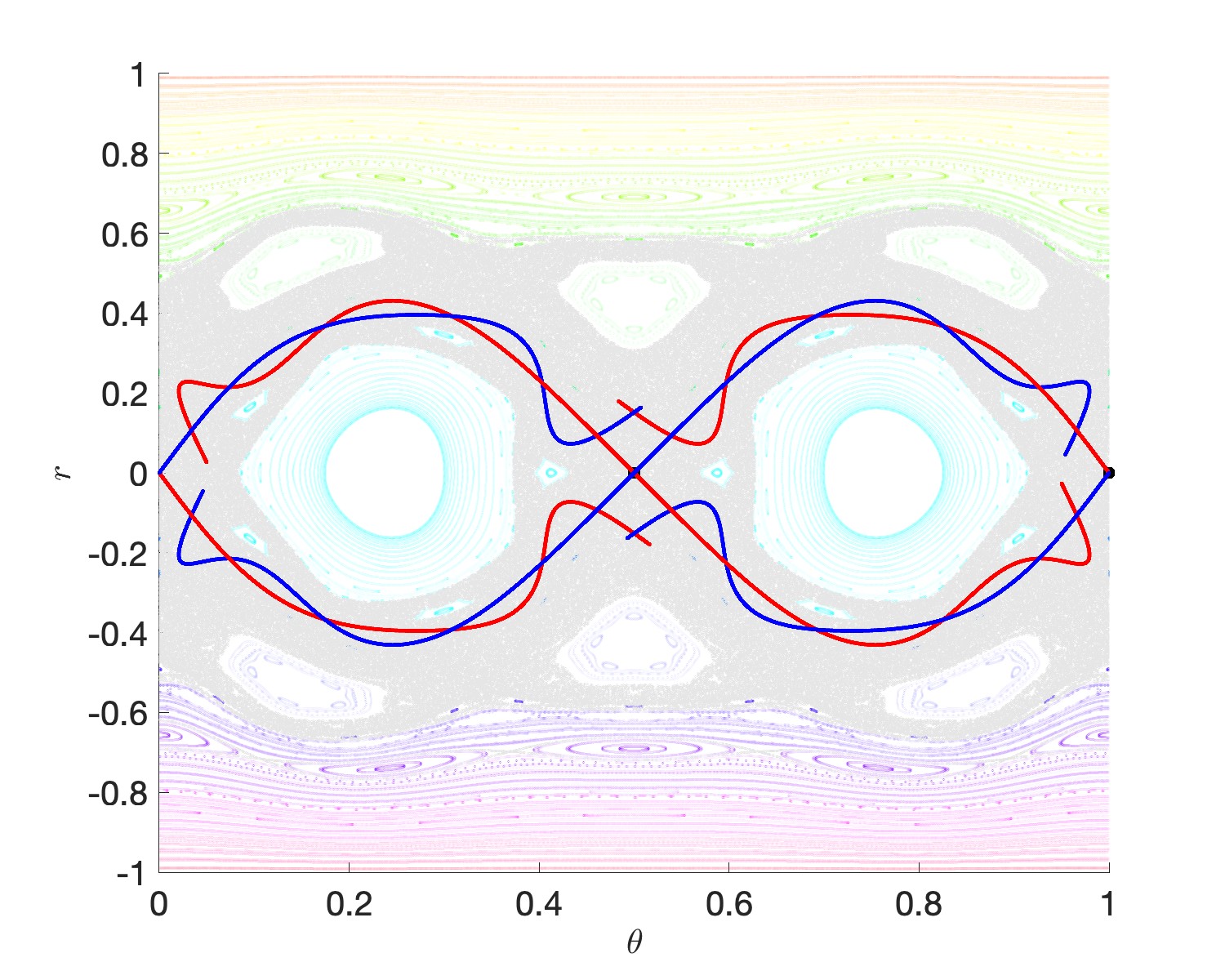}
    \caption{Local (left) and global (right) stable (blue) and unstable (red) manifolds for a period-2 orbit for Table B (cf. Tbl.~\ref{table:coeff}). The global manifolds were computed with 4 iterates of the fundamental domain and $N=60$.
    \label{fig:manifolds_p2}}
\end{figure}

The novelty of the present computation lies in performing the Newton iteration directly in the spectral domain, allowing both stable and unstable parameterizations to be obtained with high-order accuracy and without explicit recursion on the coefficients. In contrast to traditional recursive or pointwise approaches, this formulation combines spectral evaluation and FFT-based transforms to achieve rapid convergence, ensuring both robustness and precision in the manifold computation. 

To verify the spectral accuracy of the computed parameterizations, we analyze the decay of the coefficients $a_n$ of $P^N$ obtained from the Newton iteration. For both the stable and unstable manifolds, the magnitude of each coefficient is evaluated as $ \|a_n\|_{\infty} $, and its logarithmic scale $ \log_{10}(\|a_n\|_{\infty}) $ is plotted against the mode index \( n \). The resulting coefficient profiles, shown in the center and right panels of Figure~\ref{fig:manifold_coeffs} for a period-5 orbit of Table B (cf. Tbl.~\ref{table:coeff}), exhibit a rapid exponential decay of the coefficients, reaching values of order $10^{-16}$ by the truncation index \( N \). This confirms that the spectral representation \( P^N(z) = \sum_{n=0}^{N} a_n z^n \) provides a highly accurate analytic approximation of the local manifolds within the domain of convergence of the parameterization.
\begin{figure}[htbp!]
    \centering
        \includegraphics[width=0.32\textwidth]{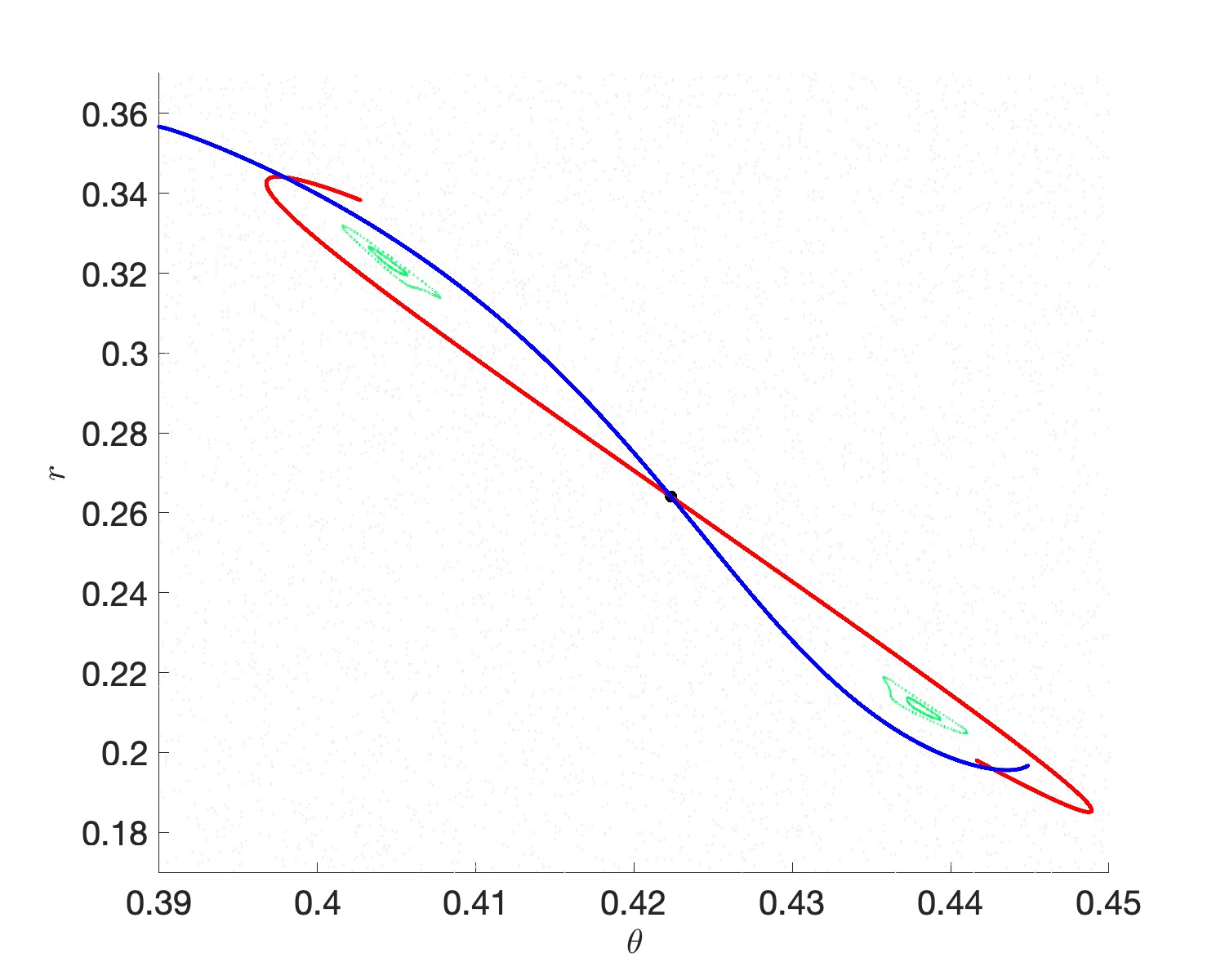}
    \hfill
        \includegraphics[width=0.32\textwidth]{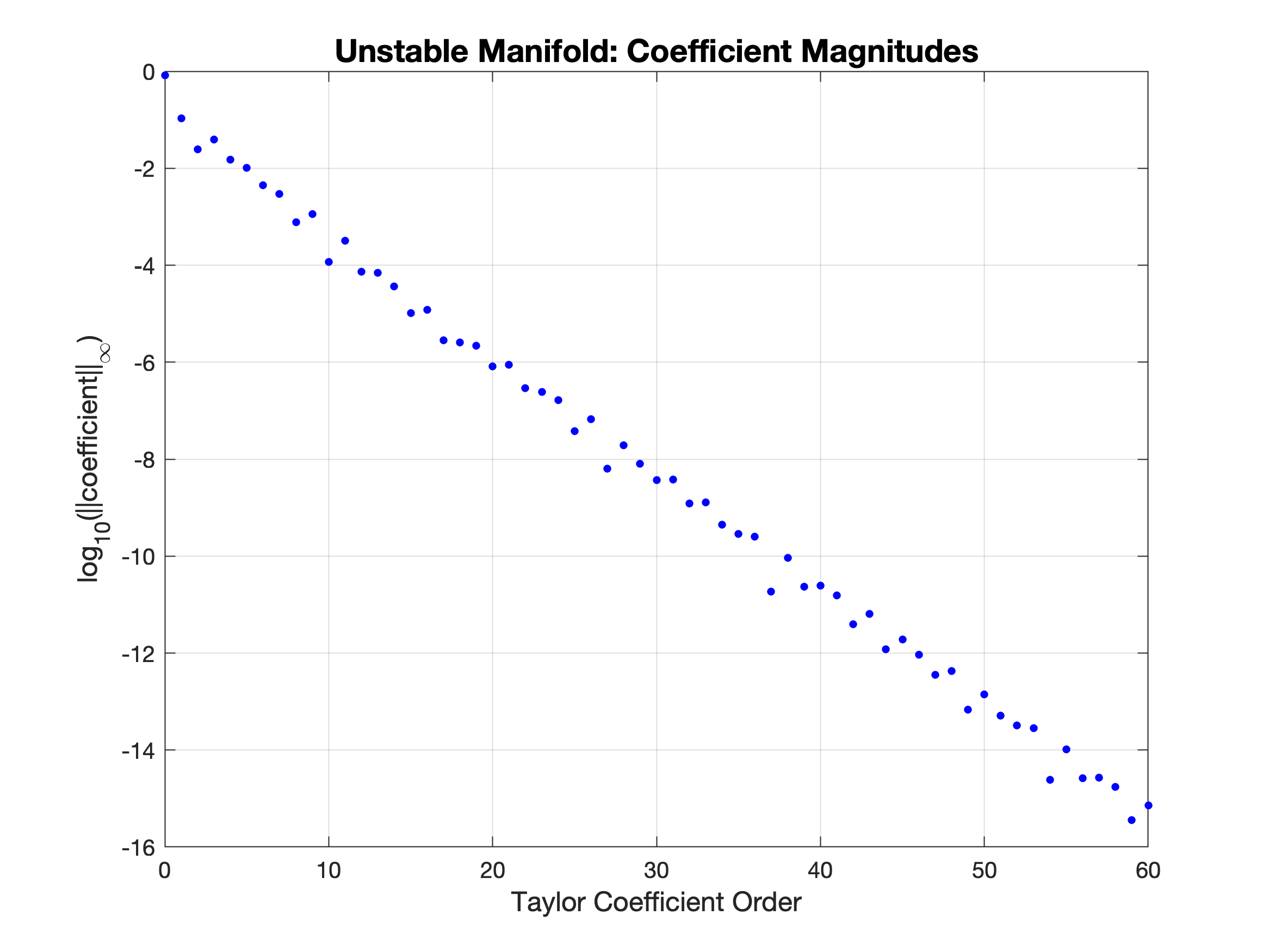}
    \hfill
        \includegraphics[width=0.32\textwidth]{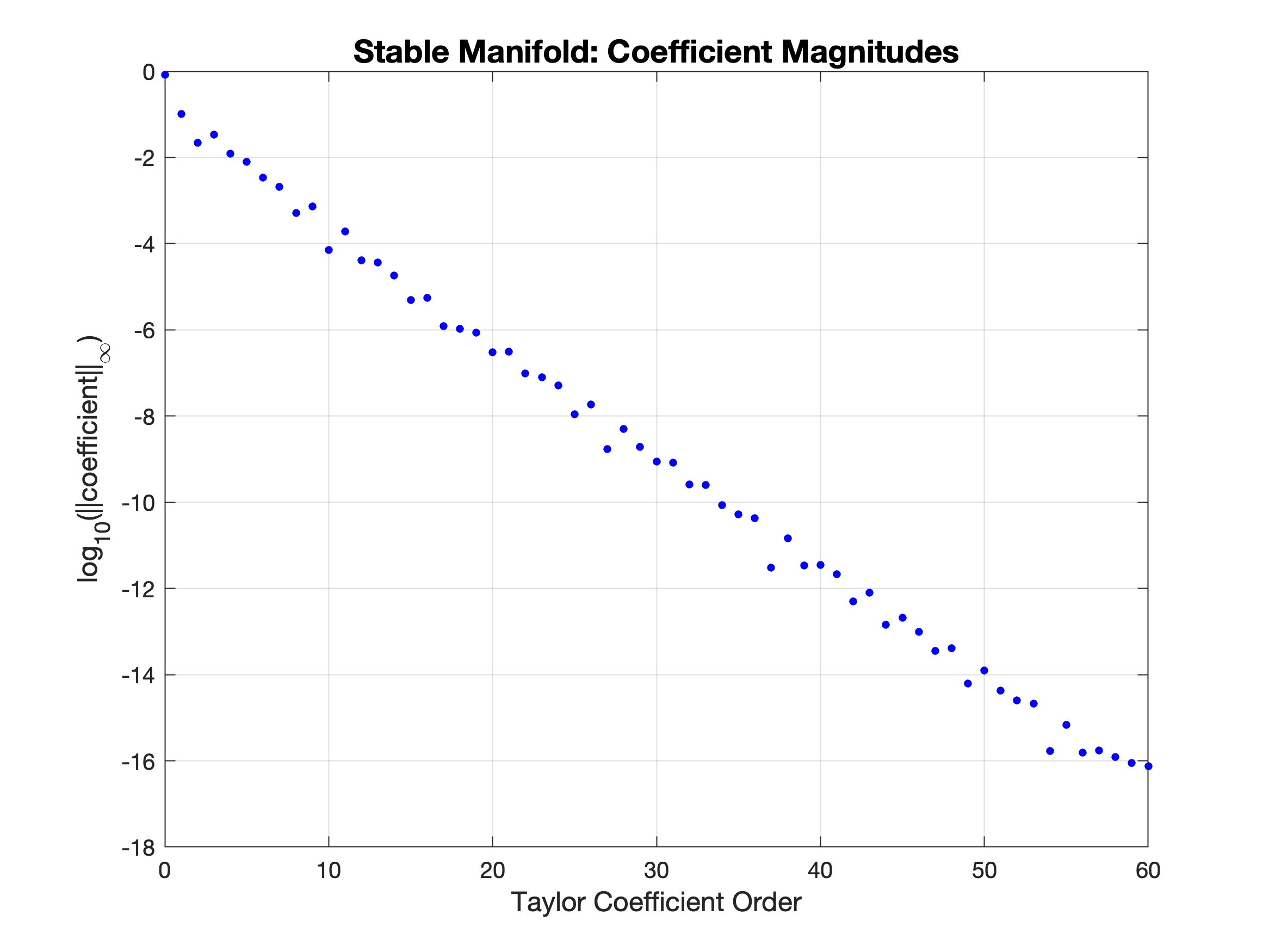}
    \caption{(Left) Computed global stable and unstable manifolds for one point of the period-5 orbit of Table B (cf. Tbl.~\ref{table:coeff}). (Center, right) Logarithmic decay of the infinity norm of Taylor coefficients $\|\mathbf{a}_n\|_\infty$ for the local unstable and stable manifolds, respectively. Computations performed using $N=60$ and 5 iterates of the fundamental domain.
    \label{fig:manifold_coeffs}}
\end{figure}

Another key advantage of our approach is that it does not stop at the linear tangent eigenvector chart, but by increasing the truncation order \(N\) of the expansion  of $P^N(z)$, one can capture the nonlinear deformation of the invariant manifold at larger amplitude. In particular, while the leading term $a_1$ corresponds to the linear eigenvector scaling, the higher coefficients $a_n$ with $n\ge2$ encode the curvature, folding and higher‐order geometry of the manifold. By systematically computing and monitoring these coefficients, our implementation allows the eigenvector to be scaled to increasingly large magnitudes and hence captures local nonlinear behavior of the manifold near the periodic orbit, see Figures~\ref{fig:manifold_coeffs_NFFT_comparison}-\ref{fig:manifold_coeffs_NFFT_comparison_Per10} for a visual illustration of this behavior. In both figures, increasing $N$ captures more nonlinear geometry of the manifolds, with scale parameters $\text{scale}_u$ and $\text{scale}_s$ (see Appendix~\ref{appen:a2} for definitions) adjusted accordingly to control the extent of the parameterization from the periodic orbit.
\begin{figure}[htbp!]
    \centering
    \includegraphics[width=0.32\textwidth]{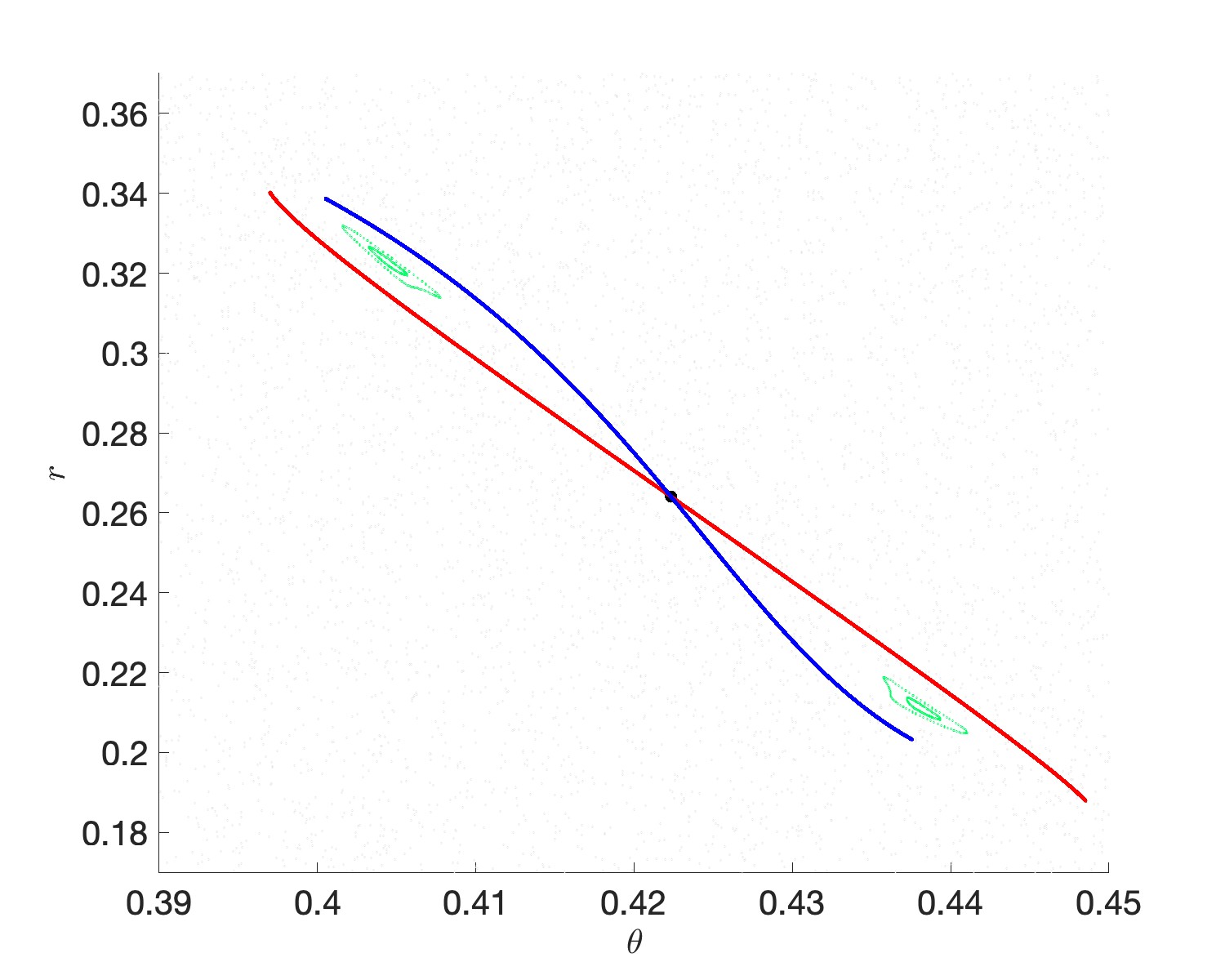}
    \hfill
    \includegraphics[width=0.32\textwidth]{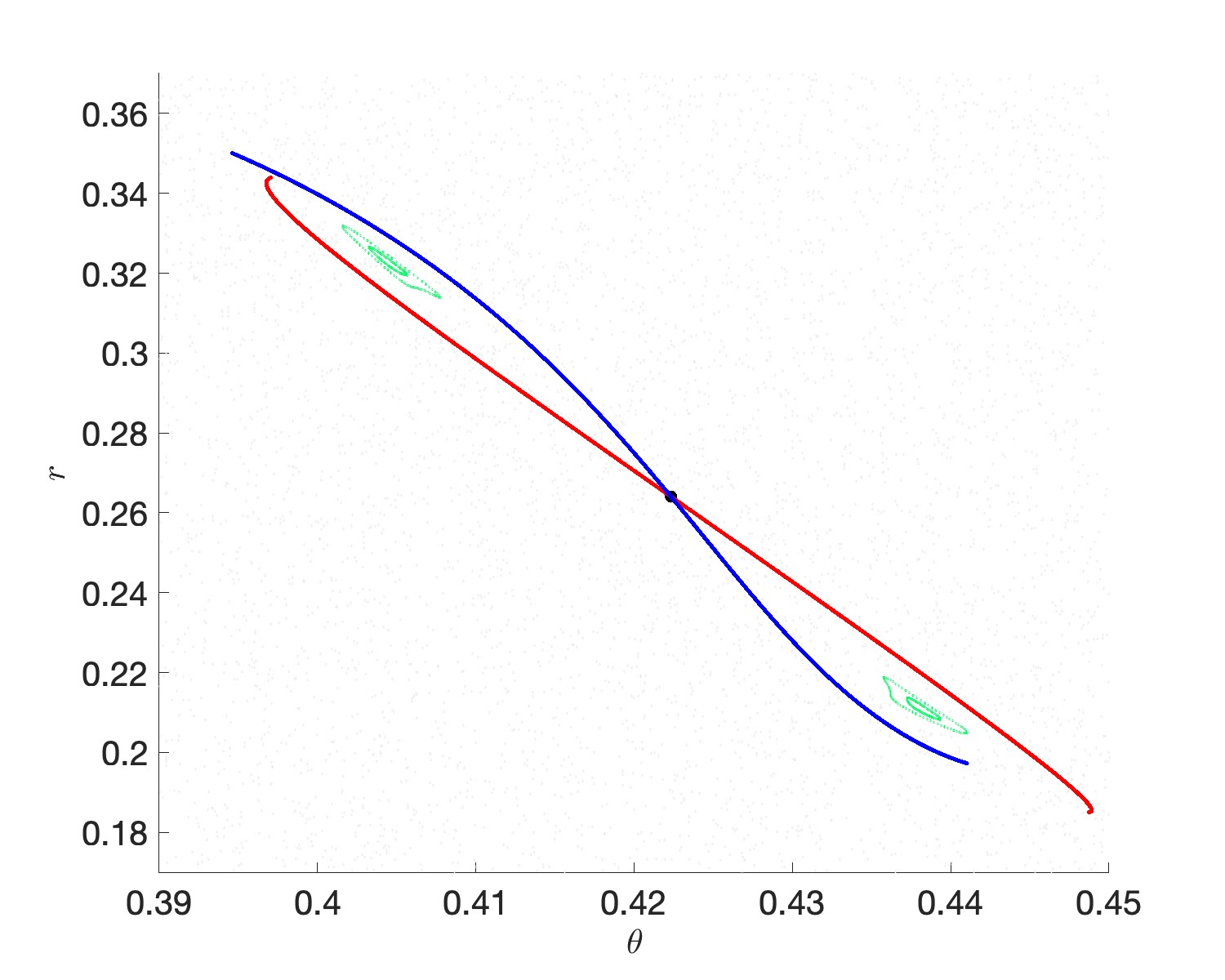}
    \hfill
    \includegraphics[width=0.32\textwidth]{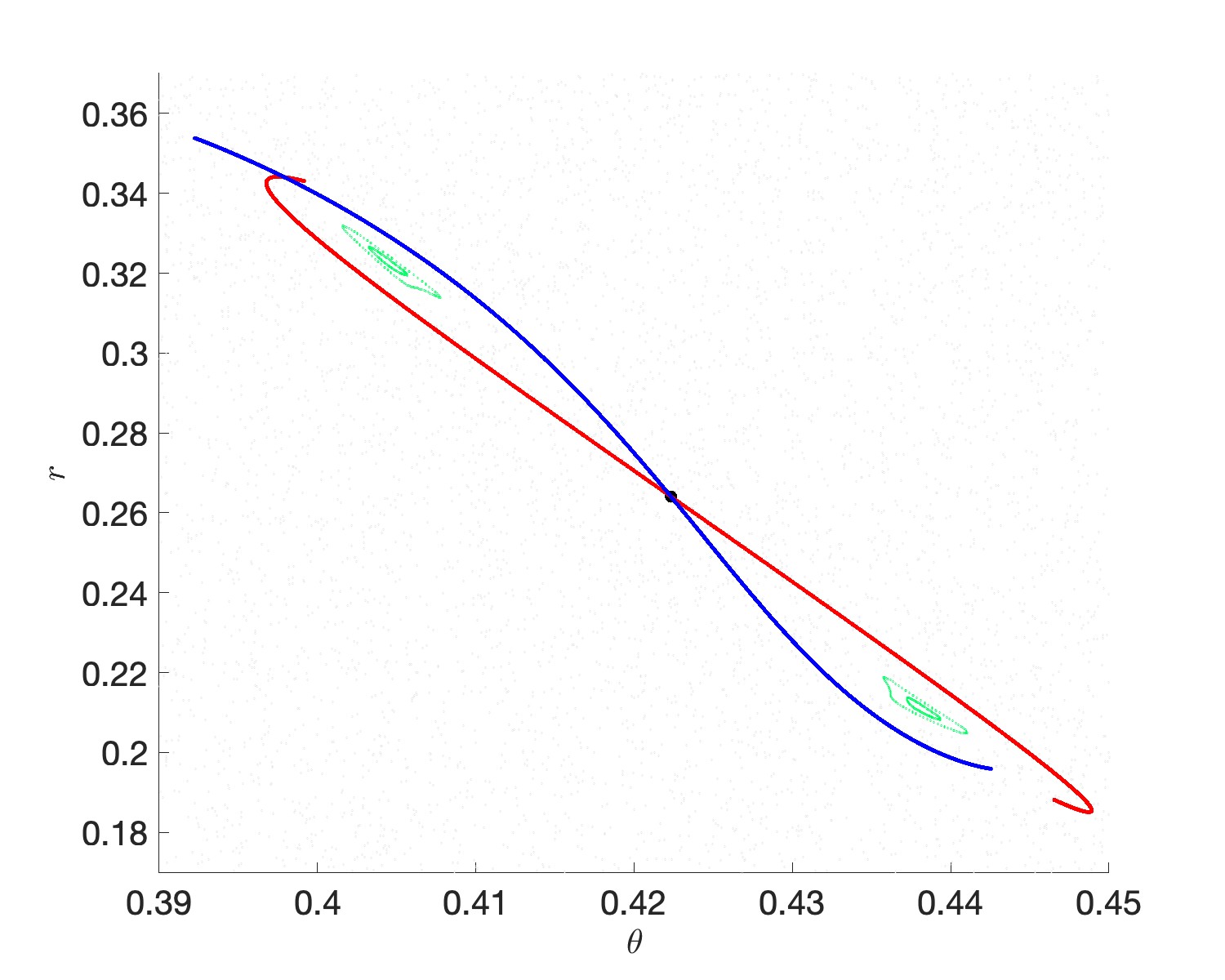}
    \caption{Computed local stable and unstable manifolds for the same period-5 orbit from Figure \ref{fig:manifold_coeffs} with varying polynomial order $N$. (Left) $N=60$, $\text{scale}_u = 0.21$, $\text{scale}_s = 0.2$. (Center) $N=100$, $\text{scale}_u = 0.26$, $\text{scale}_s = 0.26$. (Right) $N=180$, $\text{scale}_u = 0.32$, $\text{scale}_s = 0.293$; transverse intersections occur even for the local manifolds.}
    \label{fig:manifold_coeffs_NFFT_comparison}
\end{figure}

To assess the robustness of our numerical methods, we systematically increase the perturbation amplitude applied to the coefficients in Table~C (cf. Tbl.~\ref{table:coeff}). Specifically, we multiply all coefficients by a factor of 3.5, representing an extreme perturbation far beyond the nominal parameter regime. In this strongly perturbed case, the phase space reveals extensive chaotic behavior: iterates densely populate the entire accessible domain in the $(\theta, r)$ plane, forming a nearly uniform cloud of points that obscures much of the underlying structure. The chaotic sea extends throughout the region $0 \leq \theta \leq 1$ and $-0.5 \leq r \leq 0.5$, with particularly dense regions near $r \approx \pm 0.5$ where iterates exhibit wild oscillations. Despite this extreme perturbation—a 250\% increase over the nominal values—we successfully compute the stable and unstable manifolds of the hyperbolic periodic-2 orbits located near $(\theta, r) \approx (0.25, 0)$ and $(0.75, 0)$. The successful resolution of these invariant manifolds under such severe perturbation, where the system dynamics have deteriorated into an almost completely chaotic state, demonstrates the remarkable effectiveness of our computational approach even as it operates near the boundary of numerical feasibility.

\correction{B-R1-4}{As a further robustness check, we tested a boundary perturbation with many small
nonzero Fourier modes rather than only three: we set $a_{k,1} = a_{k,2} =
\nu/k^2$ for $k = 1,\dots,20$, with $\nu = 0.03$ matching the leading-mode amplitude of
Table A, keeping $a_{1,1} = 1.1$ and $b_{1,2} = 1$ as in Tables A--C. We computed the
local stable and unstable manifolds of a period-5 orbit for this perturbation. Due to the slow decay, a larger proportion of the phase space is chaotic, and the resulting manifolds are more oscillatory than the ones we found for Tables A--E, see Figure~\ref{fig:table7_manifold}. The
periodic orbit was refined to a closure error of $3.318 \times 10^{-16}$, with
multipliers $\lambda_u = 1.558838$ and $\lambda_s = 0.641503$, consistent with a
hyperbolic saddle orbit. The Newton iteration for the stable manifold converged to the
$10^{-14}$ tolerance in 6 iterations, while the unstable manifold converged to a residual
of $2.24\times 10^{-12}$ after 10 iterations. The corresponding Taylor coefficients decay
approximately exponentially over the full range $0 \le N \le 60$, reaching magnitudes
below $10^{-12}$, comparable to the decay observed for Tables A--E (cf. Figure~\ref{fig:manifold_coeffs}).
These results indicate that the efficiency and accuracy of our method are not
significantly affected by the presence of many small nonzero Fourier coefficients, as
opposed to a single dominant mode. We additionally benchmarked the per-call runtime of the map and its Jacobian for a fixed polynomial across a family of decay laws ($1/k^p$ for $p=3,\dots,7$ and
$1/2^k$) and mode counts $N = 2,3,10,20,30,40,60,80,100$, evaluated at a fixed, Newton-refined periodic orbit point; runtime scales approximately linearly with the
number of retained modes once the boundary geometry stabilizes, and is not sensitive to
the specific choice of decay law.}

\begin{figure}[htbp]
    \centering
    \includegraphics[width=0.7\textwidth]{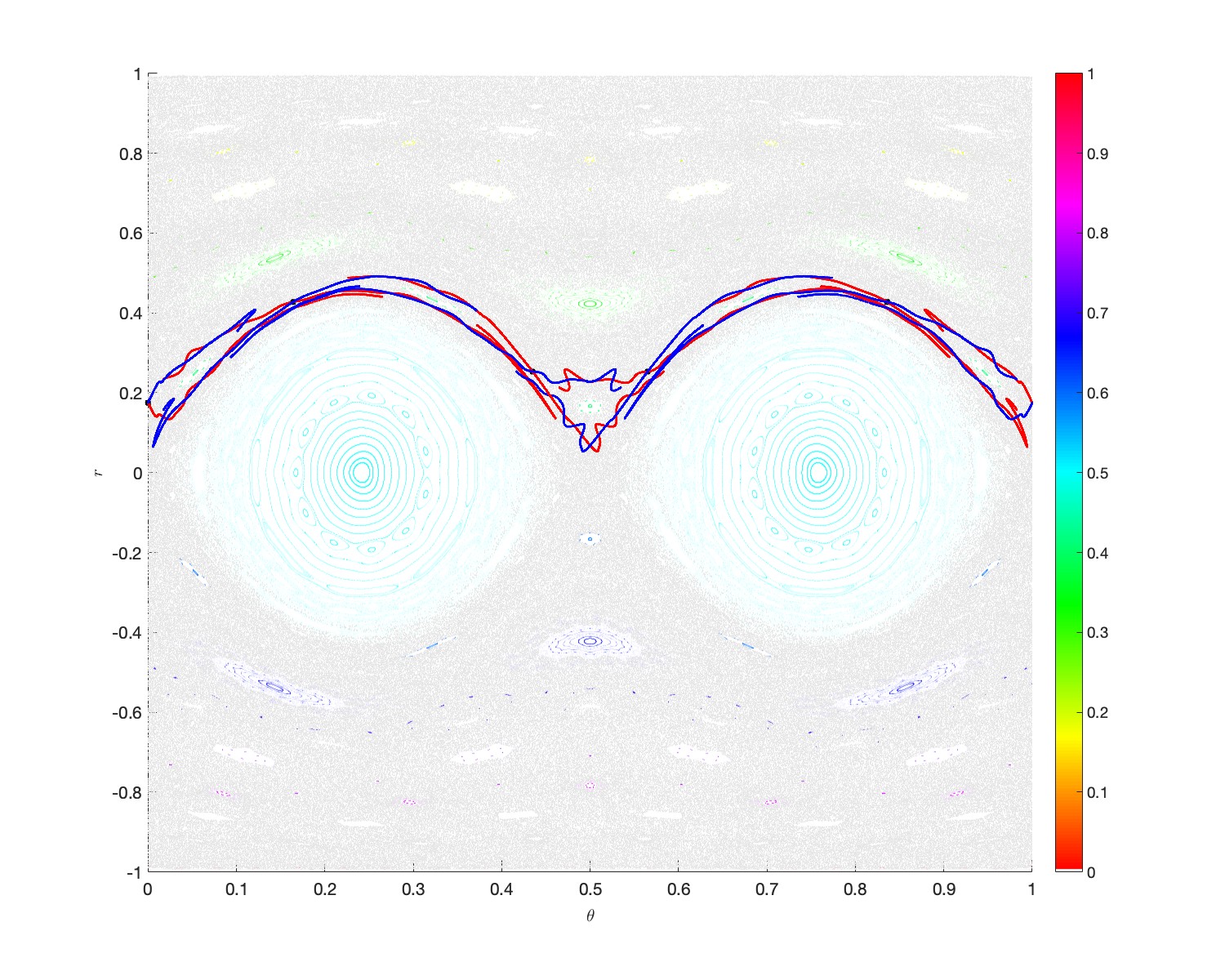}
    \caption{Stable (blue) and unstable (red) manifolds of the period-5 orbit for the
    $1/k^2$-decaying, $N=20$-mode boundary perturbation, overlaid on the phase plane
    $(\theta, r)$.}
    \label{fig:table7_manifold}
\end{figure}

\begin{figure}[htbp!]
    \centering
    \includegraphics[width=0.49\textwidth]{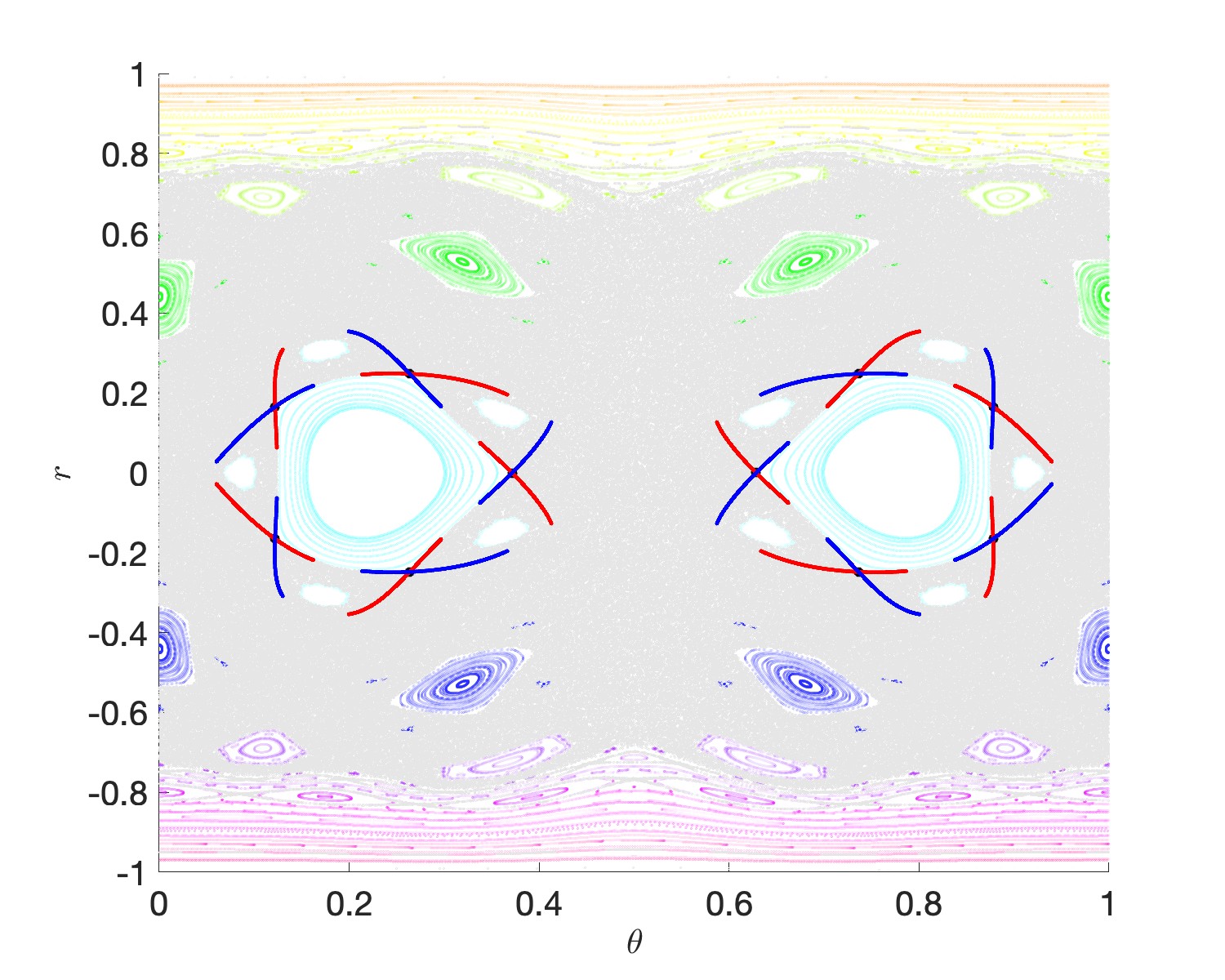}
    \hfill
    \includegraphics[width=0.49\textwidth]{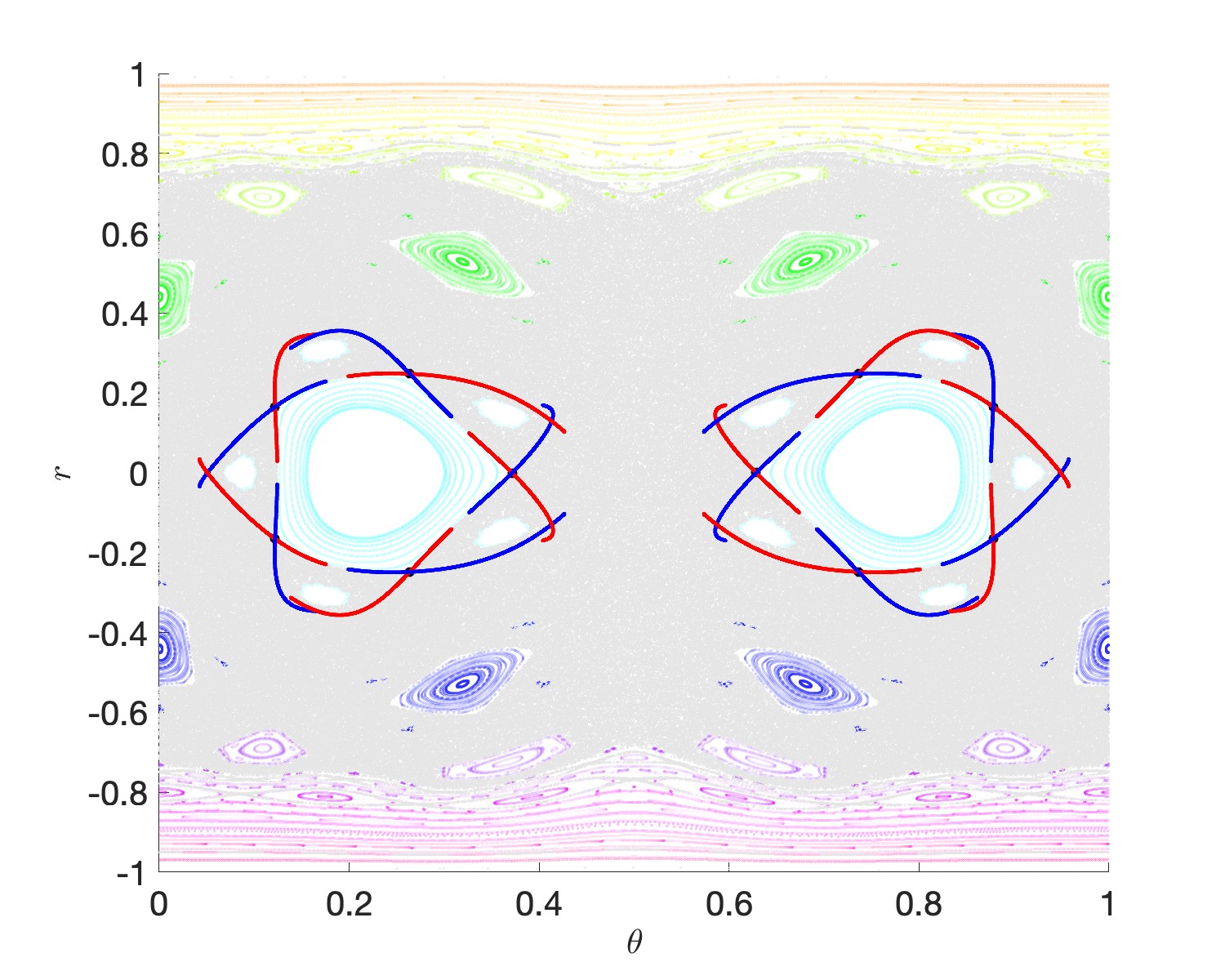}
    \caption{Computed local stable and unstable manifolds for the period-10 orbit for Table C (cf. Tbl.~\ref{table:coeff}) with varying polynomial order $N$. (Left) $N=60$, $\text{scale}_u = 0.34$, $\text{scale}_s = 0.34$. (Right) $N=200$, $\text{scale}_u = 0.51$, $\text{scale}_s = 0.51$; transverse intersections occur even for the local manifolds. 
    }
    \label{fig:manifold_coeffs_NFFT_comparison_Per10}
\end{figure}

\subsection{Results for billiards}

Here we concentrate on observations in terms of what our manifolds indicate about the billiard tables. 
In Appendix A, we have listed all of the periodic orbits found for the five tables.
For example, for Billiard Tables A and B, we found a total of eight periodic orbits and manifolds, with periods 2, 3, 5, 10, and 30, with a similar scope in all other cases. In Figures~\ref{fig:PS1}--\ref{fig:PS3}, we give an overview of all of our computations of invariant manifolds of Tables A--E listed in~\eqref{eq:pert_ell} and Tbl.~\ref{table:coeff}. 

\correction{B-R2-11}{In most of the examples we computed, we find that 
either the local manifold already intersect transversely, or that this clearly happens
after a small number of iterates. A notable exception is the period-30 case, where the splitting
angle appears to be very small.  This is not unexpected, as the splitting angle
is known to be exponentially small in the period.  For the period 30 orbit on Table~A, the intersection only
becomes visible once the number of iterates is increased to 100, see Figure~\ref{fig:period30}. For Table~B it remains unresolved whether the stable and unstable manifolds for  period-30 intersect transversely even after 100 iterates. More generally, due to exponential splitting, we expect that it would not be possible to distinguish transverse intersections for many high period orbits without using extended-precision calculations, such as the ones performed in~\cite{Martin_2016}.} 

Due to the homoclinic connections, the manifolds are trapped inside the connected chaotic region containing the original periodic orbit, wrapping around to avoid the islands. For example, see the period-2 manifolds wrapping around the two main islands in Figure~\ref{fig:manifolds_p2}, the period-5 manifolds wrapping around a small pair of islands in Figure~\ref{fig:manifold_coeffs}, and the period-10 manifolds wrapping around two star-shape pieces consisting of twelve islands in Figure~\ref{fig:manifold_coeffs_NFFT_comparison_Per10}.
The larger eccentricity cases have the same behavior, but all the islands are squeezed into a smaller space, meaning that it is harder to observe.  

\begin{figure}[h]
    \centering
    \includegraphics[width=0.89\textwidth]{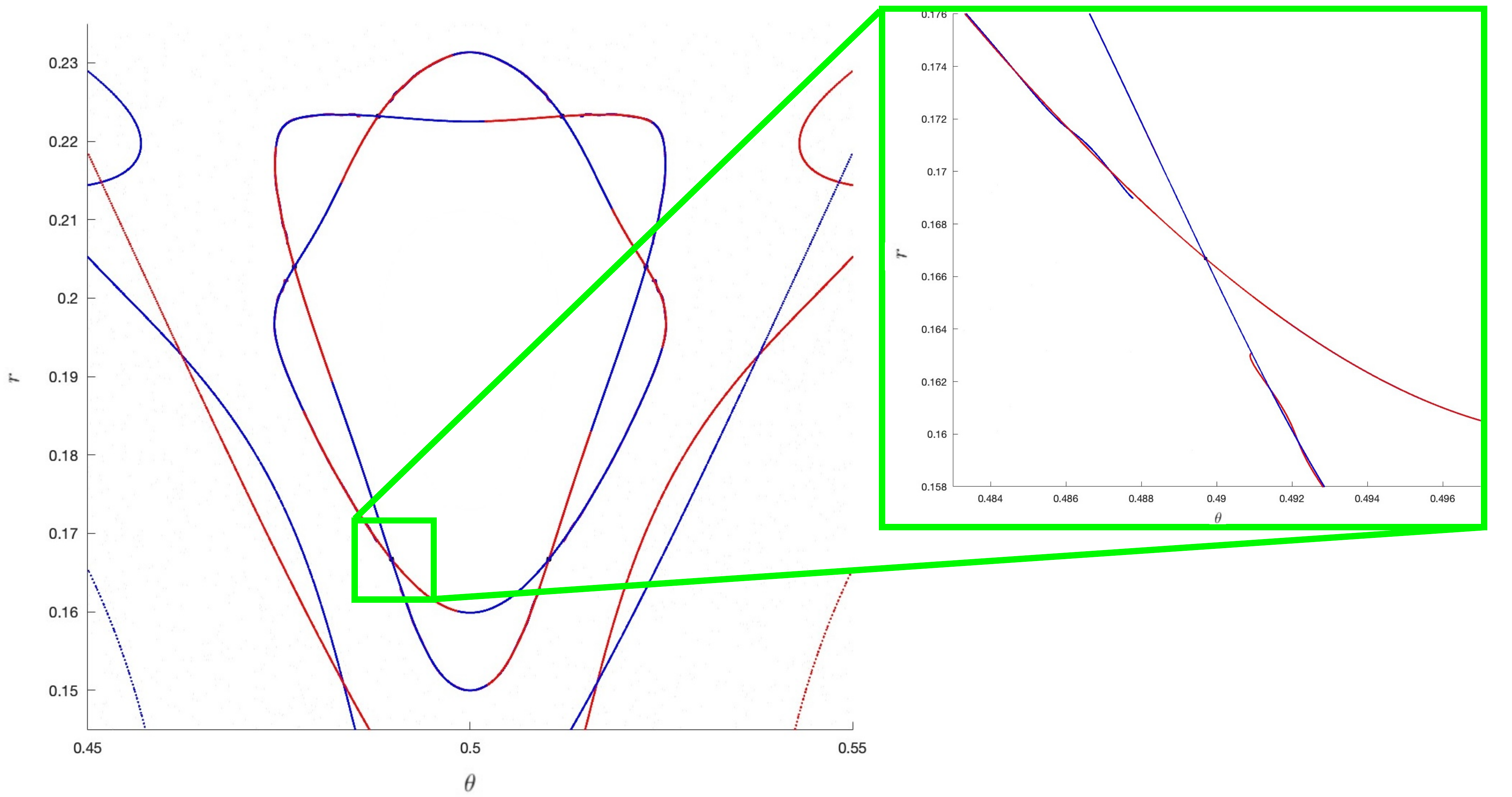}
    \caption{Computed global stable and unstable manifolds for a six point set of the period-30 orbit for Table~A (cf. Tbl.~\ref{table:coeff}), with 100 iterations of the fundamental domain. (Right) Zoomed plot at one of the period-30 points, showing the transverse intersection.}
    \label{fig:period30}
\end{figure}

The resulting manifolds and homoclinic tangle beautifully align with the chaotic regions computed using the weighted Birkhoff average.
Each set of manifolds and their intersections fills up the entire connected chaotic sea that contains it, giving good agreement between these two methods. This is already somewhat clear in Figures~\ref{fig:PS1}-\ref{fig:PS3}, although we had kept the number of iterations rather small to avoid overcrowding. It is much more obvious in Figure~\ref{fig:manifold_per2tab4_10it}, where we show only one set of manifolds and allow the number of iterations of the local manifolds to be ten. 

We are certainly not the first to observe that these maps bear a striking resemblance to the behavior of the Chirikov standard map. In the perturbed case (positive parameter for the Chirikov standard map), 
they both have elliptic islands, along with small chaotic regions, which are filled with homoclinic tangles of hyperbolic periodic orbits. 
For a strictly convex billiard table, if the boundary map $B$ is parameterized by arclength,  then the billiard map is symplectic~\cite{Meiss_1992}. Though we do not assume arclength parameterization, since the frequency is independent of the parameterization, we can conclude from this that  the frequency of the rotational  tori increases as $r$ increases, as we can observe from the figures of the phase space. 

We note that there are a few key differences. The Chirikov standard map has large elliptic period-one islands, whereas here the lowest period possible is two. For the Chirikov standard map, there is a famous parameter value for which all 
rotational orbits break up~\cite{sander_birkhoff_2020}. In the billiards case, no perturbation breaks up all rotational orbits. In particular, for every  table with a $C^6$ smooth strictly convex boundary, there exist ``whispering gallery" orbits~\cite{lazutkin73, douady82}.  Finally, the Chirikov standard map is entire, whereas billiard maps can only be continued to a finite thickness strip of the complex plane.  
\begin{figure}[htbp!]
    \centering
    \includegraphics[width=0.89\textwidth]{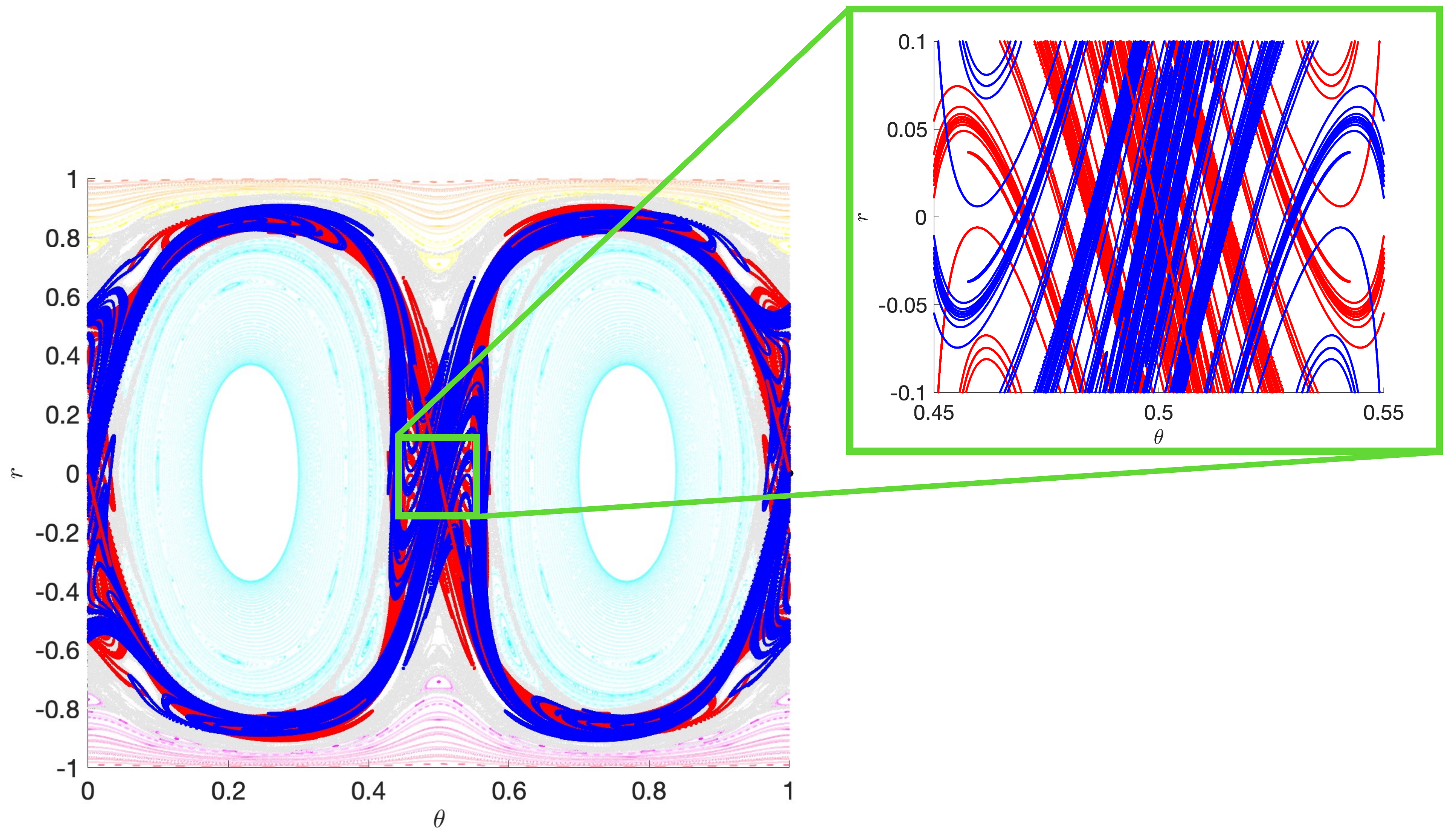}
    \caption{(Left) Computed global stable and unstable manifolds for a period-2 orbit for Table E (cf. Tbl.~\ref{table:coeff}), $N=40$, $\text{scale}_u = 0.52$, $\text{scale}_s = 0.48$ and $10$ iterations of the fundamental domain. (Right) Zoomed in plot at one of the period-2 points $(\theta, r) = (0.5,0)$ to clearly discern chaotic behavior in that particular region.}
    \label{fig:manifold_per2tab4_10it}
\end{figure}
Of particular interest is how well we can see homoclinic tangle within the chaotic region, such as the Cantor set type behavior of the manifolds, and how well it illustrates in this context the results of the $\lambda$-lemma.

The $\lambda$--lemma describes a fundamental geometric mechanism by which images of disks (or arcs) transverse to a stable (resp. unstable) manifold are stretched and asymptotically aligned with the unstable (resp. stable) manifold under forward (resp. backward) iteration of the dynamics; informally, a small transverse disk is ``carried along'' and accumulates onto the invariant manifold while its tangent directions converge to the tangent bundle of that manifold. This statement is a classical ingredient in hyperbolic dynamics and its ramifications for homoclinic and heteroclinic behavior are well documented (see, \cite{robinson1998dynamical,capinski2020computer}).

Numerically, the mechanism predicted by the $\lambda$--lemma is captured here in two complementary ways. First, the high-order parameterizations $P^N(z)$ provide accurate local charts (disks) transverse to the periodic orbit together with reliable approximations of their tangent bundles: the leading coefficient $a_1$ recovers the tangent eigenvector while the higher coefficients $a_{n\ge2}$ encode the local nonlinear deformation. Second, by iterating the computed local chart under the map \(F\) (or its inverse) and monitoring both pointwise images and the evolution of tangent information, we observe the characteristic stretching, alignment and accumulation described by the $\lambda$--lemma. Practically, the method’s spectral accuracy (exponential coefficient decay to machine precision) and the Newton method applied ensure that these iterated images remain within the radius of analyticity of the parameterization and that the observed alignment is not a numerical artifact. Thus, by combining high-order spectral charts with controlled iteration in the fundamental domain, our implementation reproduces the geometric content of the $\lambda$--lemma in a numerically verifiable way. A visualization of this phenomenon can be seen in Figure~\ref{fig:manifold_per2tab4_10it} for a period-2 orbit of Table E (cf. Tbl.~\ref{table:coeff}).


\section{Conclusions and future plans}\label{sec:conclusions}

In this paper, we developed a $C^1$ method for iterating billiard maps for 
tables with smooth convex boundaries. 
We then used the weighted-Birkhoff average to distinguish between chaos and quasiperiodic orbits, at the same time as computing the frequency of the quasiperiodic orbits. We used the method of multiple-shooting to find periodic orbits for the billiard map and selected the hyperbolic periodic orbits.  We then used the parameterization method to compute the local stable and unstable manifolds for the periodic orbits \correction{B-R2-8}{of periods 2, 3, 4, 5, 10, and 30. However, our methods could also be extended to compute for other lower and higher periods.} This use of the parameterization method is novel, in that the billiard map is fully implicit. Therefore, to find the coefficients of the Taylor series approximation of the manifolds, we analytically continued our billiard map into the complex plane, allowing us to compute  series coefficients using the Discrete Fourier Transform. Finally, we are able to approximate the global manifolds up to a chosen number of iterations and find transverse intersections. 

In addition to the computational novelty of this work, the application to billiards is itself of great interest. Previous results on homoclinic tangles for billiard maps have concentrated on the singular limit as the perturbation approaches the ellipse, showing that there are exponentially small angles of intersection of the manifolds. 
In contrast, we are interested in homoclinic tangles for billiard tables which are non-limiting perturbations from the ellipse. 

One of the advantages of the parameterization method for computing manifolds is that in addition to giving high accuracy approximations, the method can be used for computer assisted proofs. Not just local manifolds, but also a finite number of iterations of these local manifolds can be rigorously validated. In future work, we plan to validate our manifolds. Since the existence of transverse homoclinic points implies chaos, such a validation will lead to a rigorous validation of chaos for perturbed elliptic billiard maps. 

The numerical algorithms described in this work have been implemented in MATLAB and are available as open-source code at \url{https://github.com/efleurantin2103/Billiards}, enabling reproduction of all computational results presented here.


\section*{Acknowledgements}
SC was partially supported by the GMU Math Industrial Immersion Program.
EF was partially supported by the National Science Foundation  grant DMS-2137947. JMJ was partially supported by National Science Foundation 
grant DMS-2307987.
ES was partially supported by the Simons Foundation under Award 636383. 
\correction{Thanks}{Thank you to the anonymous referees, whose suggestions significantly improved this paper.}

\appendix

{\tiny

\section{Computational Parameters and Technical Details}

This appendix provides the parameters used to compute the global stable and unstable manifolds.  It also gives a technical modification for our numerical method for finding an initial guess for $f$.

\subsection{Modification to find our initial guess}

We now describe our second technical modification to our method for computing the initial guess in order to perform Newton's method to find the billiard iterate $f(\theta,r)$.  To make sure that 
the initial guess is always well defined, we consider $\mathbf{v}_{-1}$, the previous direction of $\mathbf{v}$, and use the $\mathbf{v}_e$ that is the reflection of 
$\mathbf{v}_{-1}$ off the elliptical table. 

\begin{figure}[htbp!]
\begin{center}
  \includegraphics[width=0.2\textwidth]{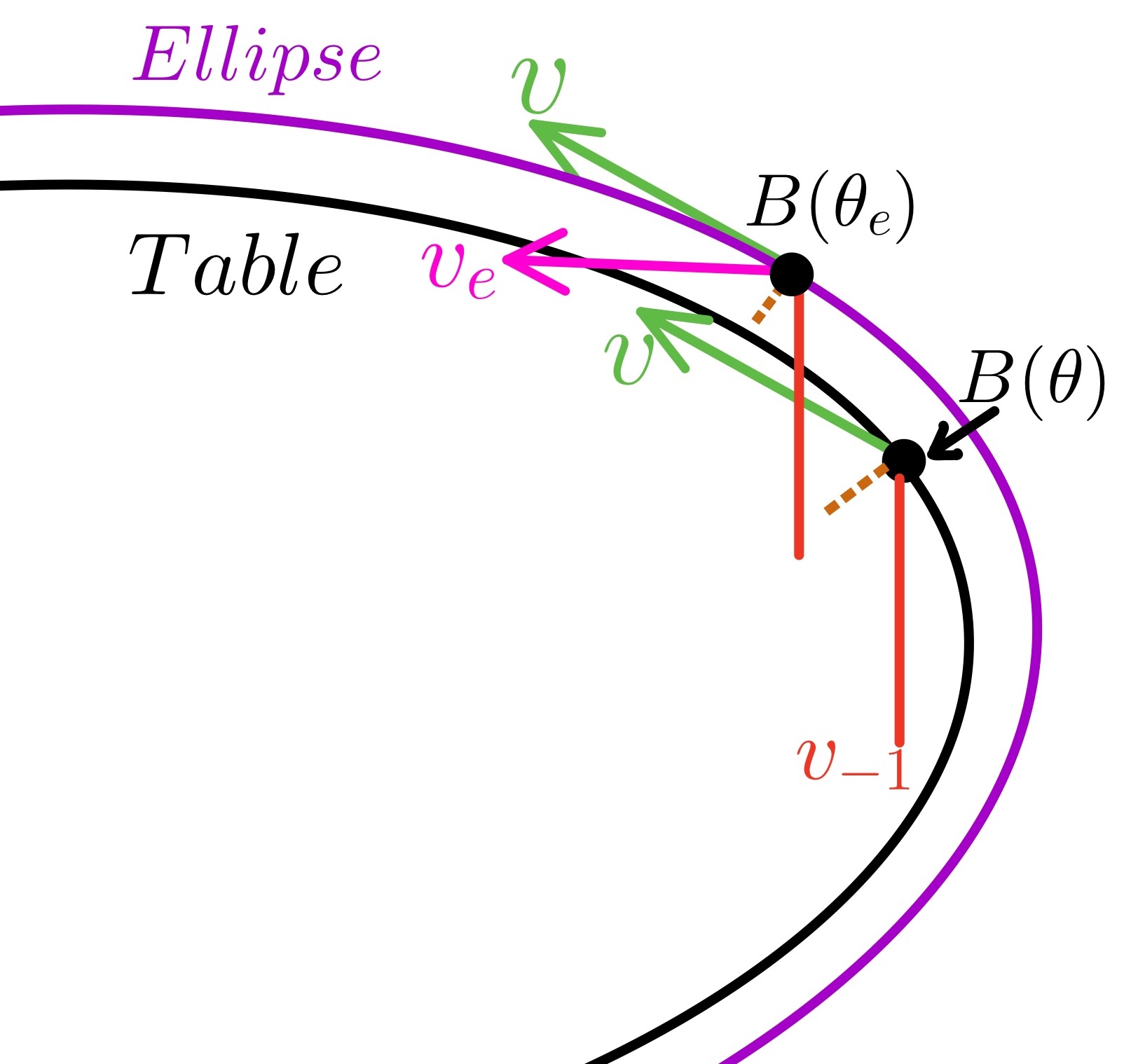}
  \caption{The reflection vector $v$ for our table might not point into the interior of the associated elliptical table. To make sure it is always well defined, we use $v_e$ in finding the guess $(\hat{\theta}_e,\hat{r}_e)$. \label{fig:modbounce}}
\end{center}
\end{figure}

\subsection{}\label{appen:a2}

 We now present the parameters and values of the periodic orbits seen in Figures~\ref{fig:PS1}--\ref{fig:PS3} in a series of data Tables, one for each period. More specifically, each data Table corresponds to orbits of a specific period and references the billiard tables in the main text where these orbits are analyzed. The parameter $r$ represents the radial coordinate, $\theta$ denotes the parameter describing the point on the boundary of the billiard table, and scale$_u$ and scale$_s$ are the scaling factors applied to the unstable and stable eigenvectors, respectively, in the parameterization method. (These were labeled $u$ and $s$ respectively in the theoretical discussion in Section~\ref{sec:parameterization}.) The variable $N$ indicates the order of the polynomial approximation, while iterations refer to the number of iterates $M$ used in the computational scheme. These parameters were carefully chosen to ensure convergence and accuracy of the computed invariant manifolds associated with each periodic orbit.


\begin{table}[htbp!]
    \centering
    \begin{tabular}{|c|c|c|c|c|c|c|}
    \hline
         \textbf{Period 2} & r  & $\theta$  & scale$_u$ & scale$_s$ & Iterations  & $N$ \\ \hline
        Table A & 0.0000 & 0.5000 & 0.48 & 0.46 & 6 & 60 \\ \hline 
        Table B & 0.0 & 0.5 & 0.45 & 0.43 & 5 & 60 \\ \hline
        Table C & 0.0 & 0.5 & 0.32 & 0.32 & 6 & 60 \\ \hline 
        Table D & 0.0 & 0.5  & 0.12 & 0.1 & 2 & 10 \\ \hline
        Table E & 0.0 & 0.5 & 0.12 & 0.12 & 2 & 10 \\ \hline
            \hline \hline
         \textbf{Period 3} & r  & $\theta$  & scale$_u$ & scale$_s$ & Iterations & $N$ \\ \hline
        Table A & 0.5333 & 0.8534 & 0.25 & 0.25 & 10 & 60\\ \hline 
        Table A & -0.5333 & 0.8534 & 0.25 & 0.25 & 10& 60\\ \hline 
        Table B & 0.5328 & 0.3348 & 0.27 & 0.27 & 7 & 60\\ \hline 
        Table B & -0.5328 & 0.3348 & 0.27 & 0.27 & 7& 60\\ \hline 
        Table C & 0.4454 & 0.5 & 0.28 & 0.28 & 6 & 60 \\ \hline 
        Table C & -0.4454 & 0.5 & 0.28 & 0.28 & 6 & 60 \\ \hline 
           \hline \hline
         \textbf{Period 4} & r  & $\theta$  & scale$_u$ & scale$_s$ & Iterations & $N$ \\ \hline
        Table C & 0.7078 & 0 & 0.2 & 0.2 & 8 & 60 \\ \hline 
        Table C & -0.7334 & 0.2152 & 0.2 & 0.2 & 8 & 60 \\ \hline 
        Table D & 0.5 & 0 & 0.114 & 0.114 & 8 & 35 \\ \hline 
        Table D & -0.8923 & 0.2216 & 0.114 & 0.114 & 8 & 35 \\ \hline 
        Table E & 0.89 & 0.22 & 0.14 & 0.115 & 3 & 25 \\ \hline 
        Table E & -0.89 & 0.22 & 0.14 & 0.115 & 3 & 25 \\ \hline 
    \end{tabular}  
    \caption{Numerical parameters for the orbits of  periods 2, 3, and 4.     The value $N$ is the order of the polynomial approximation given in \eqref{eq:polyorderN}, Iterations is the number $M$ of iterates in \eqref{eq:globalmanifolds}, and scale$_s$ and scale$_u$ are the scalings of the eigenvectors in the parameterization method defined in Lemma~\ref{def:funstable}.}
\end{table}

\begin{table}[htbp!]
    \centering
    \begin{tabular}{|c|c|c|c|c|c|c|}
        \hline
         \textbf{Period 5} & r  & $\theta$  & scale$_u$ & scale$_s$ & Iterations & $N$ \\ \hline
        Table A & 0.2585 & 0.4308 & 0.29 & 0.29 & 11 & 60 \\ \hline 
        Table A & -0.2585 & 0.4308 & 0.29 & 0.29 & 11 & 60 \\ \hline 
        Table B & 0.2641 & 0.4223 & 0.21 & 0.2 & 10 & 60 \\ \hline 
        Table B & -0.2641 & 0.4223 & 0.21 & 0.2 & 10 & 60 \\ \hline 
        Table C & 0.7695 & 0.5 & 0.17 & 0.17 & 25 & 60 \\ \hline 
        Table C & -0.7695 & 0.5 & 0.17 & 0.17 & 25 & 60 \\ \hline 
        Table D & 0.75 & 0.44 & 0.125 & 0.123 & 10 & 30 \\ \hline 
        Table D & -0.75 & 0.44 & 0.125 & 0.123 & 10 & 30 \\ \hline 
        Table E & 0.7 & 0.46 & 0.125 & 0.123 & 7 & 25 \\ \hline 
        Table E & -0.7 & 0.46 & 0.125 & 0.123 & 7 & 25 \\ \hline 
    \hline \hline
         \textbf{Period 10} & r  & $\theta$  & scale$_u$ & scale$_s$ & Iterations & $N$ \\ \hline
        Table A & -0.3283 & 0.701 & 0.39 & 0.39 & 10 & 60\\ \hline 
        Table B & -0.1746 & 0.3767 & 0.46 & 0.48 & 4 & 60\\ \hline
        Table C & 0.1653 & 0.122 & 0.34 & 0.34 & 8 & 60\\ \hline 
        Table D & 0 & 0.3802 & 0.7 & 0.7 & 20 & 60\\ \hline
        Table E & -0.4106 & 0.1275 & 0.65 & 0.65 & 20 & 60\\ \hline \hline \hline
           \textbf{Period 30} & r  & $\theta$  & scale$_u$ & scale$_s$ & Iterations & $N$\\ \hline
        Table A & 0.1667 & 0.5103 & 0.15 & 0.15 & 20 & 60\\ \hline 
        Table A & -0.1667 & 0.4897 & 0.15 & 0.15 & 20 & 60\\ \hline 
        Table B & 0.4877 & 0.5 & 0.15 & 0.15 & 30 & 60\\ \hline 
        Table B & -0.4877 & 0.5 & 0.15 & 0.15 & 30 & 60\\ \hline       
    \end{tabular}  
    \caption{Numerical parameters for the orbits of periods 5, 10, and 30.}
    
\end{table}

\section{Numerical manipulation of power series}
\label{sec:manipulationPowerSeries}
Consider a power series 
\[
P(\sigma) = \sum_{n = 0}^\infty a_n \sigma^n,
\]
with $a_n \in \mathbb{R}$ (or $\mathbb{C}$), and let $g \colon \mathbb{C} \to \mathbb{C}$
be an analytic  function.  We seek the power series coefficients of 
\[
(g \circ P)(\sigma) = \sum_{n = 0}^\infty b_n \sigma^n.
\]
These are given explicitly by the Fa{\`a} di Bruno formula (multivariate
generalization of the chain rule). This however leads to exponential
complexity, and we consider alternatives.

\subsection{Polynomial case}
A common situation is that $g$ is polynomial, in which case 
explicit recursive formulas for the coefficients
of integer powers of $P$ are worked out using the Cauchy product.
For example we have that 
\begin{equation} \label{eq:cauchyProd}
P(\sigma)^2 = \sum_{n = 0}^\infty \left(
\sum_{k = 0}^n a_{n-k} a_k
\right) \sigma^n,
\end{equation}
and that
\[
P(\sigma)^3 = \sum_{n = 0}^\infty \left(
\sum_{k = 0}^n \sum_{j = 0}^k a_{n-k} a_{k-j} a_j
\right) \sigma^n.
\]
In these cases 
\[
b_n = \sum_{k=0}^n a_{n-k}a_k, 
\quad \quad \mbox{and } \quad \quad 
b_n = \sum_{k=0}^n \sum_{j=0}^k a_{n-k} a_{k-j} a_j,
\]
respectively.
Formulas for higher order powers generalize in an 
obvious way.  Therefore Cauchy products are sufficient for working out 
the power series coefficients for $g \circ P$ in the polynomial case.

\subsection{D-finite functions: differential algebra}
If $g$ is non-polynomial, but elementary 
(sometimes referred to as D-finite), 
then the coefficients 
of $g \circ P$ are often found by combining the fact that 
the chain rule turns composition into multiplication with 
the Cauchy product.
For example, suppose $g(z) = e^z$. We write 
\[
Q(\sigma) = \sum_{n=0}^\infty b_n \sigma^n = g(P(\sigma)) =  e^{P(\sigma)},
\]
to denote the unknown power series of the composition
and have that
\[
Q'(\sigma) = \sum_{n=0}^\infty (n+1) b_{n+1} \sigma^n 
= g'(P(\sigma))P'(\sigma) = e^{P(\sigma)} P'(\sigma),
\]
which is
\[
Q'(\sigma) = Q(\sigma) P'(\sigma).
\]
Expressing this last expression in terms of power series
and calling on the Cauchy product, we have that 
\begin{align*}
 \sum_{n=0}^\infty (n+1) b_{n+1} \sigma^n &= \left(
 \sum_{n=0}^\infty b_n \sigma^n
 \right)\left(
 \sum_{n=0}^\infty (n+1) a_{n+1} \sigma^n
 \right) \\
 & = \sum_{n = 0}^\infty \left(
 (k+1)b_{n-k} a_{k+1}
 \right) \sigma^n.
\end{align*}
Matching like powers of $\sigma^n$, yields the recursive 
expression for the power series coefficients of $g\circ P$
given by 
\[
b_{n+1} = \frac{1}{n+1} \sum_{k=0}^n (k+1) b_{n-k} a_{k+1}, \quad \quad \quad n \geq 0,
\]
with 
\[
b_0 = e^{a_0}.
\]
Using this differential-algebraic framework, we find that 
the power series coefficients of $g \circ P$ are  
no more difficult to compute in this transcendental case
than they are in the quadratic case.


\subsection{Remarks on the literature}
This idea illustrated in the previous 
section is sometimes referred to as automatic differentiation 
for power series, polynomial recast, or polynomial 
embedding,
and it is sufficient for many problems with 
reasonably complicated nonlinearities.  
An important general reference is
Chapter 4.6 of the second volume of Knuth's The Art
of Computer Programming \cite{MR3077153}.
See also \cite{MR2146523} for explicit recursion relations
for many elementary functions.  
Indeed, the reference just cited 
combines automatic differentiation for power series with a
state-of-the-art parser as part of a powerful C++ software 
library called TAYLOR, and suggests that the boundary between 
simple and complex nonlinearities is fluid, when  
approached with the right tools.
We refer also to the papers 
\cite{MR4864708,MR4904407}
where more recent enhancements and applications of 
the TAYLOR package are discussed.  
Further extension of these
ideas to multivariable power series is discussed at length in 
Chapter 2.3 of \cite{HARO_2018}.

This said, there may be (depending on one's tolerance for complexity) 
problems where explicit formulas
for $g$ are either too much to hope for, or are more complicated than one
wishes consider (especially if one is unwilling to adopt
special purpose software). For example, 
Poincar{\'e} sections for ODEs and billiard maps
studied in the present work approach this boundary.  
For such maps, we have ready access only to numerical procedures that
compute the values $g$ (and its derivatives) at points.  In this case,
interpolation is a useful tool for determining
the power series coefficients of $f\circ P$, and when the functions
are analytic, the interpolation can be carried out via the discrete
Fourier transform (DFT).  An excellent reference for this material
is the book of \cite{MR1322049}.

\subsection{Composition of power series via the DFT} \label{sec:DFT}
Suppose that $P \colon \mathbb{R} \to \mathbb{R}^d$,
and $g \colon \mathbb{R}^d \to \mathbb{R}$ are real analytic, 
$g(P(\sigma)) = Q(\sigma)$, whose power series we  
approximate with the polynomial interpolant
\[
Q^N(\sigma) = \sum_{n = 0}^N \bar b_n \sigma^n.
\]
The coefficients $\bar b_n \approx b_n$, $0 \leq n \leq N$ 
of the interpolant are determined by the values of $Q$
at any $(N+1)$ distinct points, and we choose values
equidistributed on the unit circle in $\mathbb{C}$ with the 
first equal to one.

More precisely, define the $N+1$ grid points
\begin{equation}
\sigma_k = e^{2\pi i k/(N+1)}, \quad 0 \leq k \leq N. \nonumber 
\end{equation}
and evaluate the composition on this mesh to obtain
\begin{align}
Q^N(\sigma_k) &= \sum_{n=0}^N \bar b_n \sigma_k^n \approx g(P(\sigma_k)) \nonumber \\
\end{align}
for $0 \leq k \leq N$.
Letting $\omega = e^{\frac{2 \pi i}{N+1}}$, the $b_n$ are approximated 
by $\bar b_n$ solving the linear system 
\begin{equation}
\begin{pmatrix}
1 & 1 & 1 & \cdots & 1 \\
1 & \omega & \omega^2 & \cdots & \omega^N \\
\vdots & \vdots & \vdots & \ddots & \vdots \\
1 & \omega^N & \omega^{2N} & \cdots & \omega^{N^2}
\end{pmatrix}
\begin{pmatrix}
\bar b_0 \\
\bar b_1 \\
\bar b_2 \\
\vdots \\
\bar b_N
\end{pmatrix}
= \begin{pmatrix}
g(P(\sigma_0)) \\
g(P(\sigma_1)) \\
g(P(\sigma_2)) \\
\vdots \\
g(P(\sigma_N))
\end{pmatrix} \nonumber 
\end{equation}

\begin{remark}[Analytic continuation] \label{rem:mainRem}
{\em
In this step, we see that computing 
$f(P(\sigma_k))$ requires evaluating the real analytic 
function $P$ on complex inputs.  This assumes that we are
able to analytically continue $P$ to the unit disk
in $\mathbb{C}$.  Moreover, since the analytic 
continuation of $P$ is complex valued, we have to 
evaluate the real analytic 
function $g$ for complex values as well.
Again, this may require analytic continuation of $g$.}
\end{remark}

Let 
\begin{equation}
V = \begin{pmatrix}
1 & 1 & 1 & \cdots & 1 \\
1 & \omega & \omega^2 & \cdots & \omega^N \\
\vdots & \vdots & \vdots & \ddots & \vdots \\
1 & \omega^N & \omega^{2N} & \cdots & \omega^{N^2}
\end{pmatrix}\nonumber,
\end{equation}
Observe that $V$
is invertible with 
\begin{equation}
V^{-1} = \frac{1}{N+1}\begin{pmatrix}
1 & 1 & 1 & \cdots & 1 \\
1 & \omega^{-1} & \omega^{-2} & \cdots & \omega^{-N} \\
\vdots & \vdots & \vdots & \ddots & \vdots \\
1 & \omega^{-N} & \omega^{-2N} & \cdots & \omega^{-N^2}
\end{pmatrix}.\nonumber
\end{equation}
These are the well known DFT matrices. Since we 
have an explicit inverse, 
solving the linear system determining the 
coefficients of the composition is reduced to
matrix-vector multiplication.
Moreover, when $N$ is large, this matrix-vector multiplication
is accelerated using the Fast Fourier Transform (FFT).

When  $P$ itself is a polynomial (as will be the case in our
Newton procedure) we write  
\[
P(\sigma) = \sum_{n = 0}^N a_n \sigma^n, 
\]
and have that evaluation of $P$ at the interpolation points
$\sigma_k$ is expressed as 
\[
P(\sigma_k) = \sum_{n = 0}^N a_n \sigma_k^n. 
\]
That is
\begin{equation}
\begin{pmatrix}
1 & 1 & 1 & \cdots & 1 \\
1 & \omega & \omega^2 & \cdots & \omega^N \\
\vdots & \vdots & \vdots & \ddots & \vdots \\
1 & \omega^N & \omega^{2N} & \cdots & \omega^{N^2}
\end{pmatrix}
\begin{pmatrix}
a_0 \\
a_1 \\\
a_2 \\
\vdots \\
a_N
\end{pmatrix}
= \begin{pmatrix}
(P(q_0)) \\
(P(q_1)) \\
(P(q_2)) \\
\vdots \\
(P(q_N))
\end{pmatrix} \nonumber 
\end{equation}

Let $p_n = P(q_k)$ for $0 \leq k \leq N$ and write 
\[
\mathbf{p} = \left(
\begin{array}{c}
p_0 \\
\vdots \\
p_N
\end{array}
\right), \quad 
\mathbf{a} = \left(
\begin{array}{c}
a_0 \\
\vdots \\
a_N
\end{array}
\right) \quad \mbox{and} \quad
\mathbf{b} = \left(
\begin{array}{c}
\bar b_0 \\
\vdots \\
\bar b_N
\end{array}
\right).
\]

Given the coefficients $\mathbf{a}$ of a polynomial $P$, 
the procedure for approximating the series coefficients of 
the composition is as follows:
\begin{itemize}
\item Compute $\mathbf{p}$ via 
\[
V \mathbf{a} = \mathbf{p}.
\]
This step is actually known as the DFT.
\item Compute the numbers $f_k = f(p_k)$, $0 \leq k \leq N$.  
Let
\[
\mathbf{f} = \left(
\begin{array}{c}
\bar f_0 \\
\vdots \\
\bar f_N
\end{array}
\right).
\]
\item Compute 
\[
\mathbf{b} = V^{-1} \mathbf{f}.
\]
This step is actually known as the inverse discrete Fourier
transform (IDFT).
\end{itemize}
The components of $\mathbf{b}$ are the approximate Taylor 
coefficients of the composition $f \circ P$.

\subsection{Multiplication matrix for power series}\label{sec:multMat}
In the context of a Newton procedure, 
we often have to solve equations like
\begin{equation} \label{eq:basicLinearEq}
Q(\sigma) P(\sigma) = g(\sigma), 
\end{equation}
where $Q$ and $g$ are given truncated power series: that is, 
polynomials.  (The equation is an oversimplification,
but it will make the necessary point).
While this equation can be solved
by recursion, it is also amenable to numerical linear algebra.

To see this, let 
\[
\mathbf{a} = \left(
\begin{array}{c}
a_0 \\
\vdots \\
a_N
\end{array}
\right), \quad \quad \mbox{and} \quad \quad 
\mathbf{g} = \left(
\begin{array}{c}
g_0 \\
\vdots \\
g_N
\end{array}
\right),
\]
denote the vectors of known polynomial coefficients, 
and 
\[
\mathbf{b} = \left(
\begin{array}{c}
b_0 \\
\vdots \\
b_N
\end{array}
\right),
\]
denote the unknown coefficients of $P$ (the quotient of 
$g$ and $Q$).

Since this equation is linear in $P$, there is a matrix 
representing the action of $Q$ on $P$.  Applying the formula
of Equation \eqref{eq:cauchyProd} to the standard basis vectors
leads to a representation of the action of the Cauchy product 
on the coefficient vector for $P$ given by  
\[
Q \cdot P = 
\left(
\begin{array}{ccccc}
a_0 & 0 & 0 & \ldots & 0 \\
a_1 & a_0 & 0 & \ldots & 0 \\
a_2 & a_1 & a_0 & \ldots & 0 \\
\vdots & \vdots & \vdots & \ddots & \vdots \\
a_N & a_{N-1} & a_{N-2} & \vdots & a_0 
\end{array}
\right)
\left(
\begin{array}{c}
b_0  \\
b_1  \\
b_2 \\
\vdots  \\
b_N  
\end{array}
\right)
\]
Let 
\begin{equation}\label{eq:convMat}
\mbox{CauchyMat}(\mathbf{a}) = 
\left(
\begin{array}{ccccc}
a_0 & 0 & 0 & \ldots & 0 \\
a_1 & a_0 & 0 & \ldots & 0 \\
a_2 & a_1 & a_0 & \ldots & 0 \\
\vdots & \vdots & \vdots & \ddots & \vdots \\
a_N & a_{N-1} & a_{N-2} & \vdots & a_0 
\end{array}
\right).
\end{equation}
We can solve for the unknown coefficients $\mathbf{b}$ of $P(\sigma)$
in Equation \eqref{eq:basicLinearEq} by solving the 
(lower triangular) linear system
\[
\mbox{CauchyMat}(\mathbf{a}) \mathbf{b} = \mathbf{g}.
\]
}

\bibliographystyle{alpha}
\bibliography{billiards}

\end{document}